\documentclass[11pt,a4paper]{amsart}
\usepackage{amssymb,amscd}
\usepackage{mathrsfs}
\usepackage[mathcal]{eucal}
\usepackage{float}
\usepackage[usenames]{color}

\theoremstyle{plain}
\newtheorem{thm}{Theorem}[section]
\newtheorem{lem}[thm]{Lemma}
\newtheorem{pro}[thm]{Proposition}
\newtheorem{cor}[thm]{Corollary}

\newtheorem{thmABC}{Theorem}

\newtheorem{corABC}[thmABC]{Corollary}
\newtheorem*{con*}{Conjecture}

\theoremstyle{remark}
\newtheorem{rem}[thm]{Remark}
\newtheorem{qun}[thm]{Question}
\newtheorem{exm}[thm]{Example}

\newtheorem{dfn}[thm]{Definition}
\newtheorem*{acknowledgements}{Acknowledgements}

\numberwithin{equation}{section}
\numberwithin{table}{section}

\newcommand{\N}{\mathbb{N}}
\newcommand{\Z}{\mathbb{Z}}
\newcommand{\Q}{\mathbb{Q}}
\newcommand{\F}{\mathbb{F}}
\newcommand{\C}{\mathbb{C}}

\newcommand{\R}{\mathbb{R}}
\newcommand{\spl}{\mathfrak{sl}}
\newcommand{\gl}{\mathfrak{gl}}
\newcommand{\su}{\mathfrak{su}}
\newcommand{\gu}{\mathfrak{gu}}
\newcommand{\mfg}{\mathfrak{g}}
\newcommand{\mfn}{\mathfrak{n}}
\newcommand{\mfh}{\mathfrak{h}}
\newcommand{\mfp}{\mathfrak{p}}
\newcommand{\mfP}{\mathfrak{P}}
\newcommand{\diff}{\mathfrak{D}}

\newcommand{\lri}{\mathfrak{o}}
\newcommand{\Lri}{\mathfrak{O}}
\newcommand{\lfi}{\mathfrak{k}}
\newcommand{\Lfi}{\mathfrak{K}}

\newcommand{\Ldr}{\mathfrak{D}}
\newcommand{\ol}{\overline}
\newcommand{\gri}{\ensuremath{{\scriptstyle \mathcal{O}}}}
\newcommand{\smallgri}{\ensuremath{{\scriptscriptstyle \mathcal{O}}}}
\newcommand{\Gri}{\ensuremath{\mathcal{O}}}
\newcommand{\gfi}{\ensuremath{k}}
\newcommand{\Gfi}{\ensuremath{K}}

\renewcommand{\epsilon}{\varepsilon}
\renewcommand{\phi}{\varphi}
\renewcommand{\theta}{\vartheta}

\newcommand{\Qp}{\mathbb{Q}_p}

\newcommand{\bigdotcup}{\ensuremath{\mathop{\dot{\bigcup}}}}
\DeclareMathOperator{\ad}{ad}
\DeclareMathOperator{\pr}{seg}
\DeclareMathOperator{\Irr}{Irr}
\DeclareMathOperator{\Rad}{Rad}
\DeclareMathOperator{\lev}{lev}
\DeclareMathOperator{\rk}{rk}
\DeclareMathOperator{\SL}{SL}
\DeclareMathOperator{\GL}{GL}

\DeclareMathOperator{\Hom}{Hom}
\DeclareMathOperator{\Mat}{Mat}
\DeclareMathOperator{\ind}{ind}
\DeclareMathOperator{\res}{res}
\DeclareMathOperator{\Cen}{C}
\DeclareMathOperator{\SU}{SU}
\DeclareMathOperator{\GU}{GU}
\DeclareMathOperator{\Tr}{Tr}

\DeclareMathOperator{\real}{Re}

\DeclareMathOperator{\Id}{Id}
\DeclareMathOperator{\Cay}{Cay}
\DeclareMathOperator{\cay}{cay}
\DeclareMathOperator{\rank}{r}

\def \bfG {{\bf G}}

\def \bfa {{\bf a}}
\def \bfb {{\bf b}}

\def \bfm {{\bf m}}
\def \bfn {{\bf n}}
\def \bfs {{\bf s}}
\def \bfx {{\bf x}}
\def \bfy {{\bf y}}
\def \bfY {{\bf Y}}
\def \wt {\widetilde}
\def \mbbA {\ensuremath{\mathbb{A}}}
\def \mcC {\ensuremath{\mathcal{C}}}
\def \mcD {\ensuremath{\mathcal{D}}}
\def \mcL {\ensuremath{\mathcal{L}}}
\def \mcV {\ensuremath{\mathcal{V}}}
\def \mcDtilde {\ensuremath{\wt{\mathcal{D}}}}
\def \Fp {\ensuremath{\mathbb{F}_p}}

\def \mcR {\ensuremath{\mathcal{R}}}

\def \mcZ {\ensuremath{\mathcal{Z}}}
\def \mfgu {\ensuremath{\mathfrak{gu}}}

\def \mfC {\ensuremath{\mathfrak{C}}}
\def \mfI {\ensuremath{\mathfrak{I}}}

\def \Fq {\ensuremath{\mathbb{F}_q}}
\def \Zp  {\mathbb{Z}_p}

\begin{document}
\title[Representation zeta functions of groups]{Representation zeta
  functions of compact $p$-adic analytic groups and arithmetic groups}

\author{Nir Avni}\thanks{Avni was supported by NSF grant DMS-0901638.}
  \address{Department of Mathematics, Harvard University, One Oxford
  Street, Cambridge MA 02138, USA} \email{avni.nir@gmail.com}

\author{Benjamin Klopsch} \address{Department of Mathematics, Royal
  Holloway, University of London, Egham TW20 0EX, United Kingdom}
\email{Benjamin.Klopsch@rhul.ac.uk}

\author{Uri Onn} \address{Department of Mathematics, Ben Gurion
  University of the Negev, Beer-Sheva 84105 Israel}
  \email{urionn@math.bgu.ac.il}

\author{Christopher Voll} \address{School of Mathematics, University
  of Southampton, University Road, Southampton SO17 1BJ, United
  Kingdom}
\email{C.Voll.98@cantab.net}

\begin{abstract}
  We introduce new methods from $\mfp$-adic integration into the study
  of representation zeta functions associated to compact $p$-adic
  analytic groups and arithmetic groups.  They allow us to establish
  that the representation zeta functions of generic members of
  families of $p$-adic analytic pro-$p$ groups obtained from a global,
  `perfect' Lie lattice satisfy functional equations. In the case of
  `semisimple' compact $p$-adic analytic groups, we exhibit a link
  between the relevant $\mfp$-adic integrals and a natural filtration
  of the locus of irregular elements in the associated semisimple Lie
  algebra, defined by centraliser dimension.

  Based on this algebro-geometric description, we compute explicit
  formulae for the representation zeta functions of principal
  congruence subgroups of the groups $\SL_3(\lri)$, where $\lri$ is a
  compact discrete valuation ring of characteristic $0$, and of the
  groups $\SU_3(\Lri, \lri)$, where $\Lri$ is an unramified quadratic
  extension of~$\lri$.  These formulae, combined with approximative
  Clifford theory, allow us to determine the abscissae of convergence
  of representation zeta functions associated to arithmetic subgroups
  of algebraic groups of type~$A_2$.  Assuming a conjecture of Serre
  on the Congruence Subgroup Problem, we thereby prove a conjecture of
  Larsen and Lubotzky on lattices in higher-rank semisimple groups for
  algebraic groups of type $A_2$ defined over number fields.
\end{abstract}




\maketitle

\setcounter{tocdepth}{1}
\tableofcontents
\thispagestyle{empty}

\section{Introduction}\label{sec:introduction}

\subsection{Background and motivation}\label{subsec:background_motivation}
Let $G$ be a group and denote, for $n\in\N$, by $r_n(G)$ the number of
isomorphism classes of $n$-dimensional irreducible complex
representations of $G$; if $G$ is a topological or an algebraic group,
it is tacitly understood that representations are continuous or
rational, respectively.  We assume henceforth that $G$ is
(representation) \emph{rigid}, i.e.\ that $r_n(G)$ is finite for
all~$n \in \N$.

In finite group theory there is a fruitful tradition in studying
character degrees and conjugacy classes; e.g.\ see \cite{Hu98,Mo03}
and references therein.  Interesting new phenomena occur when one
takes an asymptotic point of view, as illustrated in \cite{LiSh05}
which focuses on character degrees of finite groups $H$ of Lie type as
$\lvert H \rvert$ tends to infinity.  Likewise, in the subject of
representation growth one studies, for an infinite group $G$, the
arithmetic properties of the sequence $r_n(G)$, $n \in \N$, and its
asymptotic behaviour as $n$ tends to infinity.  This line of
investigation draws its inspiration also from the area of subgroup
growth which, in a similar vein, is concerned with the distribution of
finite index subgroups in $G$; e.g.\ see \cite{LuSe03,dSGr06}.

The group $G$ is said to have \emph{polynomial representation growth}
(PRG) if the sequence $R_N(G):=\sum_{n=1}^N r_n(G)$, $N \in \N$, is
bounded by a polynomial in~$N$.  An important tool to study the
representation growth of a PRG group $G$ is its \emph{representation
zeta function}, viz.\ the Dirichlet series
$$
\zeta_{G}(s):=\sum_{n=1}^\infty r_n(G)n^{-s},
$$
where $s$ is a complex variable.  It is well-known that the
\emph{abscissa of convergence} $\alpha(G)$ of the series $\zeta_G(s)$,
i.e.\ the infimum of all $\alpha\in\R$ such that $\zeta_G(s)$
converges on the complex right half-plane $\{s \in \C \mid \real (s)>
\alpha \}$, gives the precise degree of polynomial growth: $\alpha(G)$
is the smallest value such that $R_N(G)=O(1 + N^{\alpha(G) +
  \epsilon})$ for every $\epsilon \in \R_{>0}$.

Key advances in describing the representation growth of arithmetic
groups were made by Larsen and Lubotzky in~\cite{LaLu08}.  At present
our general understanding of the representation theory of `semisimple'
arithmetic groups, like $\SL_n(\Z)$, rests on three major theories,
namely Margulis super-rigidity, the representation theory of simple
algebraic groups and Deligne-Lusztig theory for finite groups of Lie
type.  What is missing is a picture of the representations of groups
of Lie type over finite local rings; however, the corresponding
classification problem is considered to be `wild'.  Focusing on
degrees of characters and taking an enumerative point of view, we make
in this paper significant progress in describing the representation
growth of arithmetic groups and compact $p$-adic analytic groups which
arise as their completions.  Connections between our work and number
theory manifest themselves in classical papers of
Kloosterman~\cite{Kl46} as well as in recent developments related to
the Langlands programme.  For instance, certain representations of
groups of Lie type over compact discrete valuation rings play a
central role in the construction of super cuspidal representations of
reductive groups over local fields within the framework of the local
Langlands correspondence; cf.~\cite{BuKu93}.

More concretely, we study the representation zeta functions of two
types of groups: compact $p$-adic analytic groups and arithmetic
groups. Before we describe our main techniques and results in
Section~\ref{subsec:results}, we highlight some features of the
representation growth of these two types, and explain how they are
interwoven in the theory of representation zeta functions.

A compact $p$-adic analytic group $G$ is rigid if and only if it is
FAb, i.e.\ if every open subgroup of $G$ has finite abelianisation;
cf.~\cite[Proposition~2]{BaLuMaMo02}. In~\cite{Ja06}, Jaikin-Zapirain
proved rationality results for the representation zeta functions of
FAb compact $p$-adic analytic groups using tools from model theory. In
particular, the representation zeta function of a FAb compact $p$-adic
analytic pro-$p$ group is rational in $p^{-s}$.  (In the special case
$p=2$, the assertion is currently known to be true for uniformly
powerful groups, but conjectured to hold generally, as for odd $p$.)
Key examples of FAb compact $p$\nobreakdash-adic analytic groups are
the special linear groups $\SL_n(\lri)$ and their principal congruence
subgroups~$\SL_n^m(\lri)$, where $\lri$ is a compact discrete
valuation ring of characteristic $0$ and residue field
characteristic~$p$.

The arithmetic groups we are interested in are arithmetic subgroups of
semisimple algebraic groups defined over number fields. More
precisely, we consider groups $\Gamma$ which are commensurable to
$\mathbf{G}(\gri_S)$, where $\bfG$ is a connected, simply connected
semisimple algebraic group, defined over a number field $\gfi$ and
$\gri_S$ is the ring of $S$-integers in~$\gfi$, for a finite set $S$
of places of $\gfi$ including all the archimedean ones.  Let $\Gamma$
be of this form.  Lubotzky and Martin showed that $\Gamma$ has PRG if
and only if $\Gamma$ has the Congruence Subgroup Property (CSP);
cf.~\cite{LuMa04}.  Suppose that $\Gamma$ possesses these properties.
Then according to a result of Larsen and Lubotzky
(\cite[Proposition~1.3]{LaLu08}) the representation zeta function of
$\Gamma$ admits an Euler product decomposition. Indeed, if $\Gamma =
\mathbf{G}(\gri_S)$ and if the congruence kernel of $\Gamma$ is
trivial, this decomposition takes the form
\begin{equation}\label{equ:euler}
  \zeta_\Gamma(s)=\zeta_{\bfG(\C)}(s)^{\lvert k:\Q \rvert} \cdot
    \prod_{v\not\in S}\zeta_{\bfG(\smallgri_v)}(s),
\end{equation}
where the product extends over all places $v$ of $\gfi$ which are not
in $S$.  Here each archimedean factor $\zeta_{\bfG(\C)}(s)$ enumerates
the finite dimensional, irreducible rational representations of the
algebraic group $\bfG(\C)$; their contribution to the Euler product
reflects Margulis super-rigidity.  By $\gri_v$ we denote the ring of
integers in the completion $\gfi_v$ of $\gfi$ at the non-archimedean
place~$v$. The Euler product over the factors
$\zeta_{\bfG(\smallgri_v)}(s)$, $v\not\in S$, captures the
representations of $\Gamma$ with finite image. The groups
$\bfG(\gri_v)$ are FAb compact $p$-adic analytic groups.  An important
family of examples of arithmetic groups with the CSP are the special
linear groups $\SL_n(\Z)$, $n\geq3$, for which~\eqref{equ:euler} reads
$$\zeta_{\SL_n(\Z)}(s)=\zeta_{\SL_n(\C)}(s) \cdot
\prod_{p\text{ prime}}\zeta_{\SL_n(\Zp)}(s).$$

Several of the results of \cite{LaLu08} concern the abscissae of
convergence of the local representation zeta functions occurring as
Euler factors on the right hand side of~\eqref{equ:euler} for suitable
arithmetic groups $\Gamma$.  Of particular interest is the dependence
of these abscissae on natural invariants, such as the Lie rank of the
ambient group of $\Gamma$ or the place $v$ at which $\Gamma$ is
localised.  With regards to abscissae of convergence of the global
representation zeta functions for arithmetic groups, Avni proved
in~\cite{Av08} that, for an arithmetic group $\Gamma$ with the CSP,
the abscissa of convergence of $\zeta_\Gamma(s)$ is always a rational
number. Larsen and Lubotzky made the following conjecture, which can
be regarded as a refinement of Serre's conjecture on the Congruence
Subgroup Problem.

\begin{con*}[\mbox{Larsen and Lubotzky~\cite[Conjecture~1.5]{LaLu08}}]
Let $H$ be a higher-rank semi\-sim\-ple group. Then, for any two
irreducible lattices $\Gamma_1$ and $\Gamma_2$ in $H$,
$\alpha(\Gamma_1)=\alpha(\Gamma_2)$.
\end{con*}

In the current paper we introduce novel methods into the study of
representation zeta functions of compact $p$-adic analytic groups and
arithmetic groups, and use them to prove both local and global
results.  The main theorems of the current paper and related results
from~\cite{AvKlOnVo_pre_a,AvKlOnVo_pre_b} were announced
in~\cite{AvKlOnVo10}.

\subsection{Main results}\label{subsec:results} Our first two main
results concern families of $p$-adic analytic pro\nobreakdash-$p$
groups which arise, via $p$-adic Lie theory, from the completions of a
global Lie lattice~$\Lambda$, defined over the ring of integers of a
number field.  Examples of specific interest are families formed by
the principal congruence subgroups of the compact $p$-adic analytic
groups $\bfG(\gri_v)$ featuring in \eqref{equ:euler}.  More generally,
let $\gri$ be the ring of integers of a number field~$\gfi$, and
$\Lambda$ an $\gri$-Lie lattice such that $k \otimes_\smallgri
\Lambda$ is a finite dimensional, perfect $\gfi$-Lie algebra.  Let
$\lri = \gri_v$ be the ring of integers of the completion $\gfi_v$ of
$\gfi$ at a non-archimedean place $v$ above the rational prime
$p$. Given a finite extension $\Lri\vert\lri$ of compact discrete
valuation rings of characteristic~$0$, let $\mfP$ denote the maximal
ideal of $\Lri$ and write $f(\Lri,\lri)$ for the degree of inertia.
Consider the $\Lri$-Lie lattice $\mathfrak{g}(\Lri) := \Lri
\otimes_\smallgri \Lambda$.  For all sufficiently large $m\in\N_0$,
the principal congruence sublattice $\mathfrak{g}^m(\Lri) := \mfP^m
\mathfrak{g}(\Lri)$ is potent and saturable; consequently it
corresponds, via $p$-adic Lie theory, to a potent and saturable FAb
pro-$p$ group $\mathsf{G}^m(\Lri) := \exp(\mathfrak{g}^m(\Lri))$; see
Section~\ref{subsec:pro-p}.  We call such $m$ \emph{permissible
for~$\mathfrak{g}(\Lri)$}.  Proposition~\ref{pro:Zlatticeprep} states
that, uniformly for any choice of the lattice $\Lambda$, all natural
numbers $m$ exceeding an explicit linear bound in terms of the
absolute ramification index~$e(\Lri,\Zp)$ are permissible for
$\mathfrak{g}(\Lri)$.  In particular, if $\Lri$ is unramified over
$\Zp$ and $p$ odd, then every $m\geq1$ is permissible for
$\mfg(\Lri)$.

Our first theorem provides a universal formula for the zeta functions
of the groups $\mathsf{G}^m(\Lri)$ and establishes the existence of
local functional equations.

\begin{thmABC} \label{thmABC:funeq} In the setup described above,
  there exist a finite set $S$ of places of $\gfi$, $r \in \N$ and a
  rational function $R(X_1,\dots,X_r,Y)\in\Q(X_1,\dots,X_r,Y)$ such
  that, for every non-archimedean place $v$ of $\gfi$ with $v \not\in
  S$, the following is true.

  There exist algebraic integers $\lambda_1 =
  \lambda_1(v),\dots,\lambda_r=\lambda_r(v)$, such that, for all
  finite extensions $\Lri$ of $\lri = \gri_v$ and for all $m\in\N_0$
  which are permissible for $\mathfrak{g}(\Lri)$ one has
  $$
  \zeta_{\mathsf{G}^m(\Lri)}(s) = q_v^{f dm}R(\lambda_1^f, \dots,
  \lambda_r^f,q_v^{-fs}),
  $$ where $q_v$ denotes the residue field cardinality of $\lri$, $f =
  f(\Lri,\lri)$ and $d = \rk_\Lri(\mathfrak{g}(\Lri)) = \dim_\gfi(\gfi
  \otimes_\smallgri \Lambda)$. Furthermore, the functional equation
  \begin{equation}\label{equ:funeq}
    \zeta_{\mathsf{G}^m(\Lri)}(s)|_{\substack{q_v\rightarrow q_v^{-1}
        \\ \lambda_j\rightarrow
        \lambda_j^{-1}}}=q_v^{fd(1-2m)}\zeta_{\mathsf{G}^m(\Lri)}(s)
  \end{equation}
holds.
\end{thmABC}

Our second theorem concerns the possible poles of the zeta functions
$\zeta_{\mathsf{G}^m(\Lri)}(s)$.  It also provides information about
how the abscissae of convergence of these zeta functions may change
with extensions $\Lri$ of the base ring $\lri$.

\begin{thmABC} \label{thmABC:poles} In the setup described above,
  there exists a finite set $P \subset \Q_{>0}$ such that the
  following is true.
  \begin{enumerate}
  \item \label{equ:abscissae_superset} For all non-archimedean places
    $v$ of $\gfi$, for all finite extensions $\Lri$ of $\lri = \gri_v$
    and all $m \in \N_0$ which are permissible for
    $\mathfrak{g}(\Lri)$ one has
    \begin{equation*}
      \{ \real(z) \mid z \in \C \text{ a pole of
      }\zeta_{\mathsf{G}^m(\Lri)}(s)\}\subseteq P.
    \end{equation*}
    In particular, one has $\alpha(\mathsf{G}^m(\Lri)) \leq \max P$.
  \item \label{equ:abscissa_equality} There exists a set $V$ of
    non-archimedean places of $\gfi$ of positive Dirichlet density
    with the property that for all $v\in V$, all finite extensions
    $\Lri$ of $\lri=\gri_v$ and all $m\in\N_0$ which are permissible
    for $\mfg(\Lri)$ one has
    \begin{equation*}
      \alpha(\mathsf{G}^m(\Lri)) = \max P.
    \end{equation*}
  \end{enumerate}

  Furthermore, if $v$ is any non-archimedean place of $\gfi$ and if
  $\gri_v = \lri \subseteq \Lri_1 \subseteq \Lri_2$ is a tower of
  finite ring extensions, then for every $m \in \N_0$ which is
  permissible for $\mathfrak{g}(\Lri_1)$ and $\mathfrak{g}(\Lri_2)$
  one has
  \begin{equation}\label{equ:comparison_abscissae}
    \alpha(\mathsf{G}^m(\Lri_1))
    \leq \alpha(\mathsf{G}^m(\Lri_2)).
  \end{equation}
\end{thmABC}

\begin{rem}
  The proofs of Theorems~\ref{thmABC:funeq} and \ref{thmABC:poles}
  involve a description of the respective zeta functions in terms of
  $\mfp$-adic integrals generalising Igusa local zeta functions. The
  study of the latter relies on deep algebro-geometric techniques and
  results; cf.~\cite{De91}.  These include resolutions of
  singularities in characteristic $0$ and aspects of the Weil
  conjectures.

  The rational function $R(X_1,\dots,X_r,Y)$ in
  Theorem~\ref{thmABC:funeq} arises from explicit formulae which admit
  further investigation; it is remarkable that there is no dependency
  on ramification.  The main thrust of Theorem~\ref{thmABC:poles} is
  the following.  Whilst the formulae established in
  Theorem~\ref{thmABC:funeq} genuinely require finitely many
  exceptions, we show that the finite set of (real parts of)
  `candidate poles' arising from the analysis of the generic places,
  viz.\ those with `good reduction' modulo $\mfp$, yields also a
  superset of the (real parts of) poles at the finitely many
  exceptional places which have `bad reduction' modulo $\mfp$.

  Theorem~\ref{thmABC:poles} implies statements about the abscissae of
  convergence of the zeta functions of the $p$-adic analytic groups
  $\bfG(\gri_v)$ which occur as non-archimedean factors in Euler
  products like~\eqref{equ:euler}. Indeed, for sufficiently large $m$,
  Theorem~\ref{thmABC:poles} applies to the principal congruence
  subgroups $\bfG^m(\gri_v)$, and the abscissae of convergence of the
  zeta functions of commensurable groups coincide; cf.\
  \cite[Corollary~4.5]{LaLu08}.

  Theorem~\ref{thmABC:poles} also has immediate applications to Larsen
  and Lubotzky's conjecture given at the end of
  Section~\ref{subsec:background_motivation}: it shows that if almost
  all of the non-archimedean Euler factors of the zeta functions of
  two irreducible lattices $\Gamma_1$ and $\Gamma_2$ coincide, then
  $\alpha(\Gamma_1)=\alpha(\Gamma_2)$. This implies, for instance, to
  inner forms of groups of type $A_n$ defined over a number field.  In
  the current paper we apply this to groups of type~$A_2$.
\end{rem}

In our third main theorem we determine the abscissae of convergence of
the representation zeta functions of arithmetic groups of type $A_2$;
cf.~Appendix~\ref{subsec:algebraic/arithmetic_groups} for a precise
definition of this class of groups.  In the presence of the CSP, the
representation zeta function of such an arithmetic group is an Euler
product, whose non-archimedean factors are representation zeta
functions of FAb compact $p$-adic analytic groups;
see~\eqref{equ:euler}.  Using our $\mfp$-adic machinery and
approximative Clifford theory, we prove the following global result.

\begin{thmABC}
  \label{thmABC:lalu}
  Let $\Gamma$ be an arithmetic subgroup of a connected, simply
  connected simple algebraic group of type $A_2$ defined over a number
  field.  If $\Gamma$ has the CSP, then $\alpha(\Gamma) = 1$.
\end{thmABC}

\begin{rem}
  Naturally, Theorem~\ref{thmABC:lalu} calls for a closer
  investigation of the singularity at $s=1$.  In
  \cite{AvKlOnVo_pre_a}, we use a different approach to show that the
  zeta function of the special linear group $\SL_3(\gri)$ over the
  ring of integers $\gri$ of a number field admits meromorphic
  continuation beyond $\real(s) = 1$ and that in this situation $s=1$
  is a double pole.  In fact, the assertion holds more generally for
  arithmetic groups of type ${}^1 \! A_2$ with the Congruence Subgroup
  Property.
\end{rem}

A well-known conjecture of Serre asserts that an arithmetic subgroup
of a connected, simply connected simple algebraic group has the CSP if
and only if the latter is of `higher rank'; cf.\ \cite{Se70} and
\cite[Section~9.5]{PlRa94}.  Assuming this conjecture for groups of
type~$A_2$, Theorem~\ref{thmABC:lalu} implies that Larsen and
Lubotzky's conjecture, as stated at the end of
Section~\ref{subsec:background_motivation}, holds for higher-rank
semisimple groups which are products of groups of type $A_2$.  More
precisely, one has

\begin{corABC}\label{corABC:lalu}
  Assuming Serre's conjecture, Larsen and Lubotzky's conjecture holds
  for groups $H = \prod_{i=1}^r \mathbf{G}_i(K_i)$, where each $K_i$
  is a local field of characteristic $0$ and each $\mathbf{G}_i$ is an
  absolutely almost simple $K_i$-group of type $A_2$ such that
  $\sum_{i=1}^r \rk_{K_i}(\mathbf{G}_i) \geq 2$ and none of the
  $\mathbf{G}_i(K_i)$ is compact.
\end{corABC}

Indeed, by Margulis' arithmeticity theorem any irreducible lattice
$\Gamma$ in a semisimple group $H$ as in Corollary~\ref{corABC:lalu}
is arithmetic, and Theorem~\ref{thmABC:lalu} applies to such~$\Gamma$;
see~\cite[Chapter~7.2]{LuSe03} and \cite[Section~2]{LaLu08} for
details.

Key to our proof of Theorem~\ref{thmABC:lalu} is the following local
result, which we formulate in accordance with the notation introduced
just before Theorem \ref{thmABC:funeq}. Recall in particular the
definition of permissible $m$; cf.~Definition~\ref{def:permissible}.

\begin{thmABC} \label{thmABC:SL3} Let $\lri$ be a compact discrete
  valuation ring of characteristic~$0$ whose residue field has
  cardinality~$q$ and characteristic not equal to $3$.  Let
  $\mathfrak{g}(\lri)$ be one of the following two $\lri$-Lie lattices
  of type $A_2$:
  \begin{enumerate}
  \item [(a)] $\mathfrak{sl}_3(\lri) = \{ \mathbf{x} \in
    \mathfrak{gl}_3(\lri) \mid \Tr(\mathbf{x}) = 0 \}$,
  \item [(b)] $\mathfrak{su}_3(\Lri,\lri) = \{ \mathbf{x} \in
    \mathfrak{sl}_3(\Lri) \mid \mathbf{x}^\sigma = -
    \mathbf{x}^{\mathrm{t}} \}$, where $\Lri \vert \lri$ is an
    unramified quadratic extension with non-trivial automorphism
    $\sigma$ and $\mathbf{x}^{\mathrm{t}}$ denotes the transpose of
    $\mathbf{x}$.
  \end{enumerate}

  For $m \in \N$, let $\mathsf{G}^m(\lri)$ be the $m$-th principal
  congruence subgroup of the corresponding group $\SL_3(\lri)$ or
  $\SU_3(\Lri,\lri)$.  Then, for all $m$ which are permissible for
  $\mathfrak{g}(\lri)$, one has
  \begin{equation*}
    \zeta_{\mathsf{G}^m(\lri)}(s) = q^{8m} \frac{1 + u(q) q^{-3-2s}
      + u(q^{-1}) q^{-2-3s} + q^{-5-5s}}{(1 - q^{1-2s})(1 - q^{2-3s})},
  \end{equation*}
  where
  \begin{equation*} u(X) =
    \begin{cases}
      \phantom{-}X^3 + X^2 - X - 1 - X^{-1} & \text{ if }
      \mathfrak{g}(\lri)=
      \mathfrak{sl}_3(\lri),\\
      -X^3 + X^2 - X + 1 - X^{-1} & \text{ if }\mathfrak{g}(\lri)=
      \mathfrak{su}_3(\Lri,\lri).
    \end{cases}
  \end{equation*}
\end{thmABC}

\begin{rem}
  The Euler product~\eqref{equ:euler} links the local analysis in
  Theorem~\ref{thmABC:SL3} to the global results in
  Theorem~\ref{thmABC:lalu} and its Corollary~\ref{corABC:lalu} for
  arithmetic groups of type~$A_2$.  In Theorem~\ref{thmABC:SL3}, the
  close resemblance between the representation zeta functions of the
  special linear and the special unitary groups is noteworthy and
  reminiscent of the Ennola duality for the characters of the finite
  groups $\GL_n(\F_q)$ and $\GU_n(\F_{q^2},\F_q)$; cf.~\cite{Ka85}.
  In \cite{AvKlOnVo_pre_b}, we compute explicitly the representation
  zeta functions of principal congruence subgroups of $\SL_3(\lri)$,
  where $\lri$ is an unramified extension of $\Z_3$.  The resulting
  formula is uniform in the residue field size, but genuinely
  different from the formulae in Theorem~\ref{thmABC:SL3}.
\end{rem}

\subsection{Discussion of results and techniques} The core of our
technique is a $\mfp$-adic formalism for the representation zeta
functions of potent, saturable pro-$p$ groups. This approach has two
key ingredients.  Firstly, the \emph{Kirillov orbit method} for
potent, saturable pro-$p$ groups provides a way to construct the
characters of irreducible complex representations in terms of
co-adjoint orbits; see, e.g.,~\cite{Go07, Go08}. This `linearisation'
-- pioneered in~\cite{Ho77b, Ja06} -- allows us to transform the
original problem of enumerating representations by their dimension
into the task of counting co-adjoint orbits by their size.  We remark
that the Kirillov orbit method can also be applied in the context of
finitely generated nilpotent groups; see~\cite{Ho77a}.  Voll used this
method to study zeta functions which encode iso-twist classes of
irreducible characters of finitely generated nilpotent groups;
see~\cite{Vo10}. These zeta functions were first studied by Hrushovski
and Martin, who used sophisticated techniques from model theory and
did not rely on the Kirillov orbit method; see~\cite{HrMa07}.

The second main idea of our approach is to enumerate co-adjoint orbits
by means of suitable \emph{$\mfp$-adic integrals} which are closely
related to Igusa local zeta functions; cf.~\cite{De91, Ig00, VeZu08}.
A general class of such $\mfp$-adic zeta functions was introduced
in~\cite{Vo10}; it does not rely on notions from model theory such as
the concept of definable integrals utilised in~\cite{Ja06}, for which
explicit formulae are not available.  The key ingredient of our
$\mfp$-adic formalism is a description of the representation zeta
functions of groups like $\mathsf{G}^m(\Lri)$, which arise from a
global $\gri$-Lie lattice $\Lambda$ as in the setup of
Theorems~\ref{thmABC:funeq} and \ref{thmABC:poles}, in terms of
$\mfp$-adic integrals of the shape
\begin{equation} \label{equ:integral_intro}
 \mathcal{Z}_{\Lri}(r,t) = \int_{(x,\mathbf{y}) \in V(\Lri)} \lvert x
  \rvert_\mfP^t \prod_{j=1}^{\lfloor d/2 \rfloor}
  \frac{\lVert F_j(\mathbf{y}) \cup F_{j-1}(\mathbf{y})x^2
  \rVert_\mfP^r}{\lVert F_{j-1}(\mathbf{y})
  \rVert_\mfP^r} \, d\mu(x,\mathbf{y}),
\end{equation}
where $V(\Lri)\subset\Lri^{d+1}$ is a union of cosets modulo $\mfP$,
$\mu$ is the additive Haar measure on $\Lri^{d+1}$, normalised so that
$\mu(\Lri^{d+1})=1$, the $F_j(\mathbf{Y})$ are finite sets of
polynomials over~$\gri$, which may be defined in terms of the
structure constants of the $\gri$-Lie lattice~$\Lambda$ with respect
to a given $\gri$-basis, and we write $\|\cdot \|_\mfP$ for the
$\mfP$-adic maximum norm; cf.~\eqref{equ:integral_neu}.

The proof of Theorem~\ref{thmABC:funeq} relies heavily on deep
algebro-geometric techniques and results which are commonly used to
analyse integrals like~\eqref{equ:integral_intro}.  These include
principalisations of ideals for the algebraic varieties defined by the
polynomials in $F_j(\mathbf{Y})$, and aspects of the Weil conjectures
regarding the zeta functions associated to smooth projective algebraic
varieties over finite fields; see~\cite{De91, VeZu08}.
Theorem~\ref{thmABC:funeq} captures a threefold `uniformity' of the
representation zeta functions of the groups in question: regarding the
variation of the prime, ring extension and congruence level.  It is
noteworthy that, in the setting of subring and subgroup growth, the
impact of variation of all three parameters is much less understood
than in the present context.

Of particular interest are Theorems~\ref{thmABC:funeq} and
\ref{thmABC:poles} when they are applied to zeta functions associated
to principal congruence subgroups of `se\-mi\-sim\-ple' compact
$p$-adic analytic groups $\bfG(\gri_v)$, featuring on the right hand
side of the Euler product~\eqref{equ:euler}.  In this context $\bfG$
is a connected, simply connected semisimple algebraic group defined
over a number field $\gfi$, with ring of integers $\gri$.  The Euler
product links the local analysis in Theorems~\ref{thmABC:funeq},
\ref{thmABC:poles} and~\ref{thmABC:SL3} to the global result in
Theorem~\ref{thmABC:lalu} and its Corollary~\ref{corABC:lalu}.
Moreover, in the `semisimple' case, we offer a Lie-theoretic
interpretation of the varieties defined by the sets of polynomials
$F_j(\mathbf{Y})$ occurring in~\eqref{equ:integral_intro}, by
identifying them with certain algebraic subvarieties of the Lie
algebra associated to the group $\mathbf{G}$.  These varieties are
defined in terms of centraliser dimension, and yield a filtration of
the locus of irregular elements in the Lie algebra.  This makes them
amenable to tools from algebraic Lie theory, for instance the theory
of sheets; cf.~\cite{Bo81}.

Whilst Theorem~\ref{thmABC:funeq} is a local result in as much as it
concerns the representation zeta functions of pro-$p$ groups of the
form $\mathsf{G}^m(\Lri)$, we make use of the result's global
framework to exclude finitely many places which require a refined
analysis.  It is an interesting question whether the phenomenon of
(local) functional equations transcends the realm of pro-$p$ groups.
\begin{qun}
  Does a functional equation, akin to \eqref{equ:funeq}, hold for the
  (generic) non-archimedean factors $\zeta_{\bfG(\smallgri_v)}(s)$
  occurring in the Euler product~\eqref{equ:euler}, or for the
  representation zeta functions of arithmetic groups?
\end{qun}
Even in the cases where we have explicit formulae available (cf.\
\cite[Theorem~7.5]{Ja06} and also \cite{AvKlOnVo_pre_b} for
$\SL_2(\lri)$; cf.\ \cite{AvKlOnVo_pre_a} for $\SL_3(\lri)$), it is
not clear which operation should play the role of the `inversion of
the prime' in the functional equations~\eqref{equ:funeq} of
Theorem~\ref{thmABC:funeq}.

The proof of Theorem~\ref{thmABC:poles} relies on a refinement of the
formulae which we utilise to deduce the functional
equations~\eqref{equ:funeq} for `generic' places.  Indeed, the finite
set $P$ of (real parts of) `candidate poles' is obtained from the
numerical data of a principalisation of ideals associated to the
$\gri$-lattice $\Lambda$.  In the case of good reduction modulo
$\mfp$, the inclusion in \eqref{equ:abscissae_superset} follows from
an explicit formula for the $\mfp$-adic integrals describing
$\zeta_{\mathsf{G}^m(\Lri)}(s)$, which involves these numerical data.
In the proof of Theorem~\ref{thmABC:poles} we establish that $P$ is
also a superset of the real parts of the poles in the case of bad
reduction.  Whilst it is in general a difficult problem to decide
which of the `candidate real parts' in $P$ are actually real parts of
poles of a specific representation zeta function, the equality in
\eqref{equ:abscissa_equality} reflects that an irreducible
quasi-projective variety defined over $\gri$ has points modulo
$\mfp_v$ for a set of places~$v$ of positive Dirichlet density.  With
a view toward the poles of the Euler factors in~\eqref{equ:euler} we
pose the following question.

\begin{qun}
  Let $G$ be a FAb compact $p$-adic analytic group, and $H \leq G$ an
  open subgroup. Is it the case that
  $$
  \{\real(z) \mid z \text{ a pole of } \zeta_{H}(s)\}= \{\real(z) \mid
  z \text{ a pole of } \zeta_{G}(s)\}?
  $$
\end{qun}

We note that the maxima of these two finite sets coincide, because the
abscissa of convergence of the representation zeta function of a rigid
group is a commensurability invariant;
see~\cite[Corollary~4.5]{LaLu08}.

In the proof of Theorem~\ref{thmABC:lalu} we use standard results from
the theory of orders in central simple algebras over number fields to
show that almost all of the non-archimedean Euler factors of the zeta
function $\zeta_\Gamma(s)$ of the arithmetic group $\Gamma$ are of the
form $\zeta_{\SL_3(\lri)}(s)$ or $\zeta_{\SU_3(\Lri,\lri)}(s)$.  By
Theorem~\ref{thmABC:SL3}, the abscissae of convergence of these local
zeta functions are~$2/3$. Theorem~\ref{thmABC:poles} allows us to
deduce that $2/3$ is also an upper bound for the local abscissae of
convergence for the remaining finitely many exceptional places.  In a
second step we use approximative Clifford theory to `lift' our
analysis of representations of principal congruence subgroups to the
encompassing groups $\SL_3(\lri)$ and $\SU_3(\Lri,\lri)$,
respectively.  Whilst we do not produce explicit formulae for the
representation zeta functions of these groups, our method is designed
to retain enough control over suitable approximations to pin down the
global abscissa of convergence.  The formulae for the zeta functions
of the groups $\SL_3(\lri)$ given in our paper~\cite{AvKlOnVo_pre_a}
are derived using an explicit analysis of similarity classes of
$3\times3$-matrices over~$\lri$.

It is noteworthy that, in the case of arithmetic groups $\Gamma$ of
type $A_2$, the abscissa of convergence of the zeta function
$\zeta_{\Gamma}(s)$ coincides with the abscissa of convergence of the
Euler product of the zeta functions of the finite groups of Lie type
$\bfG(\gri_v/\mfp_v)$; see, specifically,
Proposition~\ref{pro:first_product}. Explicit computations of Euler
products of finite groups of Lie type suggest, however, that the
latter is a function of the Lie rank which tends to zero as the Lie
rank tends to infinity. This does not yield the abscissa of
convergence of $\zeta_\Gamma(s)$ as, by \cite[Theorem 8.1]{LaLu08},
the abscissae of convergence of almost all of the local factors
$\zeta_{\bfG(\smallgri_v)}(s)$ are bounded away from zero.

The proof of Theorem~\ref{thmABC:SL3} follows from a concrete
application of our $\mfp$-adic formalism and a description of the
irregular locus in the complex Lie algebra $\mathfrak{sl}_3(\C)$.  The
readiness with which our method can be applied in this concrete
situation is remarkable; but, of course, it will be much more
challenging to understand the algebraic varieties arising from higher
dimensional groups.  We remark that Theorem~\ref{thmABC:SL3} also
provides concrete examples of the general assertions made in
Theorem~\ref{thmABC:funeq} and illustrates some of the statements of
Theorem~\ref{thmABC:poles}.

A precursor and source of inspiration for the study of representation
zeta functions is the area of subgroup growth; see \cite{LuSe03}.
Methods and tools from one subject area can be transferred to the
other, and there is significant common ground, in particular regarding
the class of `semisimple' compact $p$-adic analytic groups.
Theorems~\ref{thmABC:funeq}, \ref{thmABC:poles} and \ref{thmABC:SL3}
should be contrasted with our extremely limited knowledge regarding
the subgroup zeta functions of `semisimple' compact $p$-adic analytic
pro-$p$ groups.  For example, the subgroup zeta function of
$\SL^1_3(\Z_p)$, or even just its pole spectrum, remains widely
unknown; cf.~\cite[p.~431]{LuSe03}.

%

\subsection{Organisation and notation}
\subsubsection{Organisation of the paper}
In the first part of the paper, we develop methods from $\mfp$-adic
integration to study representation zeta functions of certain compact
$p$-adic analytic groups. One of the key tools facilitating this is
the Kirillov orbit method.  In Section~\ref{sec:kirillov}, we review
this method for potent and saturable pro-$p$ groups, together with
some of the necessary background from $p$-adic Lie theory. In
Section~\ref{sec:poincare} we exhibit how the Kirillov orbit method
allows us to transform the problem of computing representation zeta
functions of the relevant groups into the problem of computing certain
Poincar\'e series. These in turn are amenable to tools from the theory
of $\mfp$-adic integration developed in the study of Igusa local zeta
functions and generalisations thereof.  We utilise this approach to
prove, in Section~\ref{sec:funeq}, Theorems~\ref{thmABC:funeq} and
\ref{thmABC:poles}.  In Section~\ref{sec:semisimple} we explain a link
between representation zeta functions of `semisimple' $p$-adic
analytic groups and a filtration of their associated Lie algebras,
refining the loci of irregular elements.

In the second part of the paper we apply the $\mfp$-adic formalism
developed in the first part to the study of representation zeta
functions of arithmetic groups of type~$A_2$, and related pro-$p$
groups.  We start, in Section~\ref{sec:explicit_cong_sub}, with the
computation of the zeta functions of principal congruence subgroups of
the compact $p$-adic analytic groups $\SL_3(\lri)$ and
$\SU_3(\Lri,\lri)$, thereby proving Theorem~\ref{thmABC:SL3}.  The
proof of Theorem~\ref{thmABC:lalu} is given in
Section~\ref{sec:abscissae}. The theorem is proved for groups of type
${}^1 \! A_2$ (inner forms) in Section~\ref{subsec:abscissa_inner} and
for groups of type ${}^2 \!  A_2$ (outer forms) in
Section~\ref{subsec:abscissa_outer}. We provide some background on
Dirichlet generating functions in Section~\ref{subsec:dirichlet} and
from Clifford theory in Section~\ref{subsec:clifford}.

\subsubsection{Notation}
Throughout the paper, $p$ is a rational prime.  We denote by $\gfi$ a
number field, with ring of integers $\gri$.  For any non-archimedean
place $v$ of $k$, we denote by $\gfi_v$ the completion of $\gfi$ with
respect to $v$ and by $\gri_v$ the corresponding complete valuation
ring, with maximal ideal $\mfp_v$ and residue field $\F_{q_v} = \gri_v
/ \mfp_v$ of size $q_v$.  For any finite set $S$ of places of $k$,
including all the archimedean ones, we write $\gri_S := \{ x \in k
\mid \forall v \not \in S: x \in \gri_v \}$ for the ring of
$S$-integers in $\gfi$.  We denote by $\Lambda$ an $\gri$-Lie lattice.
A finite extension field of $k$ is typically denoted by $K$, its ring
of integers $\Gri$, etc.

The field of $p$-adic numbers is denoted by $\Q_p$, the ring of
$p$-adic integers by $\Z_p$.  More generally, $\lfi$ denotes a
$\mfp$-adic field, i.e.\ a non-archimedean local field of
characteristic $0$.  By $\lri$ we denote a compact discrete valuation
ring of characteristic $0$ and residue field characteristic $p$;
typically $\lri$ is the ring of integers of the $\mfp$-adic field
$\lfi$.  The maximal ideal of $\lri$ is denoted by $\mfp = \pi \lri$,
where $\pi$ is a chosen uniformiser.  We write $\F_q$ for the residue
field $\lri/\mfp$ of cardinality $q = p^f$, where $f = f(\lfi,\Qp) =
f(\lri,\Zp)$ is the (absolute) degree of inertia.

For $x \in \lfi \setminus \{0\}$, we write $\lvert x \rvert_\mfp$ to
denote the $\mfp$-adic absolute value of $x$, normalised so that
$\lvert \pi \rvert_\mfp = q^{-1}$.  For a finite set $X \subseteq
\lfi$ we set $\| X \|_\mfp := \max \{ \lvert x \rvert _\mfp \mid x \in
X \}$.  We write $\mu$ for the additive Haar measure on spaces like
$\lri^n$, normalised so that $\mu(\lri^n)=1$.  A finite extension
field of $\lfi$ is denoted by $\Lfi$, its ring of integers by $\Lri$
with maximal ideal $\mfP$.  We write $f(\Lfi,\lfi) = f(\Lri,\lri)$ for
the (relative) degree of inertia and $e(\Lfi,\lfi) = e(\Lri,\lri)$ for
the (relative) ramification index.

We give a short pictorial summary of the global and local notation.
$$
\begin{matrix}
  \Gri & \subset & \Gfi & \phantom{x} \hookrightarrow \phantom{x} &
  \Lfi &  \supset & \Lri & \triangleright & \mfP \\
  \vert & & \vert &  & \vert & & \vert & & \vert\\
  \gri & \subset & \gfi & \hookrightarrow & k_v = \lfi & \supset &
  \gri_v = \lri &  \triangleright & \mfp_v = \mfp \\
  \vert & & \vert  & & \vert & & \vert & & \vert \\
  \Z & \subset & \Q & \hookrightarrow & \Q_p & \supset & \Z_p &
  \triangleright
  & p\Z_p
\end{matrix}
\qquad\qquad
\begin{matrix}
  \F_{q^f}, & f = f(\Lfi,\lfi) \\ \vert & \\ \F_q & \\ \vert & \\ \F_p &
\end{matrix}
$$

Typically, we use the usual notation, like $\N^n$, $\lri^n$, $\F_q^n$
etc., to denote Cartesian powers.  However, in cases where this
notation could be misunderstood we make a small modification.  For
instance, we write $\mfp^n$ to denote the $n$-th power of the maximal
ideal $\mfp$, and $\mfp^{(n)}$ for the $n$-fold Cartesian power
$\times_{i=1}^n \mfp$.  The transpose of a matrix $M$ is denoted
by~$M^{\textup{t}}$.  The non-trivial Galois automorphism of a
quadratic field extension $\Lfi \vert \lfi$ is denoted by $\sigma$,
the resulting standard involution on the matrix algebra $\Mat_n(\Lfi)$
by $\circ$, i.e.\ $\bfx^\circ := (\bfx^\sigma)^\textup{t}$.

We write $F^*$ to denote the multiplicative group of a field $F$ and
extend this notation as follows.  For any non-trivial $\lri$-module
$M$ we write $M^* := M \setminus \mfp M$.  For the trivial
$\lri$-module $\{0\}$ we define $\{0\}^* := \{0\}$.  The Pontryagin
dual of a compact abelian group $\mathfrak{a}$, e.g.\ the additive
group of an $\lri$-module, is $\widehat{\mathfrak{a}} :=
\Irr(\mathfrak{a}) := \Hom_\Z^{\textup{cont}}(\mathfrak{a},\C^*)$.
More generally, we write $\widehat{G} = \Irr(G)$ for the collection of
continuous, irreducible complex characters of a profinite group $G$.
If $N \trianglelefteq G$, the inertia group of $\theta \in \Irr(N)$ in
$G$ is denoted by $I_G(\theta)$.  The following list collects further
key notation, together with the number of the section where it is
introduced.

\smallskip

\begin{center}
\begin{tabular}{r||l|l}
  $\alpha(\Gamma)$ & abscissa of convergence &
  \ref{subsec:background_motivation} \\ $\Irr_n(\mathfrak{h})$ &
  characters of level $n$ & \ref{subsec:kirillov} \\
  $\mathcal{R}_{\mathfrak{g},\mathbf{b}}(\mathbf{Y})$ & commutator
  matrix &\ref{subsec:poincare} \\ $W(\lri) = (\lri^d)^*$ & region of
  integration & \ref{subsec:poincare} \\
  $\mathcal{P}_{\mathcal{R},\lri}(s)$ & Poincar\'e series &
  \ref{subsec:poincare} \\ $\mathcal{Z}_\lri(r,t)$ & $\mfp$-adic
  integral &\ref{subsec:integration} \\ $\rho$ & maximal rank
  &\ref{subsec:integration} \\ $\mathcal{V}_i$, $\mathcal{W}_i$ &
  varieties in stratifications & \ref{sec:semisimple} \\ $\kappa_0$ &
  normalised Killing form &\ref{sec:semisimple}
\end{tabular}
\end{center}


\part{$\mfp$-Adic formalism}\label{part:formalism}


\section{$p$-Adic analytic pro-$p$ groups and the Kirillov orbit
  method}\label{sec:kirillov} Much of the theory of compact $p$-adic
analytic pro-$p$ groups can be developed satisfactorily by using the
concept of uniformly powerful pro-$p$ groups; see \cite{DiDuMaSe99}.
However, for a more complete picture of $p$-adic Lie theory it is
advantageous to work with saturable pro-$p$ groups; see
\cite{Go07,Kl05}.

\subsection{} \label{subsec:pro-p} The notion of saturability goes
back to Lazard and is based on valuation maps; see \cite{La65, Kl05}.
In \cite{Go07}, Gonz\'alez-S\'anchez has given the following useful
characterisation.  A finitely generated pro-$p$ group $G$ is
\emph{saturable} if and only if it is torsion-free and admits a potent
filtration, i.e.\ a descending series $G_i$, $i \in \N$, of normal
subgroups of $G$ such that (i) $G = G_1$, (ii) $\bigcap_{i \in \N} G_i
= 1$, (iii) $[G_i,G] \subseteq G_{i+1}$ and $[G_i,_{p-1} G] \subseteq
G_{i+1}^p$ for all $i \in \N$.  (Here $[G_i,_{p-1} G]$ denotes the
left-normed iterated commutator with one occurrence of $G_i$ and $p-1$
occurrences of $G$, and $G_{i+1}^p$ stands for the group generated by
$p$-th powers of elements of~$G_{i+1}$.)  Uniformly powerful pro-$p$
groups and, more generally, torsion-free finitely generated pro-$p$
groups $G$ with $\gamma_p(G) \subseteq \Phi(G)^p$ are saturable.
In~\cite{GoKl09} it is shown that every torsion-free $p$-adic analytic
pro-$p$ group of dimension less than $p$ is saturable.  Klopsch proved
that every insoluble maximal $p$-adic analytic just-infinite
pro\nobreakdash-$p$ group of dimension less than $p-1$ is saturable;
see~\cite{Kl05}.  These groups occur naturally as maximal open pro-$p$
subgroups of automorphism groups of semisimple $p$\nobreakdash-adic
Lie algebras and they provide a rich class of groups whose
representation zeta functions are of considerable interest.

To a saturable pro-$p$ group $G$ one associates a saturable $\Z_p$-Lie
lattice $\mfg=\log(G)$, which coincides with $G$ as a topological
space, and there is a tight Lie correspondence between $G$ and~$\mfg$;
see \cite{GoKl09}.  The simplest way to recover the group $G$ from its
$\Z_p$-Lie lattice $\mfg$ is by defining a group multiplication on
$\mfg$ via the Hausdorff series.  When we start from a saturable
$\Z_p$-Lie lattice $\mfg$, we denote this group by $\exp(\mfg)$.  The
following straightforward proposition characterises saturable pro-$p$
groups which are FAb, i.e.\ which have the property that every open
subgroup has finite abelianisation.

\begin{pro} \label{FAb-perfect} Let $G$ be a saturable pro-$p$ group
  and $\mfg=\log(G)$ the associated saturable $\Z_p$-Lie lattice.
  Then the following are equivalent.
  \begin{enumerate}
  \item $G$ is FAb.
  \item $G$ has finite abelianisation $G/[G,G]$.
  \item $\mfg$ has finite abelianisation
    $\mfg/[\mfg,\mfg]$.
  \item $\Q_p \otimes \mfg$ is a perfect $\Q_p$-Lie algebra.
  \end{enumerate}
\end{pro}

\begin{proof}
  Everything follows from the observation that the terms of the
  derived series of the group $G$ and of the Lie lattice
  $\mfg$ coincide as sets; see~\cite[Corollary~4.7]{Go07}.
\end{proof}

\noindent We remark that a compact $p$-adic analytic group is FAb if
and only if it has an open FAb saturable pro-$p$ subgroup.

In applying the Kirillov orbit method, we will be dealing with potent
$\Z_p$-Lie lattices and pro-$p$ groups.  We recall that a $\Z_p$-Lie
lattice $\mfg$ is \emph{potent} if $\gamma_{p-1}(\mfg) \subseteq
p\mfg$ for $p>2$ and $\gamma_2(\mfg) \subseteq 4\mfg$ for $p=2$;
similarly, a pro-$p$ group $G$ is potent if $\gamma_{p-1}(G) \subseteq
G^p$ for $p>2$ and $\gamma_2(G) \subseteq G^4$ for $p=2$.  If $G$ is a
saturable pro-$p$ group and $\mfg$ the associated $\Z_p$-Lie lattice,
then $G$ is potent if and only if $\mfg$ is potent.  Saturable pro-$p$
groups of dimension less than $p$ are potent.  In particular, the main
theorem in \cite{GoKl09} shows that every torsion-free $p$-adic
analytic pro-$p$ group of dimension less than $p$ is saturable and
potent.

\begin{dfn} \label{def:permissible} Let $\lri$ be a compact discrete
  valuation ring of characteristic $0$ and residue field
  characteristic $p$, and let $\mfg$ be an $\lri$-Lie
  lattice. For $m \in \N_0$, let $\mfg^m := \mfp^m \mfg$,
  where $\mfp$ denotes the maximal ideal of $\lri$.  We call $m \in
  \N_0$ \emph{permissible for $\mfg$} if $\mfg^m$ is potent
  and saturable as a $\Z_p$-Lie lattice.
\end{dfn}

The following effective result shows that, for a given $\lri$-lattice
$\mfg$, almost all nonnegative integers are permissible.

\begin{pro}\label{pro:Zlatticeprep}
  Let $\lri$ be a compact discrete valuation ring of characteristic
  $0$ and residue field characteristic $p$, and let $\mfg$ be an
  $\lri$-Lie lattice.  Let $m \in \N_0$ and let $e := e(\lri, \Z_p)$
  be the absolute ramification index of $\lri$.

  If $m > e(p-1)^{-1}$, then $\mfg^m$ is saturable.  Moreover, if
  $p>2$ and $m \geq e(p-2)^{-1}$, then $\mfg^m$ is potent.  If
  $p=2$ and $m \geq 2e$, then $\mfg^m$ is potent.
\end{pro}

\begin{proof}
  Suppose that $m > e(p-1)^{-1}$.  For $x \in \mfg^m \setminus \{ 0\}$
  define $w(x) := \frac{1}{e}\max \{ n \in \N \mid x \in \mfg^n \}$, and put
  $w(0) := \infty$.  Then the map $w: \mfg^m \rightarrow \R_{>0} \cup
  \{\infty\}$ satisfies Lazard's condition for a valuation map,
  summarised in \cite[Section~2]{Kl05}, and $\mfg^m$ is saturable.

  Suppose that $p>2$ and $m \geq e(p-2)^{-1}$.  Then
  $\gamma_{p-1}(\mfg^m) \subseteq \mfg^{m(p-1)}\subseteq
  \mfg^{e+m} = p \mfg^m$, hence $\mfg^m$ is potent.

  Suppose that $p=2$ and $m \geq 2e$.  Then
  $[\mfg^m,\mfg^m] \subseteq \mfg^{2m} \subseteq
  \mfg^{2e+m} = 4 \mfg^m$, hence $\mfg^m$ is
  potent.
\end{proof}

Note in particular that, if $e(\lri,\Zp)=1$, then for $p>2$ every $m
\geq 1$ is permissible for every $\lri$-Lie lattice $\mfg$, and,
similarly, for $p=2$ every $m \geq 2$ is permissible.

\subsection{} \label{subsec:kirillov} The Kirillov orbit method in the
context of compact $p$-adic analytic groups was first developed by
Howe in~\cite{Ho77b}.  Jaikin-Zapirain applied the theory to study
representation zeta functions of FAb compact $p$-adic analytic groups
and extended Howe's work; see~\cite{Ja06}.  A useful description of
the Kirillov orbit method for potent saturable pro-$p$ groups can be
found in~\cite{Go08}.

Let $G$ be a FAb potent saturable pro-$p$ group and $\mfg$ the
associated potent saturable $\Z_p$\nobreakdash-Lie lattice.  We denote
by $\widehat{\mfg} := \Irr(\mfg)$ the Pontryagin dual of the compact
abelian group $(\mfg,+)$, viz.\ the group
$\Hom_\Z^{\textup{cont}}(\mfg,\mathbb{C}^*) =
\Hom_\Z^{\textup{cont}}(\mfg,\mu_{p^\infty})$ of continuous
irreducible complex characters of the additive group $\mfg$, where
$\mu_{p^\infty} \cong \Q_p / \Z_p$ stands for the group of complex
roots of unity of $p$-power order.  To $\omega \in \widehat{\mfg}$ we
associate the bi-additive and antisymmetric form
$$
b_\omega: \mfg \times \mfg \rightarrow \mu_{p^\infty}, \quad
b_\omega(x,y) := \omega([x,y]).
$$
The radical of the form $b_\omega$ is
$$
\Rad(\omega) := \Rad(b_\omega) = \{ x \in \mfg \mid \forall y
\in \mfg : b_\omega(x,y) = 1 \}.
$$
According to \cite[Corollary~2.13]{Ja06} and \cite[Theorem~5.2]{Go08}
we have
\begin{equation} \label{Andrei's_formula}
  \zeta_G(s) =
  \sum_{\omega \in \Irr(\mfg)} \lvert
  \mfg:\Rad(\omega) \rvert^{-(s+2)/2}.
\end{equation}

Next suppose that $\mfg$ is an $\lri$-Lie lattice, where $\lri$ is a
compact discrete valuation ring of characteristic $0$ and residue
field characteristic $p$.  Suppose that $m \in \N_0$ is permissible
for $\mfg$.  In order to make effective use of the Kirillov orbit
method in studying the representation zeta function of the $p$-adic
analytic pro-$p$ group $\mathsf{G}^m = \exp(\mfg^m)$, we recall a
basic description of the dual $\widehat{\mfg^m}$ of $\mfg^m$, viewed
as a $\Z_p$-module, in terms of $\lri$-functionals of $\mfg^m$, viewed
as an $\lri$-module.

\begin{lem}\label{lem:dual}
  The dual $\widehat{\mfh}$ of an $\lri$-Lie lattice $\mfh$ can be
  written as a disjoint union
  $$
  \widehat{\mfh} = \bigdotcup\nolimits_{n \in \N_0} \Irr_n(\mfh),
  \qquad \text{where } \Irr_n(\mfh) \cong
  \Hom_{\lri}(\mfh,\lri/\mfp^n)^*.
  $$
\end{lem}

\begin{proof}
  Let $\lfi$ denote the field of fractions of $\lri$, which is a
  finite extension of $\Q_p$.  The trace map induces a non-degenerate
  $\Q_p$-bilinear form $\lfi \times \lfi \rightarrow \Q_p$, $(x,y)
  \mapsto \Tr_{\lfi \vert \Q_p}(xy)$.  The codifferent of $\lfi \vert
  \Q_p$ is the fractional ideal consisting of all $x \in \lfi$ such
  that $\Tr_{\lfi \vert \Q_p}(xy) \in \Z_p$ for all $y \in \lri$.  The
  different $\diff_{\lfi \vert \Q_p}$ of the extension $\lfi \vert
  \Q_p$ is the inverse of the codifferent and thus an ideal of $\lri$,
  say $\mfp^\delta$.  Regarding the parameter $\delta = \delta(\lfi)
  \in \N_0$, it is known that $\delta \geq e(\lfi,\Z_p) - 1$ with
  equality if and only if $e(\lfi,\Q_p)$ is prime to $p$; see
  \cite[Chapter~III]{Se79}.  The Pontryagin dual of $(\lfi,+)$ is
  $\widehat{\lfi} := \Hom_\Z^\textup{cont}(\lfi,\C^*) =
  \Hom_\Z^\textup{cont}(\lfi,\mu_{p^\infty})$.  Via the Tate character
  $$
  \chi: \lfi \overset{\mathrm{Tr}_{\lfi \vert \Q_p}}{\longrightarrow}
  \Q_p \longrightarrow \Q_p/\Z_p \overset{\cong}{\longrightarrow}
  \mu_{p^\infty}
  $$
  one obtains an isomorphism
  $$
  \lfi \rightarrow \widehat{\lfi}, \quad x \mapsto \chi_x, \qquad
  \text{where $\chi_x : \lfi \rightarrow \mu_{p^\infty}, \, y \mapsto
    \chi(xy)$,}
  $$ showing that $\lfi$ is self-dual; cf.~\cite[Chapter~XV,
    Section~2.2]{CaFr67}.  Let $\mfp = \pi \lri$.  From $\diff_{\lfi
    \vert \Q_p} = \mfp^\delta = \pi^\delta \lri$ one obtains, for
  each $n \in \N_0$, a (non-canonical) pairing of finite abelian
  groups
  $$
  (\lri / \mfp^n) \times (\mfp^{-n} / \lri) \rightarrow
  \mu_{p^\infty}, \quad (x + \mfp^n,y + \lri) \mapsto
  \chi(\pi^{-\delta} xy).
  $$
  As $\mfh$ is a free $\lri$-module of finite rank, this translates
  readily into an isomorphism of the finite abelian groups
  $\widehat{\mfh/\mfp^n \mfh}$ and $\mfp^{-n} \mfh / \mfh$.  On the
  other hand there is, of course, for each $n \in \N_0$ the
  (non-canonical) isomorphism of $\lri$-modules $\mfp^{-n} \mfh / \mfh
  \rightarrow \mfh / \mfp^n \mfh$, induced by multiplication by
  $\pi^n$, and a natural pairing of $\lri$-modules
  $$
  (\mfh/\mfp^n\mfh) \times \Hom_{\lri}(\mfh/\mfp^n\mfh,\lri/\mfp^n)
  \rightarrow \lri/\mfp^n
  $$
  showing that $\mfh/\mfp^n\mfh \cong \Hom_{\lri}(\mfh,\lri/\mfp^n)$.
  We thus obtain
  $$
  \widehat{\mfh} = \bigcup_{n \in \N_0} \widehat{\mfh / \mfp^n \mfh}
  \cong \bigdotcup_{n\in\N_0} (\mfp^{-n} \mfh / \mfh)^* \cong
  \bigdotcup_{n\in\N_0} (\mfh / \mfp^n \mfh)^* \cong
  \bigdotcup_{n\in\N_0}\Hom_{\lri}(\mfh,\lri/\mfp^n)^*.
  $$
\end{proof}

From the proof of the lemma we observe that, in the case $\mfh =
\mfg^m$, the description of $\widehat{\mfg^m}$ in terms of
$\lri$-functionals is $\mathsf{G}^m$-equivariant so that one can
accurately describe the co-adjoint orbits in this model.  The
\emph{level} of $\omega \in \widehat{\mfg^m}$ is defined so that
$\lev_\lri(\omega) = n$ if $\omega \in \Irr_n(\mfg^m)$.  If $\lri$ is
unramified, the level of any $\omega \in \widehat{\mfg^m}$ is simply
the logarithm to base $p$ of its order, viz.\ $\min \{ n \in \N_0 \mid
\omega(p^n \mfg^m) = 1 \}$.


\section{Poincar\'e series and $\mfp$-adic
  formalism} \label{sec:poincare}

Let $\lri$ be a compact discrete valuation ring of characteristic $0$,
with maximal ideal $\mfp = \pi \lri$, field of fractions~$\lfi$ and
residue field $\lri/\mfp$ of cardinality $q$ and characteristic~$p$.
Let $\mfg$ be an $\lri$-Lie lattice such that $\lfi \otimes_{\lri}
\mfg$ is perfect, with $\dim_\lfi(\lfi \otimes_\lri \mfg)=d$, say.
According to Sections~\ref{subsec:pro-p} and \ref{subsec:kirillov},
for every sufficiently large $m \in \N_0$, the $m$-th principal
congruence $\lri$-Lie sublattice $\mfg^m := \mfp^m \mfg$ of $\mfg$
corresponds to a FAb $p$-adic analytic pro-$p$ group $\mathsf{G}^m =
\exp(\mfg^m)$, whose irreducible complex characters can be captured by
means of the Kirillov orbit method.  Our aim in this section is to
express the zeta functions $\zeta_{\mathsf{G}^m}(s)$ in terms of a
suitable Poincar\'e series and to derive an integral formula,
connecting $\zeta_{\mathsf{G}^m}(s)$ to Igusa local zeta functions.
Our treatment makes use of the underlying $\lri$-module structure of
the Lie lattice $\mfg$ and reveals that the dependence of
$\zeta_{\mathsf{G}^m}(s)$ on $m$ can be captured in a simple
factor~$q^{dm}$.

\subsection{Poincar\'e series} \label{subsec:poincare} In this
section we set up the Poincar\'e series which helps us to
establish, for each permissible $m$, a link between the representation
zeta function of the group $\mathsf{G}^m$ and a suitable generalised
Igusa local zeta function; see Proposition~\ref{pro:zeta=poincare}.

We fix an $\lri$-basis $\mathbf{b} := (b_1,\ldots,b_d)$ for the
$\lri$-Lie lattice $\mfg$.  The structure constants $\lambda_{ij}^h$
of the $\lri$-Lie lattice $\mfg$ with respect to the basis
$\mathbf{b}$ are encoded in the commutator matrix
\begin{equation} \label{equ:commutator_matrix} \mathcal{R}(\mathbf{Y})
  := \mathcal{R}_{\mfg,\mathbf{b}}(\mathbf{Y}) = \left( \sum_{h=1}^d
    \lambda_{ij}^h Y_h \right)_{ij} \in \Mat_d(\lri[\mathbf{Y}] )
\end{equation}
whose entries are linear forms in independent variables
$Y_1,\ldots,Y_d$. Clearly, for any $m \in \N_0$, the commutator matrix
of the $\lri$-Lie lattice $\mfg^m$ with respect to the shifted basis
$\pi^m \mathbf{b}$ is $\pi^{m} \mathcal{R}(\mathbf{Y})$.

We write $W(\lri) := \left( \lri^d \right)^*$ and consider the
finite affine cones
\begin{equation} \label{equ:affine_cones} W_n(\lri) := \left( W(\lri)
    + (\mfp^n)^{(d)} \right) / (\mfp^n)^{(d)} = \left( (\lri /
    \mfp^n)^d \right)^*, \quad n \in \N_0.
\end{equation}
Let $\mathbf{y} \in W(\lri)$.  Note that $\mathcal{R}(\mathbf{y})$ is
an antisymmetric $d \times d$ matrix over $\lri$.  Therefore its
elementary divisors are of the form $\mfp^a$ with $a \in \N_0 \cup \{
\infty \}$.  Moreover, if $d$ is even, they can be arranged in pairs.
If $d$ is odd, they also come in pairs, except for a single extra
divisor equal to $\mfp^\infty = \{0\}$.  Define
$$
\nu(\mathcal{R}(\mathbf{y})) := (a_1,\ldots,a_{\lfloor d/2 \rfloor})
\in (\N_0 \cup \{\infty\})^{\lfloor d/2 \rfloor},
$$ where $0 \leq a_1 \leq \ldots \leq a_{\lfloor d/2 \rfloor}$ and
$\mathcal{R}(\mathbf{y})$ is congruent to the block-diagonal matrix
comprised of $2 \times 2$ blocks $\left( \begin{smallmatrix} 0 &
    \pi^{a_i} \\ -\pi^{a_i} & 0
  \end{smallmatrix} \right)$, $i \in \{1,\dots,\lfloor d/2 \rfloor\}$,
and a single $1 \times 1$ block $\begin{pmatrix} 0 \end{pmatrix}$ if
$d$ is odd. (Two matrices $A$ and $B$ over a ring $R$ are said to be
congruent if there exist invertible matrices $S$ and $T$ over $R$ such
that $SAT = B$.  If $A,B$ are antisymmetric, one may take $T =
S^\textup{t}$.)

Now let $n \in \N_0$ and let $\overline{\mathbf{y}}$ denote the image
of $\mathbf{y}$ in $W_n(\lri)$.  Then
$\mathcal{R}(\overline{\mathbf{y}})$ is an antisymmetric $d \times d$
matrix over $\lri/\mfp^n$ and the valuations of its elementary
divisors are encoded by the tuple
$$
\nu(\mathcal{R}(\overline{\mathbf{y}})) :=
\nu_n(\mathcal{R}(\mathbf{y})) := (\min \{a_i,n\})_{i \in
  \{1,\dots,\lfloor d/2 \rfloor\}} \in \{0,1,\ldots,n\}^{\lfloor d/2
  \rfloor}.
$$
The central counting function entering into the picture is
$$
\mathcal{N}_{n,\mathbf{a}}^\lri := \# \big\{ \overline{\mathbf{y}} \in
W_n(\lri) \mid \nu(\mathcal{R}(\overline{\mathbf{y}})) = \mathbf{a}
\big\}, \quad (n,\mathbf{a}) \in \N_0 \times \N_0^{\lfloor d/2
  \rfloor}.
$$
It gives rise to the associated Poincar\'e series
\begin{equation} \label{equ:Poincare_new}
  \mathcal{P}_{\mathcal{R},\lri} (s) := \sum_{\substack{n \in \N_0, \;
      \mathbf{a} \in \N_0^{\lfloor d/2 \rfloor}}}
  \mathcal{N}_{n,\mathbf{a}}^\lri \, q^{ - \left( \sum_{i=1}^{\lfloor d/2
        \rfloor} (n-a_i) \right)s}.
\end{equation}
We note that $\mathcal{N}_{n,\mathbf{a}}^\lri \neq 0$ implies
$0 \leq a_1 \leq \ldots \leq a_{\lfloor d/2 \rfloor} \leq n$.

The principal aim of this section is to establish a link between the
representation zeta functions of the groups $\mathsf{G}^m$ associated
to the Lie lattices $\mfg^m$ for permissible $m$ and the Poincar\'e
series defined in \eqref{equ:Poincare_new}.

\begin{pro}\label{pro:zeta=poincare}
  Let $\mfg$ be as defined at the beginning of
  Section~\ref{sec:poincare}.  Then for all $m \in \N_0$ which are
  permissible for $\mfg$ we have
  $$
  \zeta_{\mathsf{G}^m}(s) = q^{dm}
  \mathcal{P}_{\mathcal{R},\lri}(s+2).
  $$
\end{pro}

As an immediate consequence, we observe

\begin{cor}\label{cor:powers_of_q}
  For every permissible $m \in \N_0$, the representation zeta function
  $\zeta_{\mathsf{G}^m}(s)$ is a power series in $q^{-s}$, i.e.\ the
  character degrees of $\mathsf{G}^m$ are powers of $q$.
\end{cor}

The proof of Proposition~\ref{pro:zeta=poincare} requires some
preparation.  In order to apply the Kirillov orbit method we consider
the space $\widehat{\mfg^m}$ which, according to
Section~\ref{subsec:kirillov}, decomposes as $\widehat{\mfg^m} =
\bigdotcup_{n \in \N_0} \Irr_n(\mfg^m)$ where $\Irr_n(\mfg^m) = \{
\omega \in \widehat{\mfg^m} \mid \lev_\lri(\omega) = n \} \cong
\Hom_{\lri}(\mfg^m,\lri/\mfp^n)^*$.  Moreover, this description is
$\mathsf{G}^m$-equivariant.
Note that for every $n \in \N$ there is a natural surjection
$$
\Hom_\lri(\mfg^m,\lri)^* \rightarrow
\Hom_{\lri}(\mfg^m,\lri/\mfp^n)^*
\cong \Irr_n(\mfg^m).
$$
We say that $w \in \Hom_\lri(\mfg^m,\lri)^*$ is a
representative of $\omega \in \widehat{\mfg^m}$ if the image
of $w$ under this surjection corresponds to $\omega$.

The chosen $\lri$-basis $\mathbf{b}$ for $\mfg$ yields
a shifted basis $\pi^m \mathbf{b}$ for $\mfg^m$ and thus an
explicit co-ordinate system
$$
\mfg^m \xrightarrow{\, \simeq \,} \lri^d, \quad z = \sum_{i=1}^d
z_i (\pi^m b_i) \mapsto \underline{z} = (z_1,\ldots,z_d).
$$
Likewise, the dual $\lri$-basis $\mathbf{b}^\vee$ for
$\Hom_\lri(\mfg,\lri)$ and its shift $\pi^{-m} \mathbf{b}^\vee$ for
$\Hom_\lri(\mfg^m,\lri)$ give a co-ordinate system
$$
\Hom_\lri(\mfg^m,\lri)^* \xrightarrow{\, \simeq \,}
W(\lri), \quad w = \sum_{i=1}^d w_i (\pi^{-m} b_i^\vee) \mapsto
\underline{w} = (w_1,\ldots,w_d).
$$
The duality of the two bases means that $w(z) = \underline{z} \cdot
\underline{w} := \sum_{i=1}^d z_i w_i$ for $z \in \mfg^m$ and
$w \in \Hom_\lri(\mfg^m,\lri)^*$.

\begin{lem}\label{central_lem_neu}
  Let $w \in \Hom_\lri(\mfg^m,\lri)^*$ and let $n \in \N_0$.  Consider
  the element $\omega \in \Irr_n(\mfg^m)$ which is represented by $w$.
  Then for every $z \in \mfg^m$ we have
  $$
  z \in \Rad(\omega) \quad \Longleftrightarrow \quad \underline{z}
  \cdot \pi^m \mathcal{R}(\underline{w}) \equiv_{\mfp^n} \underline{0}.
  $$
\end{lem}

\begin{proof}
  Observe that $x$ belongs to $\Rad(\omega)$ if and only if
  $\omega([x,\pi^m b_i]) = 0$ for all $i \in \{1,\ldots,d\}$, and express
  this condition in co-ordinates.
\end{proof}

\begin{lem}\label{lem:eldiv}
  Let $w \in \Hom_\lri(\mfg^m,\lri)^*$ and let $n \in \N_0$.  Consider
  the element $\omega \in \Irr_n(\mfg^m)$ which is represented by $w$.
  Let $\mathbf{a} := \nu(\mathcal{R}(\underline{w}))$.  Then
  $$
  \lvert \mfg^m : \Rad(\omega) \rvert = q^{2
    \sum_{i=1}^{\lfloor d/2 \rfloor} (n - m - \min\{a_i,n-m\})}.
  $$
\end{lem}

\begin{proof}
  We use Lemma~\ref{central_lem_neu}.  If $n \leq m$, then both
  expressions are equal to $1$.  Now suppose that $\tilde n := n-m >
  0$.  Then the claim follows from $q = \lvert \lri : \mfp \rvert$ and
  \begin{align*}
    \lvert \mfg^m : \Rad(\omega) \rvert & = \lvert \lri^{2 \lfloor d/2
      \rfloor} : \bigoplus\nolimits_{i=1}^{\lfloor d/2 \rfloor}
      (\mfp^{\tilde n - \min \{a_i,\tilde n\}} )^{(2)} \rvert \\ & =
      \lvert \lri : \mfp \rvert^{2\sum_{i=1}^{\lfloor d/2 \rfloor}
      (\tilde n - \min \{a_i,\tilde n\})}.
  \end{align*}
\end{proof}

For $\omega \in \widehat{\mfg}$ and $n \in \N_0$ with
$\lev_\lri(\omega) \geq n$ we define $\nu_n(\omega) :=
\nu_n(\mathcal{R}(\underline{w}))$ where $w \in W(\lri)$ is a
representative of $\omega$.  Clearly this definition does not depend
on the particular choice of $w$.

\begin{lem}
  Let $n \in \N_0$ with $n \geq m$ and $\mathbf{a} \in \N_0^{\lfloor
    d/2 \rfloor}$.  Then
  $$ q^{dm} \mathcal{N}_{n-m,\mathbf{a}}^\lri = \# \{ \omega \in
  \Irr_n(\mfg^m) \mid \nu_{n-m}(\omega) = \mathbf{a} \}.
  $$
\end{lem}

\begin{proof}
  Consider the natural projection $\Irr_n(\mfg^m) \rightarrow
  \Irr_{n-m}(\mfg^m)$.  In our parametrisation, this corresponds to
  the natural projection $W_n(\lri) \rightarrow W_{n-m}(\lri)$, and
  each fibre has cardinality $q^{dm}$.
\end{proof}

\begin{proof}[Proof of Proposition~\ref{pro:zeta=poincare}]
  According to \eqref{Andrei's_formula}, \eqref{equ:Poincare_new} and
  the previous three lemmata we have
  \begin{align*}
    \zeta_{\mathsf{G}^m}(s) & = \sum_{n=0}^m \, \sum_{\omega \in
      \Irr_n(\mfg^m)} \lvert \mfg^m:\Rad(\omega)
    \rvert^{-(s+2)/2} \\
    & \quad + \sum_{n=m+1}^\infty \, \sum_{\omega \in
      \Irr_n(\mfg^m)} \lvert \mfg^m:\Rad(\omega)
    \rvert^{-(s+2)/2} \\
    & = q^{dm} + \sum_{n=1}^\infty \, \sum_{\mathbf{a} \in
      \N_0^{\lfloor d/2 \rfloor}} q^{dm}
    \mathcal{N}_{n,\mathbf{a}}^\lri q^{-\left(
        \sum_{i=1}^{\lfloor
          d/2 \rfloor} (n - a_i) \right) (s+2)} \\
    & = q^{dm} \mathcal{P}_{\mathcal{R},\lri}(s+2).
  \end{align*}
\end{proof}

\subsection{Integral formula}\label{subsec:integration}
As explained in \cite[Section~2.2]{Vo10} (in a more general setting
where the underlying matrix need not be antisymmetric), the Poincar\'e
series defined in \eqref{equ:Poincare_new} can be expressed by a
$\mfp$-adic integral generalising Igusa local zeta functions.  We
first state the relevant formulae and then explain all the notation.
In accordance with \cite[Section~2.2]{Vo10}, we can write
\begin{equation} \label{equ:Poincare_Igusa_neu}
  \mathcal{P}_{\mathcal{R},\lri}(s) = 1 + (1 - q^{-1})^{-1}
  \mathcal{Z}_\lri(-s/2, \rho s-d-1)
\end{equation}
with
\begin{equation} \label{equ:integral_neu} \mathcal{Z}_\lri(r,t) =
  \int_{(x,\mathbf{y}) \in \mfp \times W(\lri)}
  \lvert x \rvert_\mfp^t \prod_{j=1}^\rho
  \frac{\lVert F_j(\mathbf{y}) \cup F_{j-1}(\mathbf{y}) x^2
    \rVert_\mfp^r}{\lVert F_{j-1}(\mathbf{y})
    \rVert_\mfp^r} \, d\mu(x,\mathbf{y}),
\end{equation}
where $\mfp \times W(\lri) \subseteq \lri^{d+1}$ and the additive Haar
measure $\mu$ on $\lri^{d+1}$ is normalised so that~ $\mu(\lri^{d+1}) =
1$.  The individual terms occurring in the formulae
\eqref{equ:Poincare_Igusa_neu} and \eqref{equ:integral_neu} have the
following meaning:
\begin{align}
  2 \rho & = \max \{ \rk_{\lfi}(\mathcal{R}(\mathbf{y})) \mid
  \mathbf{y}
  \in \lri^{d} \}, \label{def:rho}\\
  F_j(\mathbf{Y}) & = \{ f \mid f=f(\mathbf{Y}) \text{ a $2j
    \times 2j$ minor of $\mathcal{R}(\mathbf{Y})$} \}, \label{def:F_j}\\
  \lVert F(\mathbf{y}) \rVert_\mfp & = \max \{ \lvert f(\mathbf{y})
  \rvert_\mfp \mid f \in F \}.\nonumber
\end{align}

We note that in accordance with this definition $F_0(\mathbf{Y}) =
\{1\}$, as the empty matrix has determinant $1$.

\begin{rem}\label{rem:minors}
  It is worth noting that the sets $F_j(\mathbf{Y})$ may be replaced
  by (smaller) sets of polynomials defining the same polynomial
  ideals. The ensuing reduction in the number of polynomials may be of
  relevance for a computational approach toward understanding the
  associated varieties.  Specifically, in the proof of
  Theorem~\ref{thmABC:SL3} we make use of the fact that one can
  replace $F_j(\mathbf{Y})$ by the set of all \emph{principal} $2j
  \times 2j$ minors of~$\mathcal{R}(\bfY)$, noting that they are all
  squares in~$\lri[\bfY]$. Indeed, an antisymmetric $d \times
  d$-matrix $M$ over a field has rank less than $2j$ if and only if
  all principal $2j \times 2j$ minors vanish. This is evident if $M$
  is in `Witt normal form'
$$M = \left( \begin{matrix} 0&\Id_r&\\-\Id_r&0\\&&0 \end{matrix}
\right)$$ for some $r\in\{0,\dots,\lfloor d/2\rfloor\}$. For the
general case it suffices to observe that $M$ can be brought into Witt
normal form by simultaneous elementary row and column operations, and
that the set of matrices with the desired property is closed under
these operations. The latter follows from an easy matrix computation,
using the fact that antisymmetric matrices have even ranks.

In fact, it is known that every minor of an antisymmetric matrix is a
quadratic polynomial in Pfaffians (i.e.\ square roots of principal
minors); cf.~\cite{Br03, He69}.
\end{rem}

Equations \eqref{equ:Poincare_Igusa_neu} and \eqref{equ:integral_neu}
show that the Poincar\'e series we are interested in is essentially a
generalised Igusa local zeta function, which can therefore be studied
by a number of known tools; cf.~\cite{VeZu08} and
Section~\ref{sec:funeq}.  The main ingredient in deriving
\eqref{equ:Poincare_Igusa_neu} is that the numbers
$\mathcal{N}_{n,\mathbf{a}}^\lri$ which enter into the definition of
the Poincar\'e series encode precisely the measures $\mu(S)$ of
subsets $S \subseteq \mfp \times W(\lri)$ on which the integrand in
\eqref{equ:integral_neu} is constant.  Indeed, a short indication of
how to derive \eqref{equ:Poincare_Igusa_neu} is as follows.
Temporarily, set
$$
P(x,\mathbf{y}) := \prod_{j=1}^{\lfloor d/2 \rfloor} \frac{\lVert
  F_j(\mathbf{y}) \cup F_{j-1}(\mathbf{y})x^2 \rVert_\mfp}{\lVert
  F_{j-1}(\mathbf{y}) \rVert_\mfp}
$$
for convenience of notation.  Then we have
\begin{align*}
  \mathcal{P}_{\mathcal{R},\lri} (s) - 1& = \sum_{n \in \N}
  \sum_{\mathbf{a} \in \N_0^{\lfloor d/2 \rfloor}}
  \mathcal{N}_{n,\mathbf{a}}^\lri \, q^{ - \left( \sum_{i=1}^{\lfloor
  d/2 \rfloor} (n-a_i) \right)s} \\ & = \sum_{n \in \N}
  \sum_{\substack{\overline{\mathbf{y}} \in W_n(\lri) \\ \mathbf{a} :=
  \nu(\overline{\mathbf{y}})}} \, q^{ - \left( \sum_{i=1}^{\lfloor d/2
  \rfloor} (n-a_i) \right)s} \\ & = \int_{x \in \mfp}
  \frac{1}{(1-q^{-1}) \lvert x \rvert_\mfp} \int_{\mathbf{y} \in
  W(\lri)} \lvert x \rvert_\mfp^{-d} \left( \lvert x \rvert_\mfp^{-2
  \lfloor d/2 \rfloor} P(x,\mathbf{y}) \right)^{-s/2}
  d\mu(x,\mathbf{y}) \\ & = (1-q^{-1})^{-1} \int_{(x,\mathbf{y}) \in
  \mfp \times W(\lri)} \lvert x \rvert_\mfp^{\lfloor d/2 \rfloor s - d
  - 1} P(x,\mathbf{y})^{-s/2} d\mu(x,\mathbf{y}).
\end{align*}
The penultimate equality can be justified as follows.  The region of
integration $\mfp$ for $x$ is divided into disjoint subsets $X_n :=
\mfp^n \setminus \mfp^{n+1} = \{ x \in \lri \mid \lvert x \rvert_\mfp
= q^{-n} \}$, $n \in \N$, and the singleton $\{0\}$, which is
negligible as it has measure $0$.  Fix $n \in \N$.  Evaluating the
integral on $X_n$, the term $((1-q^{-1}) \lvert x \rvert_\mfp)^{-1}$
cancels with the measure of the set $X_n$.  For any $x \in X_n$ and
$\mathbf{y} \in W(\lri)$, the term $\lvert x \rvert_\mfp^{-2 \lfloor
d/2 \rfloor} P(x,\mathbf{y})$ is equal to $q^{2\sum_i (n - \min
\{a_i,n\})}$, where $\mathbf{a} = \nu_n(\mathbf{y})$.  Finally, the
term $\lvert x \rvert_\mfp^{-d}$ is the size of the set $(\lri /
\mfp^n)^{(d)}$ and compensates for the fact that we use integration
over $\mathbf{y} \in W(\lri)$ to represent summation over
$\overline{\mathbf{y}} \in W_n(\lri)$.  A short thought reveals that
the expression $\lfloor d/2 \rfloor$ in the exponent of $\lvert x
\rvert_\mfp$ may be replaced by the invariant $\rho$ without altering
the value of the integral.

We collect the results of this section as a formal corollary to
Proposition~\ref{pro:zeta=poincare}.

\begin{cor}\label{cor:zeta=integral}
  Let $\mfg$ be as defined at the beginning of
  Section~\ref{sec:poincare}.  Then for all $m \in \N_0$ which are
  permissible for $\mfg$ we have
  $$
  \zeta_{\mathsf{G}^m}(s) = q^{dm}
  \left(1+(1-q^{-1})^{-1} \mcZ_\lri(-s/2-1,\rho(s+2)-d-1) \right),
  $$
  where $\mcZ_\lri(r,t)$ is the $\mfp$-adic integral defined
  in~\eqref{equ:integral_neu}.
\end{cor}

\section{Functional equations and poles}\label{sec:funeq}

In this section we prove Theorems~\ref{thmABC:funeq}
and~\ref{thmABC:poles}.  As in Section~\ref{subsec:results} we denote
by $\Lambda$ a Lie lattice over the ring of integers $\gri$ in a
number field~$\gfi$. We assume that $\Lambda$ is rationally perfect,
i.e.\ that $k\otimes_\lri\Lambda$ is a perfect $\gfi$-Lie algebra, of
$\gfi$-dimension~$d$. Given a non-archimedean place $v$ of $\gfi$
above the rational prime~$p$, we denote by $\lri=\gri_v$ the ring of
integers of the completion $k_v$ of $\gfi$ at $v$, with residue field
cardinality~$q_v$. Given a finite extension $\Lri$ of $\lri$, with
maximal ideal $\mfP$, we define the $\Lri$-Lie lattice
$\mfg(\Lri):=\Lri \otimes_\smallgri\Lambda$. We denote by $e =
e(\Lri,\lri)$ the ramification index and by $f = f(\Lri,\lri)$ the
degree of inertia of the given extension. For every $m\in\N_0$ which
is permissible for $\mfg(\Lri)$ (cf. Definition~\ref{def:permissible})
the principal congruence sublattice $\mfg^m(\Lri):=\mfP^m\mfg(\Lri)$
corresponds to a potent and saturable FAb pro-$p$ group
$\mathsf{G}^m(\Lri) = \exp(\mfg^m(\Lri))$. We assume throughout this
section that $m$ is permissible for the relevant Lie lattices.

The $\mfp$-adic formalism developed in Section~\ref{sec:poincare}
applies to the representation zeta
functions~$\zeta_{\mathsf{G}^m(\Lri)}(s)$. However, whereas the setup
in Section~\ref{sec:poincare} is a local one, the results in the
current section exploit the fact that the groups $\mathsf{G}^m(\Lri)$
are all derived from one global object defined over~$\gri$, viz.\ the
Lie lattice~$\Lambda$. This fact has an important consequence for the
integrands of the $\mfp$-adic integrals which we set up to describe
the relevant representation zeta functions. Indeed, with the integral
$\mcZ_{\Lri}(r,t)$ defined as in~\eqref{equ:integral_neu} we have, by
Corollary~\ref{cor:zeta=integral},
$$ \zeta_{\mathsf{G}^m(\Lri)}(s) = q^{fdm} \left(1 + (1-q^{-f})^{-1}
\mcZ_{\Lri}(-s/2-1,\rho(s+2)-d-1)\right).
$$ Recall that the parameter $\rho$ was defined in~\eqref{def:rho}.
The polynomials $F_j(\mathbf{Y})$ occurring in the integrand of
$\mcZ_{\Lri}(r,t)$ are defined over $\gri$, as they are minors of the
commutator matrix $\mcR(\bfY)$ of $\Lambda$ with respect to a fixed
$\gri$-basis. Thus, essentially only the integral's domain of
integration varies with the ring~$\Lri$.

\subsection{}\label{subsec:explicit_formulae}
We prove Theorems~\ref{thmABC:funeq} and \ref{thmABC:poles} by
describing explicit formulae for the relevant $\mfp$-adic integrals,
generalising expressions given in~\cite{Vo10}. We recall from there what
we require to make the present paper reasonably self-contained, and
adapt the notation from~\cite{Vo10} to the conventions of the present
context.

\subsubsection{}\label{new subsubsec}
Fix $n,b,l\in\N$, let $I\subseteq\{1,\dots,n-1\}$, and
$\mathcal{W}(\lri)\subseteq\lri^{b}$ be a union of cosets
modulo~$\mfp^{(b)}$. Section~2 of \cite{Vo10} develops, in a more
general context, formulae for multivariate $\mfp$-adic integrals of
the form

\begin{equation}\label{equ:def_Z}
  Z_{\mathcal{W}(\lri),I}(\bfs):=\int_{\mfp^{(\lvert I \rvert)}\times
    \mathcal{W}(\lri)}\prod_{\kappa=1}^l \left\| \bigcup_{\iota\in
    J_\kappa}\left(\prod_{i\in I}x_i^{e_{i\,\kappa\iota}}\right)
    F_{\kappa\iota}(\bfy) \right\|_\mfp^{s_\kappa}d\mu(\bfx,\bfy)
\end{equation}
 with a view toward proving functional equations under `inversion of
the parameter $q$' for generating functions which are expressible in
terms of such integrals. Here $\bfs=(s_1,\dots,s_l)$ is a vector of
complex variables, the sets $J_\kappa$, $\kappa\in\{1,\dots,l\}$, are
finite index sets, $\bfx=(x_1,\dots,x_{n-1})$ and
$\bfy=(y_1,\dots,y_b)$ are independent variables, $e_{i\,\kappa\iota}$
are nonnegative integers and $F_{\kappa\iota}(\bfY)$ are finite sets
of polynomials over~$\gri$.

A general idea in the study of a $\mfp$-adic integral
like~\eqref{equ:def_Z} is to study its transformation under a
(log-)principalisation of the ideals defined by the polynomials
occurring in the integrand. A principalisation of ideals not only
provides an embedded resolution of singularities for the relevant
varieties, but also renders them locally principal.  The existence of
principalisations of ideals in characteristic~$0$ lies as deep as
Hironaka's theorem on resolutions of singularities in
characteristic~$0$; cf.~\cite{VeZu08} and references therein. The
domain of the transformed integral may be partitioned into finitely
many ($\mfp$-adically open) co-ordinate neighbourhoods on which the
evaluation of the integral is reduced to the computation of zeta
functions enumerating integral points of systems of polyhedral
cones. The multiplicities to which each of these systems occur are
given by the numbers of rational points over finite fields of
constructible algebraic sets. Explicit formulae resulting from this
kind of analysis yield, in many cases, supersets of the set of real
parts of the poles of the integral, and can be used to explain
phenomena like local functional equations. In the present setting, we
obtain principalisations over various local fields from a fixed
principalisation defined over~$\gfi$ by base extension. This allows us
to derive formulae which are `uniform' both with respect to the
variation of the place and under ring extensions.

Concretely, we fix a principalisation $(Y,h)$ over $\gfi$ of the ideal
$\mathcal{I}:=\prod_{\kappa,\,\iota}(F_{\kappa\iota}(\bfY))\subseteq
k[\bfY]$, with numerical data $(N_{t\,
\kappa\iota},\nu_t)_{t\,\kappa\iota}$, where $t\in T,
\kappa\in\{1,\dots,l\}, \iota\in J_\kappa$. Here $T$ is a finite index
set for the irreducible components $E_t$ of the pre-image under $h$ of
the variety defined by~$\mathcal{I}$. For $t, \kappa,\iota$ as above,
the number $\nu_t-1$ denotes the multiplicity of $E_t$ in the divisor
of $h^*(d\mu(\bfy))$, and $N_{t\, \kappa\iota}$ denotes the
multiplicity of $E_t$ in the pre-image under $h$ of the variety
defined by the ideal
$\mathcal{I}_{\kappa\iota}:=(F_{\kappa\iota}(\bfY))\subseteq
\gfi[\bfY]$; cf.~\cite[Section~2.1]{Vo10}.  Theorem~2.2 and
Corollary~2.1 of \cite{Vo10} provide explicit formulae for the
$\mfp$-adic integrals $Z_{\mathcal{W}(\lri),I}(\bfs)$ under additional
invariance and regularity conditions -- specified just before
\cite[Theorem 2.2]{Vo10} -- on the ideals $\mathcal{I}_{\kappa\iota}$
and on $(Y,h)$.  In our applications to representation zeta functions,
the invariance conditions will always be satisfied for all places of
$\gfi$, and are henceforth assumed to hold. They include, in
particular, the condition that $b=n^2$. The regularity conditions
comprise the condition that $(Y_\lfi,h_\lfi)$, the principalisation
over $\lfi$ obtained from $(Y,h)$ by base extension, has good
reduction modulo~$\mfp$, and that none of the ideals
$\mathcal{I}_{\kappa\iota}$ is zero modulo~$\mfp$. We denote by $S$
the finite set of places of~$\gfi$ for which these conditions fail to
hold.

The general strategy outlined above is illustrated in the current
context by \cite[Theorem~2.1]{Vo10} which states that, under the stated
conditions,
\begin{equation}\label{equ:Z good reduction}
  Z_{\mathcal{W}(\lri),I}(\bfs)= (1-q^{-1})^{\lvert I \rvert} q^{-n^2}
    \, \sum_{U\subseteq T} c_{U}(q) \, (q-1)^{\lvert U \rvert} \,
    \Xi_{U,I}(q,\bfs).
\end{equation}
Here, for any $U\subseteq T$, the coefficient
\begin{equation*}
  c_U(q) := c_{U,\mathcal{W}(\lri)}(q) = \lvert \{a \in Y(\lri/\mfp)
  \mid (a\in E_u(\lri/\mfp) \Leftrightarrow u\in U) \text{ and
  }\ol{h_\lfi}(a)\in \ol{\mathcal{W}(\lri)}\} \rvert,
\end{equation*}
where $\overline{\phantom{x}}$ denotes reduction modulo $\mfp$, is the
number of $\Fq$-rational points of the (reduction modulo $\mfp$ of
the) constructible set determined by the boolean combination of the
varieties $E_u$ prescribed by~$U$. Furthermore we set, for $U\subseteq
T$, $\bfm \in \N^{U}, \bfn \in \N^{I}$,
\begin{align*}
 L(\bfm,\bfn)&:=-\sum_{i\in I}n_i-\sum_{u\in U}\nu_u m_u,\\
  L_{\kappa\iota}(\bfm,\bfn)&:=\sum_{i\in
  I}e_{i\,\kappa\iota}n_i+\sum_{u\in U}N_{u\,\kappa\iota}m_u,
  \quad\text{ for } \kappa\in \{1,\ldots,l\}, \iota \in J_\kappa,
\end{align*}
and
\begin{equation}\label{def Xi}
  \Xi_{U,I} (q,\bfs) := \sum_{\substack{\bfm
      \in\N^{U},\;\bfn\in\N^{I}}} q^{L(\bfm,\bfn) - \sum_{\kappa=1}^l
      s_\kappa\min\{L_{\kappa\iota}(\bfm,\bfn) \mid \iota \in
      J_\kappa\}}.
\end{equation}

In order to analyse the functions $\Xi_{U,I}(q,\bfs)$ for fixed $U$
and $I$, we may cover the open orthant $\R_{>0}^{\lvert U \rvert +
\lvert I \rvert}$ by a family $\mfC$ of (not necessarily disjoint)
rational polyhedral cones such that, on each cone $\mcC$ in $\mfC$,
the value of each of the `$\min$'-terms in~\eqref{def Xi} is attained
by a single linear functional~$L_{\kappa\iota}$. Indeed, for
$g=(g_\kappa)\in\prod_{\kappa =1}^l J_\kappa$, let $\mcC_g$ be the
cone in $\R_{>0}^{\lvert U \rvert + \lvert I \rvert}$ defined by the
condition
$$ \forall \kappa \in \{1,\ldots,l\}\,\forall \iota\in
J_\kappa:\,L_{\kappa g_\kappa}\leq L_{\kappa\iota}.
$$ Then the family $\mathfrak{C}:=(\mcC_g)_{g\in\prod_\kappa
J_\kappa}$ of closed polyhedral cones clearly has the desired
property. For each $g\in\prod_\kappa J_\kappa$, the restriction of the
summation in~\eqref{def Xi} to $\mcC_g\cap\N^{\lvert U \rvert + \lvert
I \rvert}$ may be interpreted in terms of a multivariate Hilbert
series associated to the monoid of solutions of a homogeneous system
$A_g\bfx={\bf 0}$ of linear diophantine equations
(cf.~\cite[I.3]{St96}). Here $A_g$ is a matrix with integer
coefficients. It is well-known that these Hilbert series are rational
functions in variables $\xi_1,\dots,\xi_{\lvert U \rvert + \lvert I
\rvert}$, say, and that their denominators are finite products of the
form $\prod_i(1-\prod_{j=1}^{\lvert U \rvert + \lvert I
\rvert}\xi_j^{n_{ij}})$ for integers $n_{ij}$ depending on $\mcC_g$. A
formula for the restriction of the summation in~\eqref{def Xi} to
$\mcC_g\cap\N^{\lvert U \rvert + \lvert I \rvert}$ can be obtained by
substituting for these variables suitable monomials in $q$ and
$q^{-s_\kappa}$, $\kappa \in \{1,\ldots,l\}$.

In the case that $v\not\in S$, \cite[Corollary~2.1]{Vo10} translates
\eqref{equ:Z good reduction} into an explicit formula for the
normalised integrals
$$ \widetilde{Z_{\mathcal{W}(\lri),I}}(\bfs) :=
Z_{\mathcal{W}(\lri),I}(\bfs)\left(\left(1-q^{-1}\right)^{\lvert I
\rvert} \prod_{i=1}^n\left(1-q^{-i}\right)\right)^{-1}.
$$
More precisely, we have that
\begin{equation}\label{equ:tilde good reduction}
 \widetilde{Z_{\mathcal{W}(\lri),I}}(\bfs) =
 \left(\prod_{i=1}^n\frac{q-1}{q^i-1}\right) \sum_{U\subseteq T}
 b_U(q) \sum_{V\subseteq U}(-1)^{\vert U \setminus V\vert} \,
 (q-1)^{\vert V\vert} \, \Xi_{V,I}(q,\bfs),
\end{equation} where, for each $U$, we denote by $b_U(q)$
the number of $\Fq$-rational points of the reduction modulo $\mfp$ of
the smooth projective algebraic variety $E_U:=\cap_{u\in U}
E_u$. These are related to the numbers $c_{U,\mathcal{W}(\lri)}(q)$ in
a simple way by the inclusion-exclusion principle. Crucially for our
application to the proof of Theorem~\ref{thmABC:funeq}, Corollary~2.4
of \cite{Vo10} establishes, for all $i\in\{1,\dots,n-1\}$, a functional
equation of the form
\begin{multline}\label{equ:funeq_zeta}
  \left(\wt{Z_{\mathcal{W}(\lri),\varnothing}}(\bfs) +
    (1-q^{-n})\wt{Z_{\mathcal{W}(\lri),\{i\}}}(\bfs)\right)|_{\substack{q\rightarrow
    q^{-1}\\\lambda_i\rightarrow
    \lambda_i^{-1}}}=\\q^{-n}\left(\wt{Z_{\mathcal{W}(\lri),\varnothing}}(\bfs)
    + (1-q^{-n})\wt{Z_{\mathcal{W}(\lri),\{i\}}}(\bfs)\right).
\end{multline}
Definitions of the terms $\lambda_i$ and a precise explanation of the
operation $q\rightarrow q^{-1}$, $\lambda_i\rightarrow \lambda_i^{-1}$
are given in Section~\ref{subsec:good reduction}.

\subsubsection{}
Our aim for the remainder of this section is fourfold: in
Section~\ref{subsec:connection} we explain how the formalism developed
in Section~\ref{new subsubsec} applies to the study of representation
zeta functions of groups of the form $\mathsf{G}^m(\Lri)$ featuring in
Theorems~\ref{thmABC:funeq} and \ref{thmABC:poles}. In
Section~\ref{subsec:good reduction} we show that, in the case
$v\not\in S$, the formula~\eqref{equ:tilde good reduction} and the
functional equations~\eqref{equ:funeq_zeta} hold uniformly for all
extensions $\Lri\vert\lri$ (possibly ramified), upon substitution of
$q^f$ for $q$, where $f = f(\Lri,\lri)$; the precise meaning of this
substitution is explained in Section~\ref{subsec:good reduction}. This
uniformity phenomenon under field extensions is in analogy with the
situation for Igusa local zeta functions;
cf.~\cite[Theorem~3]{DeMe91}. In Section~\ref{subsec:greenberg} we
give a formula for $Z_{\mathcal{W}(\lri),I}(\bfs)$,
generalising~\eqref{equ:Z good reduction}, which is also valid for
$v\in S$. The main idea there is to write the integral as a sum over
suitably chosen $\mfp$-adic co-ordinate neighbourhoods, indexed by the
cosets of $Y(\lri)$ modulo $\mfp^N$ for some $N\in\N$, on which the
integrand is `monomial'. This generalises the case of good reduction,
where $N=1$ is sufficient. We use the Greenberg transform of level $N$
to identify the points of $Y(\lri/\mfp^N)$ for which the
($\mfp$-adically) `local' integrals are identical with rational points
of constructible subsets of a variety defined over~$\Fq$. (In explicit
co-ordinates this amounts to a description in terms of Witt vectors.)
In Section~\ref{subsec:unramified} we exploit the functoriality of the
Greenberg transform to obtain a formula which holds uniformly for all
unramified extensions of $\lri$. In Section \ref{subsec:arbitrary
extensions} we study totally ramified extensions
$\Lri\vert\lri$. Dealing with these special cases of ramification is
sufficient for our purposes, as every extension $\Lri\vert\lri$ has a
unique maximal unramified subextension $\Lri^{\text{ur}}\vert\lri$
such that $\Lri^{\text{ur}}\vert\Lri$ is totally
ramified. Section~\ref{subsec:cones} contains a key result on poles of
zeta functions of cones. The proofs of Theorems \ref{thmABC:funeq} and
\ref{thmABC:poles} are given in Section~\ref{subsec:proof_thm_poles}.

\subsubsection{}\label{subsec:connection}
The link between the formulae given in Section~\ref{new subsubsec} and
representation zeta functions of groups is as follows. How, in
general, Poincar\'e series enumerating elementary divisors of matrices
of forms may be expressed in terms of integrals of the
form~\eqref{equ:def_Z} is explained in~\cite[Section~2.2]{Vo10}. In this
application of the general $\mfp$-adic integration framework developed
in~\cite{Vo10}, the set $\mathcal{W}(\lri)$ is $\GL_d(\lri)$, but in
fact the integrand only depends on the entries in the first rows, say,
of matrices $\bfy \in \GL_d(\lri)$.  This connects $\mathcal{W}(\lri)$
with the set $W(\lri) = (\lri^d)^*$ introduced in
Section~\ref{subsec:poincare} of the present paper.

In our context, setting $n:=d$, $b:=d^2$ and
$\mathcal{W}(\Lri):=\GL_d(\Lri)$, it is easy to deduce
from~\eqref{equ:integral_neu} that, for suitable data of $l$,
$J_\kappa$, $F_{\kappa\iota}(\bfy)$ in \eqref{equ:def_Z}, and vectors
$\bfa,\bfb\in\Z^l$, $I\in\{\varnothing,\{1\}\}$,
\begin{align}
   \wt{Z_{\mathcal{W}(\Lri),\varnothing}}(\bfa s +
  \bfb)&= 1 ,\nonumber\\
  (1-q^{-fd})\wt{Z_{\mathcal{W}(\Lri),\{1\}}}(\bfa s +
  \bfb)&= (1-q^{-f})^{-1}\mcZ_{\Lri}(-s/2-1,\rho (s+2)-d-1).\label{equ:Z_0 = tilde}
\end{align}
Here we write $\bfa s + \bfb$ for
$(a_1s+b_1,\dots,a_ls+b_l)$. Invoking
Corollary~\ref{cor:zeta=integral} we thus have
\begin{align}\nonumber
  \zeta_{\mathsf{G}^m(\Lri)}(s) &= q^{fdm} \left(1 +
    (1-q^{-f})^{-1}\mcZ_{\Lri}(-s/2-1,\rho (s+2)-d-1)\right)\\&=
    q^{fdm} \left(\wt{Z_{\mathcal{W}(\Lri),\varnothing}}(\bfa s +
    \bfb) + (1-q^{-fd})\wt{Z_{\mathcal{W}(\Lri),\{1\}}}(\bfa s +
    \bfb)\right)\label{equ:zeta=integral II}.
\end{align}

The sets $F_{\kappa\iota}(\bfY)$ featuring in the integrand of the
integral $\mcZ_\Lri(r,t)$ are given by sets of minors of the
commutator matrix $\mcR(\bfY)$ associated to the lattice $\Lambda$;
cf.\ Section~\ref{subsec:integration}.

\subsection{Places with good reduction}\label{subsec:good reduction}
Recall from Section~\ref{new subsubsec} that $(Y,h)$ is a
principalisation of ideals for the ideal $\mathcal{I}\subseteq
\gfi[\bfY]$ defined over the number field $\gfi$. We obtain, by base
extension, a principalisation $(Y_\Lfi,h_\Lfi)$ over every extension
$\Lfi\vert\lfi$, where $\lfi=k_v$ for a non-archimedean place $v$
of~$\gfi$. Throughout Section~\ref{subsec:good reduction} we assume
that $v\not\in S$. For such $v$, all of the principalisations
$(Y_\Lfi,h_\Lfi)$ have good reduction modulo~$\mfP$;
cf.~\cite[Proposition~2.3 and Theorem~2.4]{De87}. The formulae in
\cite[Theorem~2.2 and Corollary~2.1]{Vo10} extend to `uniform' formulae
for
$$ \widetilde{Z_{\mathcal{W}(\Lri),I}}(\bfs) :=
Z_{\mathcal{W}(\Lri),I}(\bfs)\left(\left(1-q^{-f}\right)^{\lvert I \rvert}
\prod_{i=1}^n\left(1-q^{-if}\right)\right)^{-1},
$$ where $\Lri$ is any finite extension of $\lri$ with $f =
f(\Lri,\lri)$: as in~\cite[Theorem~3]{DeMe91} one can argue that an
expression for the integrals $\wt{Z_{\mathcal{W}(\Lri),I}}(s)$ is
obtained by replacing $q$ by $q^f$ on the right hand side
of~\eqref{equ:tilde good reduction}. This description also allows us
to explain functional equations like~\eqref{equ:funeq_zeta}. Indeed,
the numbers $b_U(q^f)$ in~\eqref{equ:tilde good reduction} may, by the
Weil conjectures, be written as alternating sums of powers of
Frobenius eigenvalues, i.e.\
$$
b_U(q^f) = \sum_{i=0}^{2\left(\binom{n}{2} - \lvert U \rvert
  \right)}(-1)^i\sum_{j=1}^{t_{U,i}}\alpha_{U,i,j}^f
$$ for suitable nonnegative integers $t_{U,i}$ and non-zero
complex numbers $\alpha_{U,i,j}$; cf.\ \cite[Section 2.1]{Vo10} and
\cite[Section 2]{DeMe91}.  The numbers $\alpha_{U,i,j}$ have the
property that
$$ b_U(q^{-f}):=\sum_{i=0}^{2\left(\binom{n}{2} - \lvert U \rvert
\right)} (-1)^i \sum_{j=1}^{t_{U,i}}
\alpha_{U,i,j}^{-f}=q^{-f\left(\binom{d}{2} - \lvert U \rvert \right)}
b_U(q^f).
$$ We note that the numbers $\alpha_{U,i,j}$ are in fact algebraic
integers, and relabel them to
$\lambda_1=\lambda_1(v),\dots,\lambda_r=\lambda_r(v)$. The functions
$\Xi_{U,I}(q,\bfs)$ are rational functions in monomials in $q$ and
$q^{-s_\kappa}$, $\kappa\in\{1,\dots,l\}$, and the effect of inverting
these terms is clear. The proof of \cite[Corollary 2.4]{Vo10} may thus
be extended to show that, for $i\in\{1,\dots,n-1\}$,
\begin{multline}\label{equ:funeq good reduction}
  \left(\wt{Z_{\mathcal{W}(\Lri),\varnothing}}(\bfs) +
    (1-q^{-fn})\wt{Z_{\mathcal{W}(\Lri),\{i\}}}(\bfs)\right)|_{\substack{q\rightarrow
    q^{-1}\\\lambda_j\rightarrow \lambda_j^{-1}}}=\\q^{-f
    n}\left(\wt{Z_{\mathcal{W}(\Lri),\varnothing}}(\bfs) +
    (1-q^{-fn})\wt{Z_{\mathcal{W}(\Lri),\{i\}}}(\bfs)\right).
\end{multline}

\subsection{Places with bad reduction}\label{subsec:greenberg}
We now drop the assumption that $v\not\in S$. For every
non-archimedean place $v$ of $\gfi$ there exists $N\in\N$, such that
$Y(\lri)$ can be covered disjointly by cosets modulo $\mfp^N$, which
are homeomorphic to $(\mfp^{N})^{(b)}$ and indexed by the elements of
$Y(\lri/\mfp^N)$, on which the terms in the integrand are monomial,
possibly allowing for a factor with locally constant $\mfp$-adic
valuation, and on which the Haar measure transforms in a similar
way. More precisely, we can choose $N\in\N$ such that, on the cosets
$\{\bfy\in Y(\lri) \mid \bfy \equiv \bfa \mod\mfp^N\}$, indexed by the
elements $\bfa\in Y(\lri/\mfp^N)$, the following holds: there exist
$U=U(\bfa)\subseteq T$, $j=j(\bfa)\in\N_0$ and, for
$\kappa\in\{1,\dots,l\}, \iota\in J_\kappa$, integers
$d_{\kappa\iota}=d_{\kappa\iota}(\bfa)\in\N_0$ such that
\begin{equation*}
 \|F_{\kappa\iota}\circ h\|_\mfp = q^{-d_{\kappa\iota}} \prod_{u\in
 U}\lvert \gamma_u \rvert_\mfp^{N_{u\,\kappa\iota}} \quad \text{ and }
 \quad h^*(d\mu(\bfy))= q^{-j} \prod_{u\in U}\lvert \gamma_u
 \rvert_\mfp^{\nu_u-1} d\mu(\boldsymbol{\gamma}).
\end{equation*}
for co-ordinate functions $\gamma_u, u\in U$.  This follows as
$(Y_\lfi, h_\lfi)$ still yields a principalisation of ideals for the
$\lfi$-manifold defined by $\mathcal{I}$. In fact, we may choose $N$
such that the above properties hold uniformly for all unramified
extensions of $\lri$; cf.~\cite[Theorem~2]{Me86}.  In the case that
$v\not\in S$, one may choose $N=1$ and all of the integers $j(\bfa)$,
$d_{\kappa\iota}(\bfa)$ are zero for all $\bfa\in Y(\lri/\mfp)$. As
the set $S$ is finite, we may thus choose $N$ which is good for all
places $v$ of~$\gfi$. For given $U\subseteq T$,
$(d_{\kappa\iota})\in\N_0^{\prod_{\kappa}J_{\kappa}}$ and
$I\subseteq\{1,\dots,n-1\}$ we set
$$
\Xi^N_{U,(d_{\kappa\iota}),I}(q,\bfs) := \sum_{\substack{\bfm
\in\N_{\geq N}^{U},\; \bfn \in\N^{I}}}
q^{L(\bfm,\bfn)-\sum_{\kappa=1}^l
s_\kappa\min\{L_{\kappa\iota}(\bfm,\bfn)-d_{\kappa\iota} \mid \iota
\in J_\kappa\}}.
$$
\begin{pro}\label{pro: bad reduction}
  For $I\subseteq\{1,\dots,n-1\}$ we have
\begin{equation}\label{equ:formula Z bad reduction}
  Z_{\mathcal{W}(\lri),I}(\bfs) =
      \frac{(1-q^{-1})^{|I|}}{q^{Nn^2}}\sum_{\substack{\bfa\in
      Y(\lri/\mfp^N)\\\ol{h_\lfi}(\bfa)\in \mathcal{W}(\lri/\mfp^N)}}
      (q^N-q^{N-1})^{|U(\bfa)|} q^{-j(\bfa)} \;
      \Xi^N_{U(\bfa),(d_{\kappa\iota}(\bfa)),I}(q,\bfs),
\end{equation}
where $\ol{\phantom{x}}$ denotes reduction modulo~$\mfp^N$.
\end{pro}
\noindent Note that \eqref{equ:formula Z bad reduction} generalises
\eqref{equ:Z good reduction}. The proof of Proposition~\ref{pro: bad
reduction} is analogous to the proof of~\cite[Theorem~2.1]{Vo10}.

We may rewrite \eqref{equ:formula Z bad reduction} by collecting the
points $\bfa\in Y(\lri/\mfp^N)$ for which the summands coincide.  More
precisely we consider, for $U\subseteq T$, $j\in\N_0$,
$(d_{\kappa\iota})\in \N_0^{\prod_\kappa J_\kappa}$, the sets
\begin{multline}
  C_{U,j,(d_{\kappa\iota})}(\lri/\mfp^N)
  :=C_{U,j,(d_{\kappa\iota}),\mathcal{W}(\lri)}(\lri/\mfp^N) :=\\ \{
  \bfa\in Y(\lri/\mfp^N) \mid U(\bfa)=U,\, j(\bfa)=j,\,
  (d_{\kappa\iota}(\bfa))=(d_{\kappa\iota}),\; \ol{h_\lfi}(\bfa)\in
  \mathcal{W}(\lri/\mfp^N)\}.\label{def:c_UJ}
\end{multline}
By the theory of the Greenberg transform (see, e.g., \cite{Gr61}),
there exists a variety $\mathcal{Y}_N$ over $\Fq$ such that
$$Y(\lri/\mfp^N) = \mathcal{Y}_N(\Fq).$$ The sets
$C_{U,j,(d_{\kappa\iota})}(\lri/\mfp^N)$ correspond to the
$\Fq$-rational points of constructible subsets
$\mathcal{C}_{U,j,(d_{\kappa\iota})}$ of $\mathcal{Y}_N$ defined over
$\Fq$, and
$$
c_{U,j,(d_{\kappa\iota})}(q) := \lvert
\mathcal{C}_{U,j,(d_{\kappa\iota})}(\Fq) \rvert = \lvert
C_{U,j,(d_{\kappa\iota})}(\lri/\mfp^N) \rvert.
$$
Thus we obtain
\begin{cor}\label{cor: bad reduction} For $I\subseteq\{1,\dots,n-1\}$
  we have
\begin{equation}
  Z_{\mathcal{W}(\lri),I}(\bfs) = \frac{(1-q^{-1})^{|I|}}{q^{Nn^2}}
  \sum_{U,j,(d_{\kappa\iota})} c_{U,j,(d_{\kappa\iota})}(q) \,
  (q^N-q^{N-1})^{\vert U \vert}q^{-j} \;
  \Xi^N_{U,(d_{\kappa\iota}),I}(q,\bfs).\label{equ:final formula Z bad
  reduction}
\end{equation}
\end{cor}
This generalises the formula given in
\cite[Theorem~2.1]{Vo10}. Following arguments analogous to those in
\cite[Theorem~2.2]{Vo10}, we obtain formulae for the normalised
integrals $\widetilde{Z_{\mathcal{W}(\lri),I}}(\bfs)$ and thus, via
\eqref{equ:Z_0 = tilde} and Corollary~\ref{cor:zeta=integral},
formulae for the respective representation zeta functions.

\begin{rem}
  We do not expect in general an analogue of \cite[Corollary~2.1]{Vo10}
  to hold, as it is not clear how the numbers
  $c_{U,j,(d_{\kappa\iota})}(q)$ may be expressed in terms of the
  numbers of rational points of \emph{smooth} projective
  varieties. Local functional equations may therefore not hold for
  non-archimedean places $v\in S$. This phenomenon is well-known in
  the theory of Igusa local zeta functions and is, in our context,
  illustrated by formulae for the congruence subgroups of
  $\SL_3(\lri)$ in the case of residue field characteristic~$3$;
  cf.~\cite{AvKlOnVo_pre_b}.
\end{rem}

\subsubsection{Unramified extensions}\label{subsec:unramified}
Consider now extensions $\Lri\vert\lri$ with $e(\Lri,\lri)=1$ and $f =
f(\Lri,\lri)$ arbitrary. Recall from the beginning of
Section~\ref{subsec:greenberg} that we may choose $N$ such that we can
write $Z_{\mathcal{W}(\Lri),I}(\bfs)$ as a sum over cosets modulo
$\mfP^N$, with the same `local numerical data' $(U(\bfa), j(\bfa),
(d_{\kappa\iota}(\bfa)))$ as in the case $\Lri=\lri$, uniformly for
all~$f$. By functoriality of the Greenberg transform we have
$$Y(\Lri/\mfP^N) = \mathcal{Y}_N(\mathbb{F}_{q^f})$$ and the sets
$C_{U,j,(d_{\kappa\iota})}(\Lri/\mfP^N)$, defined in analogy
to~\eqref{def:c_UJ}, have the property that
$$
c_{U,j,(d_{\kappa\iota})}(q^f) := \lvert
\mathcal{C}_{U,j,(d_{\kappa\iota})}(\mathbb{F}_{q^f}) \rvert = \lvert
C_{U,j,(d_{\kappa\iota})}(\Lri/\mfP^N) \rvert.
$$
Thus we obtain
\begin{pro} \label{pro:unramified}
 For $I\subseteq\{1,\dots,n-1\}$ and $\Lri\vert\lri$ unramified, i.e.\
 $e(\Lri,\lri)=1$, we have
\begin{align*}
 Z_{\mathcal{W}(\Lri),I}(\bfs) &=
      \frac{(1-q^{-f})^{|I|}}{q^{fNn^2}}\sum_{\substack{\bfa\in
      Y(\Lri/\mfP^N) \\ \ol{h_\Lfi}(\bfa)\in
      \mathcal{W}(\Lri/\mfP^{N})}}
      (q^{fN}-q^{f(N-1)})^{|U(\bfa)|}q^{-fj(\bfa)}
      \Xi^N_{U(\bfa),(d_{\kappa\iota}(\bfa)),I}(q^f,\bfs)\\ & =
      \frac{(1-q^{-f})^{|I|}}{q^{fNn^2}}
      \sum_{U,j,(d_{\kappa\iota})}c_{U,j,(d_{\kappa\iota})}(q^f)(q^{fN}-q^{f(N-1)})^{\vert
      U \vert}q^{-fj} \; \Xi^N_{U,(d_{\kappa\iota}),I}(q^f,\bfs),
\end{align*}
where $\ol{\phantom{x}}$ denotes reduction modulo~$\mfP^N$.
\end{pro}
This generalises Proposition~\ref{pro: bad reduction} and its
Corollary~\ref{cor: bad reduction}.

\subsubsection{Totally ramified extensions}\label{subsec:arbitrary extensions}
We now consider totally ramified extensions $\Lri$ of $\lri$.  In
contrast to Section~\ref{subsec:unramified} where we provided, in
Proposition~\ref{pro:unramified}, a uniform formula for all unramified
extensions of $\lri$, we derive here a formula for the integral over
$\lri$ by restriction of a formula for the integral over a fixed
totally ramified extension $\Lri$ of $\lri$. In fact, we do not know
whether a general `uniform' formula akin to
Proposition~\ref{pro:unramified} exists for totally ramified
extensions.

Let $\Lri\vert\lri$ be an extension with $f(\Lri,\lri)=1$ and
$e=e(\Lri,\lri)$ arbitrary. As explained in
Section~\ref{subsec:greenberg} there exists $N\in\N$ such that we may
cover $Y(\Lri)$ disjointly by cosets modulo $\mfP^{eN}$, which are
homeomorphic to $(\mfP^{eN})^{(b)}$, so that on each coset the
integrand is `monomial', with local numerical data $(U,
j,(d_{\kappa\iota}))$. Proposition~\ref{pro: bad reduction} yields a
formula for $Z_{W(\Lri),I}(\bfs)$:

$$Z_{\mathcal{W}(\Lri),I}(\bfs) = \frac{(1-q^{-1})^{|I|}}{q^{eNn^2}}
\sum_{\substack{\bfa\in Y(\Lri/\mfP^{eN}) \\ \ol{h_\Lfi}(\bfa)\in
\mathcal{W}(\Lri/\mfP^{eN})}} (q^{eN}- q^{eN-1})^{|U(\bfa)|}
q^{-j(\bfa)} \, \Xi^{eN}_{U(\bfa), (d_{\kappa\iota}(\bfa)),
I}(q,\bfs).$$ Each summand in this sum is associated to an integral
over
$$ \mfP^{(|I|)} \times \{ \bfy \in Y(\Lri) \mid \bfy \equiv \bfa \mod
(\mfP^{eN}) \}.$$ We consider the set $Y(\lri/\mfp^N)$ as a subset of
$Y(\Lri/\mfP^{eN})$ and obtain a formula for the integral
$Z_{\mathcal{W}(\lri),I}(\bfs)$ by `restricting' these integrals to
\begin{equation}\label{equ:restriction}
 \mfp^{(|I|)} \times \{ \bfy \in Y(\lri) \mid \bfy \equiv \bfa \mod
 (\mfp^{N}) \},
\end{equation}
taking into account the different normalisations of the Haar measure
on $(\mfP^{eN})^{(b)}$ and $(\mfp^N)^{(b)}$, respectively. (Recall
that, for a uniformiser $\pi$ of $\lri$, we have $\vert \pi
\vert_{\mfp} = q^{-1}$, whereas $\vert \pi \vert_{\mfP} = q^{-e}$.) We
note that the sets~\eqref{equ:restriction} are non-empty only if $\bfa
\in Y(\lri/\mfp^N)$. In this case, $j(\bfa) \equiv
d_{\kappa\iota}(\bfa) \equiv 0 \mod (e)$, as all the polynomials in
$F_{\kappa\iota}(\bfY)$ and the principalisation $h$ are defined
over~$\gfi$. Thus
$$Z_{\mathcal{W}(\lri),I}(\bfs) = \frac{(1-q^{-1})^{|I|}}{q^{Nn^2}}
\sum_{\substack{\bfa\in Y(\lri/\mfp^N) \\ \ol{h_\lfi}(\bfa)\in
\mathcal{W}(\lri/\mfp^N)}} (q^N - q^{N-1})^{|U(\bfa)|}
q^{-j(\bfa)/e}\; \Xi^{N}_{U(\bfa), d_{\kappa\iota}(\bfa)/e,
I}(q,\bfs).$$

\subsection{Poles of zeta functions of polyhedral cones}\label{subsec:cones}
The various explicit formulae for the functions
$Z_{\mathcal{W}(\Lri),I}(\bfs)$ show that central to a study of the
poles of these zeta integrals is an analysis of the poles of certain
zeta functions of polyhedral cones.

\begin{pro}\label{pro:cones}
Let $q,f,N\in\N$, let $(\delta_{\kappa\iota})_{\kappa\in\{1,\dots,l\},
\iota\in J_\kappa}$ be a family of integers, $\bfa,\bfb\in\Z^l$,
$U\subseteq T$, $I\subseteq\{1,\dots,n-1\}$ and let $L(\bfm,\bfn)$ and
$(L_{\kappa\iota}(\bfm,\bfn))$ be linear forms as defined in
Section~\ref{new subsubsec}. Then the set of real parts of the poles
of
$$\Xi^N_{U,(\delta_{\kappa\iota}),I}(q^f,\bfa
s+\bfb):=\sum_{\bfm\in\N^U_{\geq N},\,\bfn \in
  \N^I}(q^f)^{L(\bfm,\bfn)-\sum_{\kappa=1}^l (a_\kappa s + b_\kappa)
  \min\{L_{\kappa\iota}(\bfm,\bfn)-\delta_{\kappa\iota} \mid \iota \in
  J_\kappa\}}$$ is independent of $q$, $f$, $N$ and
$(\delta_{\kappa\iota})$.
\end{pro}

\begin{proof}
We first observe that
$$\Xi^N_{U,(\delta_{\kappa\iota}),I}(q^f,\bfa s + \bfb) =
q^{\sum_{u}\nu_u(1-N)} \;
\Xi^1_{U,(\delta_{\kappa\iota}+\sum_uN_{u\,\kappa\iota}(N-1)),I}(q^f,\bfa
s + \bfb),$$ so we may assume $N=1$, and drop the superscript.

For $(\delta_{\kappa\iota})=(0)$, recall the decomposition of
$\mathbb{R}_{>0}^{\lvert U \rvert + \lvert I \rvert}$ by the family
$\mathfrak{C}$ of rational polyhedral cones described in
Section~\ref{new subsubsec}. It implies that the denominators of the
generating functions $\Xi_{U,I}(q^f,\bfa s + \bfb)$ divide products of
the form $\prod_{i\in\mfI(U)}(1-(q^f)^{A_i-B_is})$, for some finite
index sets $\mfI(U)$ and nonnegative integers $A_i,B_i$, which only
depend on $L(\bfm,\bfn)$, $(L_{\kappa\iota}(\bfm,\bfn))$, $\bfa$ and
$\bfb$ (but not, for instance, on $q^f$). Note that the fact that the
union of the closed cones $\mcC_g$ in Section~\ref{new subsubsec} is
not disjoint is insubstantial: by construction, if two such cones do
not coincide then they overlap only on components of their boundaries,
and the summation over these boundary components does not introduce
new denominators. The latter follows from the fact that, for each of
the polyhedral cones $\mcC_g$, the denominator of the generating
function associated to $\mcC_g$ may be described in terms of the
primitive vectors on the extremal rays of $\mcC_g$;
cf.~\cite[Proposition~4.6.10, Theorem~4.6.11]{St97}. We set $P :=
\bigcup_{U\subseteq T}\left\{{A_i}/{B_i} \mid i \in \mfI(U) \right\}$.

In the case of general $(\delta_{\kappa\iota})$ we consider the affine
translates
$$ L_{\kappa\iota}'(\bfm,\bfn) := L_{\kappa\iota}(\bfm,\bfn) -
\delta_{\kappa\iota} = \sum_{i\in I}e_{i\,\kappa\iota}n_i + \sum_{u\in
U}N_{u\,\kappa\iota}m_u - \delta_{\kappa\iota}$$ of the linear forms
$L_{\kappa\iota}(\bfm,\bfn)$ defined in
Section~\ref{subsec:explicit_formulae}. For
$g=(g_\kappa)\in\prod_{\kappa}J_\kappa$, let $\mcC'_g$ be the cone in
$\R_{>0}^{\lvert U \rvert + \lvert I \rvert}$ defined by the
condition
$$ \forall \kappa\in\{1,\dots,l\}\,\forall \iota\in
J_\kappa:\,L'_{\kappa g_\kappa}\leq L'_{\kappa\iota}.
$$ Each cone $\mcC'_g$ in the family
$\mathfrak{C}':=(\mcC'_g)_{g\in\prod_\kappa J_\kappa}$ is an affine
translate of a cone in~$\mfC$. In continuation of our discussion in
Section~\ref{new subsubsec}, the restriction of the summation to each
$\mcC'_g\cap\N^{\lvert U \rvert + \lvert I \rvert}$ can be interpreted
in terms of the Hilbert series associated to the solutions of an
\emph{in}homogeneous system $A_g\bfx = \boldsymbol{\beta}$, whilst the
summation over the corresponding cone $\mcC_g\cap\N^{\lvert U \rvert +
  \lvert I \rvert}$ is interpretable in terms of the Hilbert series
associated to the solutions of $A_g\bfx={\bf 0}$. However, it is known
(cf.\ \cite[I.3]{St96}) that the Hilbert series associated to the
inhomogeneous system $A_g\bfx = \boldsymbol{\beta}$ is either $0$ or
else has the same denominator as the Hilbert series associated to
$A_g\bfx=0$. The former clearly does not occur in the current context.
\end{proof}

\begin{rem}
A formula for Igusa local zeta functions for good reduction which is
uniform under field extensions is given in~\cite{DeMe91}. Igusa-type
zeta integrals at places with bad reduction have been studied
in~\cite{dSGr00}. We note that \cite[Proposition 3.3]{dSGr00} seems to
be incorrect; the integrals $\int_{\Zp} \lvert p \rvert_p^{-s}
d\mu(x)$ and $\int_{\{(x,y)\in\Zp^2 \mid x \vert py\}} \lvert xy
\rvert_p^{-s}d\mu(x,y)$ provide counterexamples. These
notwithstanding, \cite[Corollary 3.4]{dSGr00} remains valid, e.g.\ by
arguments akin to the ones given in our proof of
Theorem~\ref{thmABC:poles}.
\end{rem}

\subsection{Proofs of Theorem~\ref{thmABC:funeq} and \ref{thmABC:poles}}\label{subsec:proof_thm_poles}
Theorem~\ref{thmABC:funeq} holds, with $S$ and
$\lambda_1,\dots,\lambda_r$ defined in
Sections~\ref{subsec:explicit_formulae} and~\ref{subsec:good
reduction}, because of \eqref{equ:funeq good reduction} in conjunction
with~\eqref{equ:zeta=integral II}.

The first claim of Theorem~\ref{thmABC:poles} follows, with $P$
defined as in the proof of Proposition~\ref{pro:cones}, from the
formulae we gave for the integrals $Z_{\mathcal{W}(\Lri),I}(\bfs)$ in
Sections \ref{subsec:explicit_formulae} to~\ref{subsec:greenberg},
together with Proposition~\ref{pro:cones}.

For the second claim of Theorem~\ref{thmABC:poles} we note that it
suffices to consider places~$v\not\in S$. The claim follows from the
fact that, for each $U\subseteq T$, the numbers $c_U(q^f)$ are
non-zero for a set of places $v$ of~$\gfi$ of positive density. This
in turn is a consequence of the fact that the varieties $E_U$ are
smooth, quasi-projective varieties defined over $\gfi$ which are
irreducible over $\gfi$. This implies that, for every $U\subseteq T$,
the numbers $b_U(q)$ are non-zero for a set of places of positive
density. In the special case $k=\Q$ this follows essentially from the
Lang-Weil estimate; see e.g.\ \cite[Lemma~4.7]{dSGr00}, which is
easily generalised to arbitrary number fields.

The last claim of Theorem~\ref{thmABC:poles} follows (i) in the case
that $v\not\in S$ from the monotonicity of the numbers $c_U(q^f)$
under residue field extensions and (ii) in the case that $v\in S$ from
the monotonicity of the numbers $c_{U,j,(d_{\kappa\iota})}(q^f)$ in
the case of unramified extensions, and from the discussion in
Section~\ref{subsec:arbitrary extensions} and
Proposition~\ref{pro:cones} in the totally ramified case.

\section{Semisimple groups} \label{sec:semisimple}

In this section we establish a link between the representation zeta
functions of `se\-mi\-sim\-ple' $p$-adic analytic pro-$p$ groups and
stratifications of the associated semisimple Lie algebras by algebraic
varieties defined in terms of centraliser dimension.

Let $\gri$ be the ring of integers of a number field $\gfi$, and
$\Lambda$ an $\gri$-Lie lattice such that $\mathcal{L} := k
\otimes_\smallgri \Lambda$ is a finite dimensional, perfect $\gfi$-Lie
algebra.  Put $d := \dim_\gfi(\mcL)$.  We choose an $\gri$-basis
$\mathbf{b} = (b_1,\ldots,b_d)$ of $\Lambda$, and we define the
commutator matrix $\mathcal{R}(\mathbf{Y})$ of $\Lambda$ with respect
to the basis $\mathbf{b}$ as in \eqref{equ:commutator_matrix}.

For any non-archimedean place $v$ of $\gfi$, the completion $\lri =
\gri_v$ is a compact discrete valuation ring of characteristic $0$,
with maximal ideal $\mfp$ and residue field $\lri/\mfp$ of cardinality
$q$ and characteristic~$p$, say.  The $\lri$-Lie lattice $\mfg := \lri
\otimes_\smallgri \Lambda$ gives rise to a family of FAb $p$-adic
analytic pro-$p$ groups $\mathsf{G}^m(\lri)$, $m \in \N_0$ permissible
for $\mfg$.  Equation~\eqref{equ:Poincare_Igusa_neu} expresses the
relevant Poincar\'e series $\mathcal{P}_{\mathcal{R},\lri}(s)$ in
terms of the generalised Igusa local zeta function
$\mathcal{Z}_\lri(r,t)$.  If $\mathcal{L}$ is semisimple, then the
parameter $\rho$ defined in \eqref{def:rho} has a natural
interpretation in terms of the absolute root system and the absolute
rank of the Lie algebra $\mathcal{L}$.  These are the root system
$\Phi(\mathcal{L}_\C)$ and the rank $\rank(\mathcal{L}_\C)$ of the
complex semisimple Lie algebra $\mcL_\C := \C \otimes_\gfi \mcL$.
Indeed, we have $2 \rho = d - \rank(\mcL_\C) = \lvert
\Phi(\mathcal{L}_\C) \rvert$.

Looking in more detail at the integrand in \eqref{equ:integral_neu},
we observe that the following stra\-ti\-fi\-ca\-tion of
$d$-dimensional affine space $\mathbb{A}^d$ plays a significant role:
$$
\mathbb{A}^d = \mathcal{U}_0 \supseteq \mathcal{U}_1 \supseteq \ldots
\supseteq \mathcal{U}_\rho = \{ 0 \},
$$ where for $i \in \{0,\ldots,\rho\}$ the subvariety $\mathcal{U}_i$
is defined by the set of polynomials $F_{\rho+1-i}(\mathbf{Y})$, over
$\gri$.  We note that for each $i$ in this range
\begin{align*}
  \mathcal{U}_i(\gfi) = & \{ \underline{x} \in \gfi^d \mid \forall f
  \in F_{\rho+1-i}(\mathbf{Y}) : f(\underline{x}) = 0 \} \\ = & \{
  \underline{x} \in \gfi^d \mid \rk_\gfi(\mathcal{R}(\underline{x}))
  \leq d - \rank(\mcL_\C) - 2i \}.
\end{align*}

We now make the assumption that the Lie algebra $\mcL$ is semisimple.
It is remarkable that in this important case the stratification above
can be given a natural algebraic interpretation, as we shall now
explain.  (Our description makes use of properties of the Killing
form, namely that it is non-degenerate and invariant; more generally,
it applies to any Lie lattice such that $\mcL$ admits some
non-degenerate invariant symmetric bilinear form.)

By a choice of co-ordinates, we identify the $\gfi$-vector spaces
$\mcL$ and and its dual $\mcL^\vee$ with the $\gfi$-points of
$d$-dimensional affine space $\mathbb{A}^d$.  Indeed, the $\gri$-basis
$\mathbf{b}$ of $\mcL$ and the dual $\gri$-basis $\mathbf{b}^\vee$ of
$\mcL^\vee$ induce co-ordinate maps
\begin{align*}
 \mcL \xrightarrow{\, \simeq \,} \gfi^d, & \quad x =
   \sum\nolimits_{i=1}^d x_i b_i \mapsto \underline{x} =
   (x_1,\ldots,x_d), \\ \mcL^\vee \xrightarrow{\, \simeq \,} \gfi^d, &
   \quad x = \sum\nolimits_{i=1}^d x_i b_i^\vee \mapsto \underline{x}
   = (x_1,\ldots,x_d),
\end{align*}
defined over $\gri$.  The image $\omega_{\underline{x}} = \sum_{i=1}^d
x_i b_i^\vee$ of $\underline{x}$ under the inverse of the second
isomorphism is uniquely defined by the requirements
$$
\omega_{\underline{x}}(b_i) = x_i \quad \text{for $i \in \{1,\ldots,d\}$.}
$$ This allows us to interpret the stratification of $\mathbb{A}^d$ in
terms of $\mathcal{U}_i$, $i \in \{0,\dots,\rho\}$, by a
stratification $\mathcal{W}_i$, $i \in \{0,\dots,\rho\}$, of
$\mathbb{A}^d$ which is also defined over $\gri$ and satisfies
$$
\mathcal{W}_i(\gfi) = \{ \omega_{\underline{x}} \in \mcL^\vee \mid
\rk_\gfi(\mathcal{R}(\underline{x})) \leq d - \rank(\mcL_\C) - 2i \}.
$$

As $\mcL$ is semisimple, there is also a canonical vector space
isomorphism from $\mcL$ to its dual $\mcL^\vee$ via the Killing form.
It is well-known that the Killing form
$$
\kappa: \mcL \times \mcL \longrightarrow k, \quad
\kappa(x,y) := \Tr(\ad(x) \ad(y))
$$
of the semisimple Lie algebra $\mcL$ is non-degenerate and thus
induces a natural isomorphism of vector spaces
$$
\iota: \mcL \longrightarrow \mcL^\vee, \quad x \mapsto
\kappa(x,\cdot).
$$
We note that, while $\iota$ is defined over $\gri$, its inverse is
typically not defined over $\gri$.  This issue will be discussed
more closely below.

In any case, we claim that $\mathcal{W}_i$, $i \in \{0,\dots,\rho\}$, is
the image under $\iota$ of the stratification
$$
\mathbb{A}^d = \mathcal{V}_0 \supseteq \mathcal{V}_1 \supseteq \ldots
\supseteq \mathcal{V}_\rho = \{ 0 \},
$$
where
$$
\mathcal{V}_i(\gfi) = \{ x \in \mcL \mid \dim_\gfi \Cen_{\mcL}(x) \geq
\rank(\mcL_\C) + 2i \}.
$$

The non-degeneracy and invariance of the Killing form $\kappa$ show
that for $x \in \mcL$ the radical $\Rad(\kappa(x,\cdot))$ is equal to
the centraliser $\Cen_\mcL(x)$.  Indeed, we have
\begin{align*}
  \Rad(\kappa(x,\cdot)) & = \{ y \in \mcL \mid \forall z \in
  \mcL : \kappa(x,[y,z]) = 0 \} \\
  & = \{ y \in \mcL \mid \forall z \in \mcL :
  \kappa([x,y],z) =0 \} \\
  & = \{ y \in \mcL \mid [x,y] = 0 \} \\
  & = \Cen_\mcL(x).
\end{align*}
This implies that, for $i \in \{0,\ldots,\rho\}$,
\begin{equation} \label{equ:zweiter_streich}
  \begin{split}
    \iota \mathcal{V}_i(\gfi) & = \{ \kappa(x,\cdot) \in \mcL^\vee
    \mid \dim_\gfi \Rad(\kappa(x,\cdot)) \geq \rank(\mcL_\C) + 2i \}
    \\ & = \{ \omega_{\underline{x}} \in \mcL^\vee \mid \dim_\gfi
    \Rad(\omega_{\underline{x}}) \geq \rank(\mcL_\C) + 2i \}.
  \end{split}
\end{equation}
Now let $\omega_{\underline{x}} \in \mcL^\vee$ with co-ordinates
$\underline{x}$.  Note that for $y, z \in \mcL$,
$$
\omega_{\underline{x}}([y,z]) = \underline{y} \cdot \mathcal{R}(\underline{x})
\cdot \underline{z}^{\textup{t}}.
$$ Therefore the radical $\Rad(\omega_{\underline{x}}) = \{ y \in \mcL
\mid \forall z \in \mcL : \omega_{\underline{x}}([y,z]) = 0 \}$ maps
under the co-ordinate isomorphism $\mcL \simeq \gfi^d$, $y \mapsto
\underline{y}$ onto $\ker \mathcal{R}(\underline{x})$.  From the fact
that $d = \dim_\gfi \Rad(\omega_{\underline{x}}) + \rk_\gfi
(\mathcal{R}(\underline{x}))$ we deduce that, for $i \in
\{0,\ldots,\rho\}$,
\begin{equation} \label{equ:erster_streich}
 \mathcal{W}_i(\gfi) = \{ \omega_{\underline{x}} \in \mcL^\vee \mid
  \dim_\gfi \Rad(\omega_{\underline{x}}) \geq \rank(\mcL_\C) + 2i \}.
\end{equation}

From \eqref{equ:zweiter_streich} and \eqref{equ:erster_streich} we
deduce by extension of scalars that the stratification
$\mathcal{V}_i$, $i \in \{0,\dots,\rho\}$, is mapped to $\mathcal{W}_i$,
$i \in \{0,\dots,\rho\}$, under the linear isomorphism $\iota$.

In practice, the translation between the varieties needs to be carried
out locally, over the ring $\lri = \gri_v$.  Instead of the Killing
form we use -- in the case of a simple Lie algebra of Chevalley type
-- the normalised Killing form $\kappa_0$ which is related to the
ordinary Killing form $\kappa$ by the equation $2 h^\vee \kappa_0 =
\kappa$ where $h^\vee$ is the dual Coxeter number.  For instance, for
$\spl_n$ the dual Coxeter number is $h^\vee = n$.

A good understanding of the representation zeta functions of
`semisimple' $p$-adic analytic pro-$p$ groups would depend on detailed
knowledge of algebro-geometric and arithmetic properties of the
projective subvarieties of $\mathbb{P}(\mathcal{L})\cong
\mathbb{P}^{d-1}(\C)$ defined by the affine varieties $\mathcal{V}_i$,
$i \in \{0,\ldots,\rho\}$. Indeed, the discussion in
Section~\ref{sec:funeq} shows that explicit formulae for the relevant
zeta functions may be obtained from a principalisation of ideals
defining the terms of this filtration. The theory of sheets may be
relevant in this context; cf.~\cite{Bo81}.  In the first place we
would like to understand the geometry of this stratification over $\C$
(or any algebraically closed field above $\gfi$), and in particular
its singular loci. In a second step one would need to describe the
varieties over completions $\lri=\gri_v$ of $\gri$.  We put forward
the following concrete question.

\begin{qun}
  Let $\mcL_\C$ be a complex semisimple Lie algebra of dimension $d$,
  representing $\mathbb{A}^d(\C)$.  Consider the stratification $
  \mathbb{A}^d = \mathcal{V}_0 \supseteq \mathcal{V}_1 \supseteq
  \ldots \supseteq \mathcal{V}_\rho = \{ 0 \}$, where $2 \rho = d -
  \rank(\mcL_\C) = \lvert \Phi(\mathcal{L}_\C) \rvert$ and
  $$
  \mathcal{V}_i(\C) = \{ x \in \mcL_\C \mid \dim_\C \Cen_{\mcL_\C}(x)
  \geq \rank(\mcL_\C) + 2i \} \qquad \text{for $i \in \{0,\dots,\rho\}.$}
  $$
  Is it true that the singular locus of any non-smooth term
  $\mathcal{V}_i$ in this stratification is given by the first
  successor term $\mathcal{V}_j$, $j>i$, which is properly contained
  in $\mathcal{V}_i$?
\end{qun}

Already for $\mathcal{L}_\C=\spl_4(\C)$, the singularities of the
varieties $\mathcal{V}_i$ are challenging. The case of Lie algebras of
type $A_2$, however, is understood.

\begin{exm}\label{exa:sl3}
 Let $\mathcal{L}_\C=\spl_3(\C)$. Hence $d=8$, and it is not hard to
 verify that~$\rho=3$. Consider the map
 \begin{equation*}
  \beta_\C: \gl_3(\C) \rightarrow \spl_3(\C), \quad X \mapsto X-
  \frac{\Tr(X)}{3} \Id _3,
 \end{equation*}
 and set $\mathcal{D}_1(\C):= \{ X \in \gl_3(\C) \mid \rk_\C(X) \leq
 1\}$.  Regarding the filtration $\mathbb{A}^8 = \mathcal{V}_0
 \supseteq \mathcal{V}_1 \supseteq \mathcal{V}_2 \supseteq
 \mathcal{V}_3 = \{ 0 \}$ it is known that $\mathcal{V}_2(\C) = \{0\}$
 and that $\mathcal{V}_1(\C)=\beta_\C(\mathcal{D}_1(\C))$;
 cf.\ \cite[Example 9.6]{Br98}. The variety $\mathcal{V}_1$ is a
 $5$-dimensional subvariety of $\mathbb{A}^8$, defined over $\Z$,
 which is smooth away from the origin.  The projective variety
 $\mathbb{P}(\mathcal{V}_1 \setminus \{0\})$ is isomorphic to
 $\mathbb{P}^2\times \mathbb{P}^2$.
\end{exm}

The straightforward geometric setup of Example~\ref{exa:sl3} holds the
key to our analysis of Lie algebras of type $A_2$ in
Section~\ref{sec:explicit_cong_sub}, enabling us to compute explicit
formulae for the corresponding zeta functions.


\part{Applications}


\section{Formulae for principal congruence subgroups of $\SL_3(\lri)$ and
  $\SU_3(\Lri,\lri)$}\label{sec:explicit_cong_sub}

In this section we use the methods developed in
Part~\ref{part:formalism} to prove Theorem~\ref{thmABC:SL3}. We first
fix some notation, mainly for the unitary case.  Let $\lri$ denote a
compact discrete valuation ring of characteristic $0$ with field of
fractions $\lfi$, and let $\Lfi$ be an unramified quadratic extension
of~$\lfi$, with ring of integers~$\Lri$.  We write $\mfp$ and $\mfP$
for the maximal ideals of $\lri$ and $\Lri$; the residue field
characteristic and cardinality of $\lri$ are denoted by $p$ and $q$.
We write $\F_q$ for the residue field $\lri/\mfp$.

Let $\sigma: \Lfi \rightarrow \Lfi$ denote the non-trivial Galois
automorphism of $\Lfi \vert \lfi$.  For $n \in \N$, consider on $V =
\Lfi^n$ the standard non-degenerate Hermitian sesquilinear form $V
\times V \rightarrow \Lfi$, $(\mathbf{v},\mathbf{w}) \mapsto
\sum_{i=1}^n v_i^\sigma w_i$.  The structure matrix of this form with
respect to the standard basis is simply the identity matrix.  The
standard involution ${}^\circ$ on the matrix algebra $\Mat_n(\Lfi)$ is
given by $ \bfx^\circ := (\bfx^\sigma)^\textup{t}$, i.e.\ conjugate
transpose.  The standard unitary group $\GU_n(\Lfi,\lfi)$ and its
subgroups, e.g.\ the standard special unitary group
$\SU_n(\Lfi,\lfi)$, can be realised as subgroups of $\GL_n(\Lfi)$ as
follows:
$$ \GU_n(\Lfi,\lfi) = \{ g \in \GL_n(\Lfi) \mid g^\circ \, g = 1 \}.
$$ Similarly, the standard unitary Lie algebra $\gu_n(\Lfi,\lfi)$ and
its subalgebras, e.g.\ the standard special unitary Lie algebra
$\su_n(\Lfi,\lfi)$, can be realised as subalgebras of the $\lfi$-Lie
algebra $\gl_n(\Lfi)$:
$$
\gu_n(\Lfi,\lfi) = \{ \mathbf{x} \in \gu_n(\Lfi) \mid \mathbf{x}^\circ
+ \mathbf{x} = 0 \}.
$$
We define the compact $p$-adic analytic group
\begin{equation}\label{def:SU3}
 \SU_n(\Lri,\lri) := \SU_n(\Lfi,\lfi) \cap \GL_n(\Lri)
\end{equation}
and the $\lri$-Lie lattice
$$
\su_n(\Lri,\lri) := \su_n(\Lfi,\lfi) \cap \gl_n(\Lri).
$$

Throughout the remainder of this section, we will assume that the
residue field characteristic of $\lri$ is not equal to $3$, i.e.\ that
$p \not = 3$.  Under this restriction, we compute explicit formulae
for the representation zeta functions of the principal congruence
subgroups $\SL_3^m(\lri)$, where $m$ is permissible for
$\spl_3(\lri)$, in Section~\ref{subsec:sl3princ} and of the principal
congruence subgroups $\SU_3^m(\Lri,\lri)$, where $m$ is permissible
for $\su_3^m(\Lri,\lri)$, in Section~\ref{subsec:su3princ}.
Calculations of zeta functions associated to principle congruence
subgroups of $\SL_3(\lri)$, where $\lri$ is an unramified extension of
$\Z_3$, can be found in~\cite{AvKlOnVo_pre_b}.

\subsection{Principal congruence subgroups of
  $\SL_3(\lri)$} \label{subsec:sl3princ} In this section we prove the
first part of Theorem~\ref{thmABC:SL3}: we compute explicit formulae
for the representation zeta functions of the principal congruence
subgroups $\SL_3^m(\lri)$, where $m$ is permissible for
$\spl_3(\lri)$.  For this we compute the
integral~\eqref{equ:integral_neu} over $\mfp \times W(\lri)$, where
$W(\lri) = \big( \lri^8 \big)^* \cong \Hom_{\lri}(\spl_3(\lri),\lri)^*
$.  The explicit computation of this integral is feasible because of
two facts.  Firstly, only the family $F_3(\mathbf{Y})$ gives a
non-trivial contribution in the integrand and, secondly, the
projective subvariety of $\mathbb{P}^7$ determined by the ideal
generated by $F_3(\mathbf{Y})$ is smooth.  To be concrete, one may
form the commutator matrix $\mathcal{R}(\mathbf{Y})$ and the sets of
polynomials $F_j(\mathbf{Y})$, $j \in \{0,\dots,\rho\}$, with respect
to the $\lri$-basis $\bfb$ for $\spl_3(\lri)$ comprising the elements
\begin{gather}
  \mathbf{h}_{12} = \left(
  \begin{smallmatrix}
    1 & & \\
    & -1 & \\
    & & 0
  \end{smallmatrix} \right), \quad \mathbf{h}_{23} = \left(
  \begin{smallmatrix}
    0 & & \\
    & 1 & \\
    & & -1
  \end{smallmatrix} \right), \nonumber \\
  \mathbf{e}_{12} = \left(
  \begin{smallmatrix}
    0 & 1 & \\
    & 0 & \\
    & & 0
  \end{smallmatrix} \right), \quad \mathbf{e}_{23} = \left(
  \begin{smallmatrix}
    0 &  & \\
    & 0 & 1 \\
    & & 0
  \end{smallmatrix} \right), \quad \mathbf{e}_{13} = \left(
  \begin{smallmatrix}
    0 & 0 & 1\\
    & 0 & 0 \\
    & & 0
  \end{smallmatrix} \right), \nonumber\\
  \mathbf{f}_{21} = \left(
  \begin{smallmatrix}
    0 & & \\
    1 & 0 & \\
    & & 0
  \end{smallmatrix} \right), \quad
  \mathbf{f}_{23} = \left(
  \begin{smallmatrix}
    0 &  & \\
    & 0 & \\
    & 1 & 0
  \end{smallmatrix} \right), \quad
  \mathbf{f}_{13} = \left(
  \begin{smallmatrix}
    0 & & \\
    0 & 0 & \\
    1 & 0 & 0
  \end{smallmatrix} \right). \nonumber
\end{gather}
In fact, the argument we carry out below is
co-ordinate free.

Recall the definition~\eqref{def:rho} of the parameter $\rho$. It
takes a short computation to verify that, in the current context,
$\rho = 3$ and that $\max \{ \lvert f(\mathbf{y}) \rvert_\mfp \mid f
\in F_2(\mathbf{Y})\} = 1$ for all $\mathbf{y} \in W(\lri) = \big(
\lri^{8}\big)^*$; cf.\ Example~\ref{exa:sl3}.  It therefore suffices
to compute the integral
\begin{equation}\label{equ:sl3 integral}
  \mathcal{Z}_\lri(r,t) = \int_{(x,\mathbf{y}) \in \mfp \times W(\lri)}
  \lvert x \rvert_\mfp^t \lVert F_3(\mathbf{y}) \cup  \{x^2\}
  \rVert_\mfp^r \,  d\mu(x,\mathbf{y}).
\end{equation}
In accordance with Section~\ref{sec:semisimple}, the affine variety
$\mathcal{W}_1$ defined by $F_3(\mathbf{Y})$ corresponds to the
subvariety $\mathcal{V}_1$ of $\spl_3$ whose $\lfi$-points coincide
with the irregular elements in the $8$-dimensional Lie algebra
$\spl_3(\lfi)$.  The translation is carried out by means of the
normalised Killing form $\kappa_0(\cdot,\cdot) = 6^{-1}
\kappa(\cdot,\cdot)$ which can be presented, with respect to the basis
$\bfb$, by the non-degenerate symmetric matrix
$$
[\kappa_0(\cdot,\cdot)]_\bfb =
\left(
\begin{smallmatrix}
  2 & -1 & & & & & & \\
  -1 & 2 & & & & & & \\
  & & & & & 1 & & \\
  & & & & & & 1 & \\
  & & & & & & & 1 \\
  & & 1 & & & & & \\
  & & & 1 & & & & \\
  & & & & 1 & & &
\end{smallmatrix}
\right).
$$
Since $p \not = 3$, we conclude that $\kappa_0$ induces an isomorphism
of $\lri$-modules between $\spl_3(\lri)^*$ and
$\Hom_{\lri}(\spl_3(\lri),\lri)^*$.

It is known (cf. Example~\ref{exa:sl3}) that the irregular locus
$\mathcal{V}_1$ of $\spl_3$ may be realised as the image of the
variety $\mathcal{D}_1$ of matrices of rank at most~$1$ under the
projection
$$
\beta: \mathfrak{gl}_3 \rightarrow \spl_3,
\quad \mathbf{x} \mapsto \mathbf{x} - \frac{\Tr(\mathbf{x})}{3} \Id_3.
$$ In the language of sheets (cf.\ \cite{Bo81}), the image
$\beta(\mathcal{D}_1)$ is the union of the sheets associated to the
partitions $(1,1,1)$ and $(2,1)$ of~$3$, that is the null-sheet and
the unique subregular sheet.  One verifies that $\beta
\vert_{\mathcal{D}_1}$ is injective, and $\beta$ induces a morphism
$\beta_\lri : \mathfrak{gl}_3(\lri) \rightarrow \spl_3(\lri)$ whose
restriction to $\mathcal{D}_1(\lri)$ has good reduction modulo~$\mfp$.
Moreover, one easily checks that $\beta_\mfp$ maps
$\mathcal{D}_1(\lri)$ onto $\mathcal{V}_1(\lri)$.  Thus
$\mathcal{V}_1(\lri) = \beta_\mfp(\mathcal{D}_1(\lri))$ is an affine
cone which is smooth away from the origin.  The number of
$\F_q$-rational points of~$\overline{\mathcal{V}_1}$, viz.\ the
reduction of $\mathcal{V}_1$ modulo $\mfp$, is
\begin{equation}\label{equ:V1_sl3}
\lvert \overline{\mathcal{V}_1}(\F_q) \rvert = \lvert
\overline{\mathcal{D}_1}(\F_q) \rvert = \lvert (\mathbb{P}^2 \times
\mathbb{P}^2)(\F_q) \rvert (q-1) + 1 = (q^2 + q + 1)^2 (q-1) + 1.
\end{equation}

To evaluate the integral~\eqref{equ:sl3 integral}, we proceed in the
spirit of the proof of~\cite[Theorem~3.1]{De87}.  The factor $W(\lri)$
in the domain of integration may be written as a disjoint union of
domains of fixed residue modulo~$\mfp$: for each $\mathbf{a} \in \big(
\F_q^{8} \big)^*$, we set $W_\mathbf{a}(\lri) := \{ \mathbf{y} \in
W(\lri) \mid \mathbf{y} \equiv_\mfp \mathbf{a} \}$ and
\begin{align*}
  \mathcal{Z}_{\lri,\mathbf{a}}(r,t) & := \int_{(x,\mathbf{y}) \in
    \mfp \times W_\mathbf{a}(\lri)} \lvert x \rvert_\mfp^t \lVert
  F_3(\mathbf{y}) \cup \{x^2\} \rVert_\mfp^r \, d\mu(x,\mathbf{y}) \\
  & = \int_{(x,\widetilde{\mathbf{y}}) \in \mfp^{(9)}} \lvert x
  \rvert_\mfp^t \lVert F_3(\mathbf{a} + \widetilde{\mathbf{y}}) \cup
  \{x^2\} \rVert_\mfp^r \, d\mu(x,\widetilde{\mathbf{y}})
\end{align*}
so that $W(\lri) = \bigdotcup_{\mathbf{a} \in (\F_q^{8})^*}
W_\mathbf{a}(\lri)$ implies
$$
\mathcal{Z}_\lri(r,t) = \sum_{\mathbf{a} \in (\F_q^{8})^*}
\mathcal{Z}_{\lri,\mathbf{a}}(r,t).
$$

We call a point $\mathbf{a} \in \big( \F_q^8 \big)^*$ and any
$\mathbf{y} \in W_\mathbf{a}(\lri)$ \emph{regular} if $\mathbf{a}$ is
not an $\F_q$-rational point of~$\overline{\mathcal{V}_1}$.  A
functional $w \in \Hom_{\lri}(\spl_3(\lri),\lri)^*$ and the
representations associated to the Kirillov orbits of the images of $w$
in $\Hom_\lri(\spl_3(\lri),\lri/\mfp^n)^*$, $n \in \N$, are said to be
\emph{regular} if the co-ordinate vector $\mathbf{y} \in W(\lri)$
corresponding to $w$ is regular.  The computation of
$\mathcal{Z}_{\lri,\mathbf{a}}(r,t)$ is particularly straightforward
if $\mathbf{a} \in \big( \F_q^{8} \big)^*$ is regular: in this case
$\mathcal{Z}_{\lri,\mathbf{a}}(r,t) = \mathcal{Z}_\lri^{[0]}(r,t),$
where
$$ \mathcal{Z}_\lri^{[0]}(r,t) := \int_{(x,\widetilde{\mathbf{y}}) \in
  \mfp^{(9)} } \lvert x \rvert_\mfp^t \,
d\mu(x,\widetilde{\mathbf{y}}) = \frac{q^{-9-t}(1 - q^{-1})}{1 -
  q^{-1-t}}.
$$

Next we consider the more complex situation where $\mathbf{a} \in
\big( \F_q^{8} \big)^*$ is an $\F_q$-rational point
of~$\overline{\mathcal{V}_1}$.  In this case the point $\mathbf{a}$,
any co-ordinate vector $\mathbf{y} \in W_\mathbf{a}(\lri)$, the
functionals $w \in \Hom_{\lri}(\spl_3(\lri),\lri)^*$
represented by such $\mathbf{y}$ and the representations associated to
the Kirillov orbits of the images of such $w$ in
$\Hom_\lri(\spl_3(\lri),\lri/\mfp^n)^*$, $n \in \N$, are said to be
\emph{irregular}.

We observe that by a suitable transformation of the co-ordinate
functions $\mathbf{y}$ we may simplify the integral defining
$\mathcal{Z}_{\lri,\mathbf{a}}(r,t)$ significantly.  Indeed, recall
that, by Remark~\ref{rem:minors}, we may
replace $F_3(\mathbf{Y})$ by the set of principal $6\times6$-minors of
the commutator matrix~$\mcR(\mathbf{Y})$. These minors are
squares. Near a smooth point (in the case under consideration any
point away from the origin; cf.\ Example~\ref{exa:sl3}) we may thus
replace the set of functions $F_3(\mathbf{Y})$ defining the
$5$-dimensional variety $\mathcal{V}_1$ by the simpler set $\{
y_1^2,y_2^2,y_3^2 \}$ consisting of the squares of the first three
co-ordinate functions, say.  This transformation yields that
$\mathcal{Z}_{\lri,\mathbf{a}}(r,t) = \mathcal{Z}_\lri^{[1]}(r,t)$,
where
\begin{align*}
  \mathcal{Z}_\lri^{[1]}(r,t) & := \int_{(x,\widetilde{\mathbf{y}})
    \in \mfp^{(9)}} \lvert x \rvert_\mfp^t \lVert \{\widetilde{y}_1,
  \widetilde{y}_2, \widetilde{y}_3,x\} \rVert_\mfp^{2r} \,
  d\mu(x,\widetilde{\mathbf{y}}) \\ & = \sum_{(l,n) \in \N^2} (1 -
  q^{-1}) q^{-n} M_l \, q^{-nt - 2 \min \{l,n\} r}
\end{align*}
with
$$ M_l = \mu \left( \left\{ \widetilde{\mathbf{y}} \in \mfp^{(8)} \mid
    \max \{ \lvert \widetilde{y}_1 \rvert_\mfp, \lvert \widetilde{y}_2
    \rvert_\mfp, \lvert \widetilde{y}_3 \rvert_\mfp \} = q^{-l}
    \right\} \right) = (1 - q^{-3}) q^{-3l-5}.
$$
This gives
$$ \mathcal{Z}_\lri^{[1]}(r,t) = (1 - q^{-1}) (1 - q^{-3}) q^{-5}
\sum_{(l,n) \in \N^2} q^{(-1-t) n - 3l - 2 r \min\{l,n\}}.
$$
Using the fact that
\begin{equation}\label{equ:geometric_series}
\sum_{(l,n) \in \N^2} X_1^l X_2^n X_3^{\min\{l,n\}} = \frac{X_1 X_2
  X_3 (1 - X_1 X_2)}{(1 - X_1 X_2 X_3)(1 - X_1)(1 - X_2)}
\end{equation}
we obtain further that
$$ \mathcal{Z}_\lri^{[1]}(r,t) = \frac{q^{-9-2r-t} (1 - q^{-4-t})(1 -
  q^{-1})}{(1 - q^{-4-2r-t})(1 - q^{-1-t})}.
$$

By \eqref{equ:V1_sl3} we have
\begin{align*}
 \mathcal{Z}_\lri(r,t) &= (q^8 - 1 - (|\ol{\mathcal{V}_1}(\Fp)|-1))
 \mathcal{Z}_\lri^{[0]}(r,t) + (|\ol{\mathcal{V}_1}(\Fp)|-1)
 \mathcal{Z}_\lri^{[1]}(r,t)\\ &= (q^8 - 1 - (q^2+q+1)^2 (q-1))
 \mathcal{Z}_\lri^{[0]}(r,t) + (q^2+q+1)^2 (q-1)
 \mathcal{Z}_\lri^{[1]}(r,t),
\end{align*}
and a straightforward computation, together with
equation~\eqref{equ:Poincare_Igusa_neu}, reveals that
\begin{multline*}
  \mathcal{P}_{\mathcal{R},\lri} (s+2) = 1 + (1-q^{-1})^{-1}
  \mathcal{Z}_\lri(-s/2-1,3s-3) \\ = \frac{q^5 + (-q-q^2-q^3+q^4+q^5)
    q^{-2s} + (1+q-q^2-q^3-q^4) q^{-3s} + q^{-5s}}{q^5(1 - q^{1-2s})(1
    - q^{2-3s})}.
\end{multline*}
Multiplying the last expression by $q^{8m}$ we obtain, by
Proposition~\ref{pro:zeta=poincare}, an explicit formula for
$\zeta_{\SL_3^m(\lri)}(s)$, as stated in Theorem~\ref{thmABC:SL3}.
Note that in this particular instance the functional equation
established in Theorem~\ref{thmABC:funeq} follows from the fact that
$$
\mathcal{P}_{\mathcal{R},\lri} (s+2) \vert_{q \rightarrow
  q^{-1}} = q^8 \mathcal{P}_{\mathcal{R},\lri} (s+2),
$$ which can be easily verified directly.  The explicit formulae for
the zeta functions of groups of the form $\SL^m_3(\lri)$, $\lri$ an
unramified extension of $\Z_3$, which are provided
in~\cite{AvKlOnVo_pre_b} show that such functional equations are not
satisfied for these groups.

Anticipating our computations in Section~\ref{subsec:abscissa_inner}, we
record an alternative formula for $\zeta_{\SL_3^1(\lri)}(s)$, in the
case where $m=1$ is permissible for $\spl_3(\lri)$.  Recalling the
notion of regular and irregular representations introduced above, we
have
\begin{equation} \label{eq:starting_point} \zeta_{\SL_3^1(\lri)} (s) =
  1 + \zeta_{\SL_3^1(\lri)}^{\textup{reg}}(s) +
  \zeta_{\SL_3^1(\lri)}^{\textup{irreg}}(s),
\end{equation}
where the three summands $1$, $\zeta_{\SL_3^1(\lri)}^{\textup{reg}}(s)$
and $\zeta_{\SL_3^1(\lri)}^{\textup{irreg}}(s)$ enumerate the trivial,
the regular and the irregular representations of $\SL_3^1(\lri)$
respectively.  Our computations above yield the following formulae for
these summands:
\begin{equation}\label{equ:series-factor}
\begin{split}
  \zeta_{\SL_3^1(\lri)}^{\textup{reg}}(s) & = \left( q^8 - 1 -
    (q^2+q+1)^2(q-1) \right) \left( 1 + q^8 (1-q^{-1})^{-1}
    \mathcal{Z}_\lri^{[0]}(-(s+2)/2,3s-3) \right) \\
  & = \left( q^8 - 1 - (q^2+q+1)^2(q-1) \right) \; \frac{1}{1- q^{2-3s}}, \\
  \zeta_{\SL_3^1(\lri)}^{\textup{irreg}}(s) & = (q^2+q+1)^2(q-1) \;
  \left( 1 + q^8 (1-q^{-1})^{-1} \mathcal{Z}_\lri^{[1]}(-(s+2)/2,3s-3)
  \right) \\
  & = (q^2+q+1)^2(q-1) \; \frac{1 - q^{1-2s} - q^{2-3s} +
    q^{4-2s}}{(1 - q^{1-2s})(1 - q^{2-3s})}.
\end{split}
\end{equation}
In Section~\ref{subsec:abscissa_inner} the right-most factors on the
right hand side will be referred to as series factors.


\subsection{Principal congruence subgroups of $\SU_3(\Lri,\lri)$}\label{subsec:su3princ}

Recall that $\Lri \vert \lri$ denotes an unramified quadratic
extension with non-trivial automorphism $\sigma$, and that we assume
$p\not=3$.  In this section we prove the second part of
Theorem~\ref{thmABC:SL3}: we compute explicit formulae for the
representation zeta functions of the principal congruence subgroups
$\SU_3^m(\Lri,\lri)$, where $m$ is permissible for
$\su_3(\Lri,\lri)$. We need to compute the
integral~\eqref{equ:integral_neu} over $\mfp \times W(\lri)$, where
$W(\lri) = \big( \lri^8 \big)^* \cong
\Hom_{\lri}(\su_3(\Lri,\lri),\lri)^*$.  We make use of our
computations for the principal congruence subgroups of $\SL_3(\Lri)$
in Section~\ref{subsec:sl3princ}.  
Our argument is based on the commutative diagram
$$
\begin{CD}
  \gl_3(\Lfi) \cong \Lfi \otimes_\lfi \gu_3(\Lfi,\lfi)
  @>{\widetilde{\beta}}>> \spl_3(\Lfi)
  \cong \Lfi \otimes_\lfi \su_3(\Lfi,\lfi) \\
  @AA{\text{inclusion}}A @AA{\text{inclusion}}A \\
  \gu_3(\Lfi,\lfi) @>{\beta}>> \su_3(\Lfi,\lfi)
\end{CD}
$$ where the map $\widetilde{\beta}$ and its restriction $\beta$ map
an element $\mathbf{x}$ to $\widetilde{\beta}(\mathbf{x}) = \mathbf{x}
- \Tr(\mathbf{x})/3 \cdot \Id_3$.  Note that the trace of an element
$\mathbf{x} \in \gu_3(\Lfi,\lfi)$ lies in $\mfgu_1(\Lfi,\lfi)$ so that
$\beta(\mathbf{x})$ lies in $\su_3(\Lfi,\lfi)$, as indicated.  Since
$\spl_3(\Lfi) \cong \Lfi \otimes_\lfi \su_3(\Lfi,\lfi)$, we have for
every $\mathbf{x} \in \su_3(\Lfi,\lfi)$,
\begin{equation}\label{equ:cent_dim}
  \dim_\lfi \Cen_{\su_3(\Lfi,\lfi)}(\mathbf{x}) = \dim_\Lfi
  \Cen_{\spl_3(\Lfi)}(1 \otimes \mathbf{x}).
\end{equation}

We now fix an $\lri$-basis for $\Lri$. For $p>3$ we may, for instance,
choose $(1,\sqrt{\delta})$, where $\delta\in\lri$ is not a square
modulo $\mfp$. For any $\Lfi$-variety $\mathcal{V}$ let
$\res_{\Lfi\vert\lfi}(\mathcal{V})$ denote the $\lfi$-variety obtained
from $\mathcal{V}$ by restriction of scalars with respect to this
basis, considered as a $\lfi$-basis for $\Lfi$.  Thus the set of
$\Lri$-points $\mathcal{V}(\Lri)$ is in natural correspondence with
the set of $\lri$-points
$\res_{\Lfi\vert\lfi}(\mathcal{V})(\lri)$. Given $\bfx \in
\mathcal{V}(\Lri)$ or $\bfx \in \mathcal{V}(\Lfi)$, we write
$\res_{\Lfi\vert\lfi}(\bfx)$ to denote the corresponding point in
$\res_{\Lfi\vert\lfi}(\mathcal{V})(\lri)$ or
$\res_{\Lfi\vert\lfi}(\mathcal{V})(\lfi)$, respectively.

With this notation, the inclusions $\gu_3(\Lfi,\lfi) \subseteq
\gl_3(\Lfi)$ and $\su_3(\Lfi,\lfi) \subseteq \spl_3(\Lfi)$ admit
natural interpretations at the level of algebraic varieties.  Indeed,
$\gu_3(\Lfi,\lfi)$ can be regarded as a $9$-dimensional $\lfi$-linear
subspace of the set of $\lfi$-points $\res_{\Lfi \vert
\lfi}(\gl_3)(\lfi)$ of the $\lfi$-variety $\res_{\Lfi \vert
\lfi}(\gl_3) \cong \mathbb{A}^{18}$.  We denote, in the sequel, by
$\gu_3(\Lfi,\lfi)$ both the $\lfi$-variety and its $\lfi$-rational
points; similar remarks apply to $\su_3(\Lfi,\lfi)$.

After these preparations, we consider the $\lfi$-variety
$\mathcal{V}_1$ of irregular elements in $\su_3(\Lfi,\lfi)$.  Let
$\widetilde{\mathcal{V}}_1$ denote the $\Lfi$-variety of irregular
elements in the $\Lfi$-variety $\spl_3$.
Equation~\eqref{equ:cent_dim} says that
$$
\mathcal{V}_1(\lfi) =
\res_{\Lfi\vert\lfi}(\widetilde{\mathcal{V}}_1)(\lfi) \cap
\su_3(\Lfi,\lfi) = \{ \mathbf{x} \in \widetilde{\mathcal{V}}_1(\Lfi)
\mid \mathbf{x}^\circ + \mathbf{x} = 0 \}.
$$
We know, from the discussion in Section~\ref{subsec:sl3princ}, that
$\widetilde{\mathcal{V}}_1$ may be realised as the image of the
$\Lfi$-variety $\widetilde{\mathcal{D}}_1$ of matrices of rank at most
$1$ under the projection map $\widetilde{\beta}$ and that the
restricted map $\widetilde{\beta} \vert_{\widetilde{\mcD}_1}$ is
injective.  The next lemma shows that $\mathcal{V}_1(\lfi)$ is the
image under $\beta$ of the $\lfi$-points of the $\lfi$-variety
$$ \mathcal{D}_1 := \res_{\Lfi\vert\lfi}(\widetilde{\mathcal{D}}_1)
\cap \gu_3(\Lfi,\lfi). 
$$

\begin{lem}\label{lem:preimage_unitary}
  Let $\mathbf{x} \in \widetilde{\mcD}_1(\Lfi)$ such that
  $\widetilde{\beta}(\mathbf{x}) \in \su_3(\Lfi,\lfi)$.  Then
  $\mathbf{x} \in \gu_3(\Lfi,\lfi)$.
\end{lem}

\begin{proof}
  We need to show that $\mathbf{x}^\circ = - \mathbf{x}$.  Because
  $\widetilde{\beta}$ is injective on $\widetilde{\mcD}_1$, the claim
  follows from the observation that $\widetilde{\mcD}_1$ is invariant
  under the operation $\bfx \mapsto \bfx^\circ$ and that
  $$
  \widetilde{\beta}(\bfx^\circ) = \widetilde{\beta}(\mathbf{x})^\circ
  = - \widetilde{\beta}(\mathbf{x})^\circ =
  \widetilde{\beta}(-\mathbf{x}^\circ).
  $$
\end{proof}

We may consider $\beta$ and, by restriction of scalars,
$\widetilde{\beta}$ as morphisms of $\lfi$-varieties.  As $p \not =
3$, the morphism
$\widetilde{\beta}:\gl_3(\Lri)\rightarrow\spl_3(\Lri)$ induces a
morphism $\widetilde{\beta}_\mfp: \res_{\Lfi\vert\lfi}(\gl_3)(\lri)
\rightarrow \res_{\Lfi\vert\lfi}(\spl_3)(\lri)$ with good reduction
modulo~$\mfp$.  As we noted in Section~\ref{subsec:sl3princ}, the map
$\widetilde{\beta} \vert_{\widetilde{\mcD}_1(\Lri)} :
\widetilde{\mcD}_1(\Lri) \rightarrow \widetilde{\mcV}_1(\Lri)$ is a
bijection. Using Lemma~\ref{lem:preimage_unitary} one sees that
${\beta}_\mfp \vert_{\mcD_1(\lri)} : \mcD_1(\lri) \rightarrow
\mcV_1(\lri)$ is a bijection.

\begin{pro}\label{pro:mcD_1}
  The $\lfi$-variety ${\mcD}_1 \setminus \{\mathbf{0}\}$ is smooth and
  $5$-di\-men\-sio\-nal.  It has good reduction modulo $\mfp$, which
  has $(q^4+q^2+1)(q-1)$ $\Fq$-rational points.
\end{pro}

\begin{proof}
  We consider the surjective `Segre map' of
  $\Lfi$-varieties
  \begin{equation*}
    \pr : \mbbA^3 \times \mbbA^3 \rightarrow\mcDtilde_1,\quad
    (\bfa,\bfb) \mapsto \bfa^\textup{t} \cdot \bfb
  \end{equation*}
  and note that, for all $\bfa,\bfb\in \mathbb{A}^3(\Lfi)$ with
  $\mathbf{x} := \pr(\bfa,\bfb) \in \mcDtilde_1(\Lfi) \setminus
  \{\mathbf{0}\}$, we have $(\pr^{-1}(\mathbf{x}))(\Lfi) = \{ (\mu\bfa
  , \mu^{-1}\bfb) \mid \mu \in \Lfi^* \} \cong
  (\mathbb{A}^1)^*(\Lfi)$.  Recalling that $\sigma$ denotes the
  non-trivial Galois automorphism of $\Lfi$ over $\lfi$, we first
  establish
  \begin{equation}\label{equ:D1K_neu}
    \mcD_1(\lfi) = \left\{
      \res_{\Lfi\vert\lfi}(\pr)(\bfa,\lambda\bfa^\sigma) \mid \bfa \in
      \mathbb{A}^3(\Lfi) \setminus\{\mathbf{0}\}, \lambda \in
      \mfgu_1(\Lfi,\lfi) \right\}.
  \end{equation}
  The inclusion `$\supseteq$' is clear. To prove the reverse inclusion
  `$\subseteq$', let $\bfx=\bfa^\textup{t}\cdot\bfb\in
  \widetilde{\mcD}_1(\Lfi)\setminus\{\mathbf{0}\}$ such that
  $\res_{\Lfi\vert\lfi}(\mathbf{x})\in\mfgu_3(\Lfi,\lfi)$. Then
  $\bfx^\circ=(\bfb^\sigma)^\textup{t}\cdot\bfa^\sigma$. Since $\bfx$
  and $\bfx^\circ$ have the same row span, the vectors $\bfb$ and
  $\bfa^\sigma$ are proportional, say $\bfb=\lambda\bfa^\sigma$, where
  $\lambda\in\Lfi^*$. The equation $\bfx^\circ+\bfx=0$ implies that
  $\lambda^\sigma+\lambda=0$, i.e. that
  $\lambda\in\mfgu_1(\Lfi,\lfi)^*$. This finishes the proof
  of~\eqref{equ:D1K_neu}.  The same argument shows that
  \begin{equation}\label{equ:D1o}
  \mcD_1(\lri)^* := \{ \bfx \in \mcD_1(\lri) \mid \bfx \not\equiv_\mfp
      \mathbf{0} \} = \left\{
      \res_{\Lfi\vert\lfi}(\pr)(\bfa,\lambda\bfa^\sigma) \mid \bfa \in
      (\Lri^3)^*, \lambda \in \mfgu_1(\Lri,\lri)^* \right\}.
  \end{equation}


  In particular, \eqref{equ:D1o} shows that the number of
  $\F_q$-rational points of the reduction modulo $\mfp$ of the
  pre-image $\left(\res_{\Lfi \vert
  \lfi}(\pr)\right)^{-1}(\mathcal{D}_1(\lri)^*)$ in
  $\res_{\Lfi\vert\lfi}(\mathbb{A}^3 \times \mathbb{A}^3)$ is equal to
  $$
  \lvert (\mathbb{A}^3)^*(\F_{q^2}) \rvert \cdot \lvert
  \mfgu_1(\F_{q^2},\F_q) \setminus \{0\} \rvert = (q^6-1)(q-1).
  $$ Since the fibres of the projection $\pr$ are -- with the
  exception of the fibre above zero -- all isomorphic to
  $(\mbbA^1)^*$, this implies that the number of $\F_q$-rational
  points of the reduction modulo $\mfp$ of $\mcD_1
  \setminus\{\mathbf{0}\}$ is equal to $(q^4+q^2+1)(q-1)$.

  To show that $\mcD_1 \setminus \{\mathbf{0}\}$ is a smooth
  $\lfi$-variety, we show how it can be covered by open charts which
  are each isomorphic to the $\lfi$-variety $\mathbb{A}^4 \times
  (\mathbb{A}^1)^* \cong \res_{\Lfi\vert\lfi}(\mathbb{A}^2) \times
  (\mathbb{A}^1)^*$.  For $i \in \{1,2,3\}$ set
  $$
  \mcD_1^{[i]} := \{ \res_{\Lfi\vert\lfi}(X) \mid X \in
  \widetilde{\mcD}_1 \text{ such that } \res_{\Lfi\vert\lfi}(X) \in
  \mcD \text{ and } X_{ii} \not = 0 \}.
  $$
  Clearly, these three charts cover $\mcD_1 \setminus \{\mathbf{0}\}$.
  It remains to show that each of the charts is isomorphic to
  $\res_{\Lfi\vert\lfi}(\mathbb{A}^2) \times (\mathbb{A}^1)^*$.
  Without loss of generality we consider only $\mcD_1^{[1]}$.  For
  ease of notation we will set up an isomorphism between
  $\lfi$-rational points of varieties, which can easily be extended to
  an isomorphism of varieties.

We first deal with the case that $p\not=2$. Since $\Lri$ is unramified
  over $\lri$, we may write $\Lri = \lri(\sqrt{\delta}) = \lri + \lri
  \sqrt{\delta}$, where $\delta \in \lri$ is not a square
  modulo~$\mfp$. Writing $\xi := \sqrt{\delta}$, we have
  $\mfgu_1(\Lfi,\lfi) \setminus \{0\} = \lfi^*\xi$.  Hence $U :=
  \Lfi^2 \times (\mfgu_1(\Lfi,\lfi) \setminus \{0\}) \cong
  (\res_{\Lfi\vert\lfi}(\mathbb{A}^2) \times (\mathbb{A}^1)^*)(\lfi)$.
  Consider the morphism
  $$
  \phi: U \rightarrow \mcD_1^{[1]}(\lfi), \quad
  \phi((a_2,a_3,\lambda)) \mapsto
  \res_{\Lfi\vert\lfi}(\pr((1,a_2,a_3), \lambda
  (1,a_2^\sigma,a_3^\sigma))).
  $$ Writing $a_2 = a_{21} + a_{22} \xi$, $a_3 = a_{31} + a_{32} \xi$
  and $\lambda = c \xi$, where $a_{21}, \ldots, a_{32}, c \in \lfi$
  with $c \not = 0$, we can describe the effect of $\phi$ implicitly
  in co-ordinates over $\lfi$:
  \begin{equation*}
    \phi((a_2,a_3,\lambda)) = \res_{\Lfi\vert\lfi}(\mathbf{x}) \quad
    \text{where} \quad  \mathbf{x} =
    c \xi
    \begin{pmatrix}
      1 & a_{21}-a_{22}\xi & a_{31}-a_{32}\xi \\
      a_{21}+a_{22}\xi & * & * \\
      a_{31}+a_{32}\xi & * & *
    \end{pmatrix}.
  \end{equation*}
  From this description we can deduce that $\phi$ constitutes the
  desired isomorphism by extracting explicitly the inverse $\phi^{-1}$
  as a morphism over $\lfi$,
  $$
  \phi^{-1} : \mcD_1^{[1]}(\lfi) \rightarrow U, \quad
  \res_{\Lfi\vert\lfi}(\mathbf{x}) \mapsto
  \left(\frac{x_{12,2}}{x_{11,2}} -
    \frac{x_{12,1}}{x_{11,2}\delta}\xi,\frac{x_{13,2}}{x_{11,2}} -
    \frac{x_{13,1}}{x_{11,2}\delta}\xi, x_{11,2} \xi \right)
  $$
  where we write $x_{ij} = x_{ij,1} + x_{ij,2} \xi$.

Suppose now that $\lri$ has residue characteristic $p=2$.  In
accordance with Artin-Schreier theory, we can write $\Lri =
\lri(\omega) = \lri + \lri \omega$, where $\omega = (1 +
\sqrt{\delta})/2$ with $\delta \in 1 + 4 \lri^*$, such that the
reduction of $X^2 - X - (\delta - 1)/4 \in \lri[X]$ modulo $\mfp$ is
irreducible over $\F_q = \lri/\mfp$.  As above, let $\sigma$ denote
the non-trivial Galois automorphism of $\Lfi$ over $\lfi$, where
$\Lfi$ and $\lfi$ denote the fields of fraction of $\Lri$ and $\lri$.
Then $(a + b \omega)^\sigma = (a+b) - b\omega$ for $a,b \in \lfi$, and
hence $\gu_1(\Lfi,\lfi) = \{ a+b\omega \mid 2a+b = 0 \}$.

To show that the chart $\mcD_1^{[1]}$ is isomorphic to
$\res_{\Lfi\vert\lfi}(\mathbb{A}^2) \times (\mathbb{A}^1)^*$, one
defines as before the morphism $\phi: U \rightarrow
\mcD_1^{[1]}(\lfi)$.  Writing $a_2 = a_{21} + a_{22} \omega$, $a_3 =
a_{31} + a_{32} \omega$ and $\lambda = c (1-2\omega)$, where $a_{21},
\ldots, a_{32}, c \in \lfi$ with $c \not = 0$, one describes $\phi$ in
co-ordinates and extracts explicitly the inverse $\phi^{-1}$ as a
morphism over $\lfi$.  For instance, one recovers $a_2 = a_{21} +
a_{22} \omega$ by means of the formulae
$$
a_{22} = \frac{x_{12,1} - x_{21,1}}{x_{11,1} \delta} \qquad \text{and}
\qquad a_{21} = \frac{x_{21,1}}{x_{11,1}} + \frac{(\delta-1)
  a_{22}}{2} ,
$$
where we write $x_{ij} = x_{ij,1} + x_{ij,2} \omega$.
\end{proof}

Proposition~\ref{pro:mcD_1} allows us to compute the representation
zeta functions of the principal congruence subgroups
$\SU_3^m(\Lri,\lri)$, analogously to our work in
Section~\ref{subsec:sl3princ}.  Indeed, for every permissible $m$ we
obtain formulae analogous to \eqref{eq:starting_point} and
\eqref{equ:series-factor}, with $(q^2+q+1)^2$ replaced by $q^4 + q^2 +
1$.
A short computation yields the explicit formulae stated in
Theorem~\ref{thmABC:SL3}.


\section{Abscissae of convergence for arithmetic groups of type $A_2$}
\label{sec:abscissae}

In this section we prove Theorem~\ref{thmABC:lalu}.  Prerequisites
about the structure of algebraic groups of type $A_2$ and their
arithmetic subgroups are collected in
Appendix~\ref{subsec:algebraic/arithmetic_groups}.  Our general
strategy is as follows.  In Section~\ref{subsec:abscissa_inner} we
deal in detail with arithmetic groups of type ${}^1 \! A_2$ (inner
forms) and in Section~\ref{subsec:abscissa_outer} we describe the
necessary modifications for treating groups of type ${}^2 \! A_2$
(outer forms).  We show that, in fact, it suffices to study infinite
Euler products of representation zeta functions associated to compact
$p$-adic analytic groups of the form $\SL_3(\lri)$ and, in the case of
outer forms, also $\SU_3(\Lri,\lri)$.  In order to prove that the
abscissae of convergence of such Euler products are equal to~$1$, we
produce via Clifford theory suitable approximations of the local Euler
factors.  These approximations allow us to control the analytic
properties of the original products.  The relevant group theoretic
tools for the `approximative Clifford theory' are prepared in
Section~\ref{subsec:clifford}.  The approximations of the local Euler
factors are designed to remove the need to decide whether or not
characters on principal congruence subgroups -- which we dealt with in
Section~\ref{sec:explicit_cong_sub} -- are extendable to their
respective inertia groups.  We manufacture the approximative Dirichlet
series according to a finite case distinction, reflecting aspects of
the adjoint actions of the finite groups $\GL_3(\mathbb{F}_q)$ and
$\GU_3(\mathbb{F}_{q^2},\mathbb{F}_{q})$ on the finite Lie algebras
$\mathfrak{sl}_3(\mathbb{F}_q)$ and
$\su_3(\mathbb{F}_{q^2},\mathbb{F}_q)$, respectively.  Relevant data
of these finite actions, including centralisers and the numbers and
sizes of orbits of various types, are collected in Tables~\ref{table1}
and \ref{table2} (for~$\mathfrak{sl}_3(\mathbb{F}_q)$) and
Tables~\ref{table3} and~\ref{table4}
(for~$\su_3(\mathbb{F}_{q^2},\mathbb{F}_q)$).  The tables are derived
in Appendices~\ref{sec:aux_sl3} and \ref{sec:aux_su3}, respectively.
Section~\ref{subsec:dirichlet} contains a few elementary facts and
definitions which we need for our approximations of Dirichlet series.

\begin{rem}\label{rem:other_paper}
  In \cite{AvKlOnVo_pre_a}, we give explicit formulae for the
  representation zeta functions of groups of the form
  $\SL_3(\lri)$. This allows for an alternative proof of
  Theorem~\ref{thmABC:lalu} for groups of type ${}^1 \! A_2$ and
  stronger analytic results in this case. 
We also show in \cite{AvKlOnVo_pre_a}
  that irreducible characters of the principal congruence subgroups
  $\SL_3^1(\lri)$ \emph{are} extendable to their inertia groups in
  $\SL_3(\lri)$.  We anticipate that the arguments in the present
  paper will serve as a template for algebraic groups where this may
  either not be the case, not effectively decidable or where it is
  simply not practical to produce an explicit analysis.
\end{rem}

\subsection{Dirichlet generating functions}\label{subsec:dirichlet}
Consider Dirichlet generating functions $\xi(s) = \sum_{n \in \N} a_n
n^{-s}$ and $\eta(s) = \sum_{n \in \N} b_n n^{-s}$, encoding sequences
$(a_n)_{n\in\N}$ and $(b_n)_{n\in\N}$ of nonnegative integers.  We
write $\xi(s) \ll \eta(s)$ if $\sum_{n = 1}^N a_n \leq \sum_{n = 1}^N
b_n$ for all $N \in \N$.  Recall that, if the sequence $(a_n)$
contains infinitely many non-zero terms, the abscissa of convergence
of $\xi(s)$ is equal to $\limsup_{N \to \infty} \log \left( \sum_{n =
  1}^N a_n \right) / \log(N)$.  Thus we observe that $\xi(s) \ll
\eta(s)$ implies that the abscissa of convergence of $\xi(s)$ is less
than or equal to the abscissa of convergence of $\eta(s)$.  We require
the following straightforward lemma.

\begin{lem} \label{lem:generating_series}
  Let $\xi(s)$ and $\eta(s)$ be Dirichlet generating functions
  which admit product decompositions $\xi(s) = \prod_{i \in I}
  \xi_i(s)$ and $\eta(s) = \prod_{i \in I} \eta_i(s)$ over a
  countable index set $I$ in the ring of all Dirichlet generating
  functions.  Suppose that $\xi_i(s) \ll \eta_i(s)$ for each $i
  \in I$.

  Then $\xi(s) \ll \eta(s)$; in particular the abscissa of
  convergence of $\xi(s)$ is less than or equal to the abscissa of
  convergence of $\eta(s)$.
\end{lem}

\begin{proof}
  If $\prod_{j \in J} \xi_j(s) \ll \prod_{j \in J} \eta_j(s)$ for
  every finite subset $J \subseteq I$, then taking limits yields
  $\xi(s) \ll \eta(s)$.  Hence we may assume that $I$ is finite,
  and by induction on $\lvert I \rvert$ it is enough to consider the
  case $I = \{1,2\}$.

  It is convenient to introduce notation which puts the given
  generating functions in a more general setting.  For $i \in \{1,2\}$,
  we write $\xi_i(s) = \sum_{x \in X_i} \deg(x)^{-s}$ and $\eta_i(s) =
  \sum_{y \in Y_i} \deg(y)^{-s}$, where $X_i$ and $Y_i$ are countable
  sets and the degree maps $\deg: X_i \rightarrow \N$ and $\deg : Y_i
  \rightarrow \N$ have finite fibres.  Put $\mathbf{X} := X_1 \times
  X_2$ and define $\deg(\mathbf{x}) := \deg(x_1) \deg(x_2)$ for
  $\mathbf{x} = (x_1,x_2) \in \mathbf{X}$.  Similarly, we put
  $\mathbf{Y} := Y_1 \times Y_2$ and $\deg(\mathbf{y}) := \deg(y_1)
  \deg(y_2)$ for $\mathbf{y} = (y_1,y_2) \in \mathbf{Y}$.  Then
  \begin{equation}\label{equ:Dirichlet}
    \xi(s) = \xi_1(s) \xi_2(s) = \sum_{x_1 \in X_1}
    \deg(x_1)^{-s} \sum_{x_2 \in X_2}
    \deg(x_2)^{-s} = \sum_{\mathbf{x} \in \mathbf{X}}
    \deg(\mathbf{x})^{-s}
  \end{equation}
  and, similarly, $\eta(s) = \sum_{\mathbf{y} \in \mathbf{Y}}
  \deg(\mathbf{y})^{-s}$.

  For $i \in \{1,2\}$, the condition $\xi_i(s) \ll \eta_i(s)$ is
  equivalent to the existence of an injective map $\iota_i : X_i
  \rightarrow Y_i$ such that $\deg(x) \leq \deg(\iota_i(x))$ for all
  $x \in X_i$.  Clearly, the product map $\iota: X_1 \times X_2
  \rightarrow Y_1 \times Y_2$, $(x_1,x_2) \to (\iota_1(x_1),
  \iota_2(x_2))$ is injective and $\deg(\mathbf{x}) \leq
  \deg(\iota(\mathbf{x}))$ for all $\mathbf{x} \in X_1 \times X_2$.
  In conjunction with \eqref{equ:Dirichlet} and the analogous
  description of $\eta(s)$, this implies that $\xi(s) \ll \eta(s)$.
 \end{proof}


\subsection{Clifford theory} \label{subsec:clifford}

In this section we explain how Clifford theory can be used to compute
precisely or approximatively the representation zeta function of a FAb
compact $p$-adic analytic group $G$ from the representation zeta
function of a saturable normal pro-$p$ subgroup $N$, whose irreducible
characters can be described by the Kirillov orbit method.  For
instance, in Section~\ref{subsec:abscissa_inner} the theory will be
applied to the concrete groups $G = \SL_3(\lri)$ and $N =
\SL_3^1(\lri)$.  A general and accessible reference for Clifford
theory is \cite[\S~19--22]{Hu98}.

\subsubsection{}
We begin in a more general setting and specialise to our specific
cases of interest during the course of the discussion. Let $G$ be
group and $N \unlhd G$ a normal subgroup of finite index. Clifford
theory provides a connection between the irreducible characters of $G$
and those of $N$.  To start with, there is a natural action of $G$ on
$\widehat{N} = \Irr(N)$.  The stabiliser of $\theta \in \widehat{N}$
under this action is the inertia group $I_G(\theta)$.  Let $\Irr(G ,
\theta)$ denote the set of all irreducible characters $\rho$ of $G$
such that $\theta$ occurs as an irreducible constituent of the
restricted character $\res_N^G(\rho)$.  Then one has
$$
\Irr(G , \theta) = \{ \ind_{I_G(\theta)}^G(\psi) \mid \psi \in
\Irr(I_G(\theta) , \theta) \}.
$$ For any $\rho \in \Irr(G , \theta)$, the irreducible components of
$\res_N^G(\rho)$ are precisely the $G$-conjugates of $\theta$ and, in
particular, their number is $\lvert G : I_G(\theta) \rvert$.  We write
$$
\zeta_{G, \theta}(s) := \theta(1)^s \lvert G : I_G(\theta)
\rvert^s \sum_{\rho \in \Irr(G , \theta)} \rho(1)^{-s}
= \theta(1)^s \sum_{\psi \in
  \Irr(I_G(\theta) , \theta)} \psi (1)^{-s}.$$  If $N$
admits only finitely many irreducible characters of any given degree,
it follows that
\begin{equation}\label{equ:hochheben}
\zeta_G(s) = \sum_{\rho \in \widehat{G}} \rho(1)^{-s} = \sum_{\theta
  \in \widehat{N}} \theta(1)^{-s} \cdot \lvert G : I_G(\theta) \rvert^{-1-s}
\zeta_{G,  \theta}(s).
\end{equation}
\begin{rem}\label{rem:translation_factors}
In our applications of this equation, e.g.\ in Sections
\ref{subsec:abscissa_inner} and \ref{subsec:abscissa_outer}, the terms
$\lvert G : I_G(\theta) \rvert^{-1-s} \zeta_{G,  \theta}(s)$
appearing on the right hand side of \eqref{equ:hochheben} will be
referred to as `translation factors'.
\end{rem}

In the special case where $\theta$ extends to an irreducible character
$\hat \theta$ of $I_G(\theta)$, there is an effective description of
the elements $\psi \in \Irr(I_G(\theta) , \theta)$: according to
\cite[Theorem~19.6]{Hu98} one has
$$
\Irr(I_G(\theta) , \theta) = \left\{ \hat \theta \phi \mid \phi \in
\Irr(I_G(\theta)/N) \right\} \quad \text{and} \quad \zeta_{G ,
\theta}(s) = \zeta_{I_G(\theta)/N}(s).
$$
There are several basic sufficient criteria for the extensibility
of $\theta$.  For instance, if $N$ is a finite $p$-group it suffices
that a Sylow-$p$ subgroup of $I_G(\theta)/N$ is cyclic; see
\cite[Theorem~19.13]{Hu98}.

\begin{lem}\label{lem:translation-factor_neu}
  Let $G$ be a finite group, and let $N$ be a normal subgroup of $G$.
  Let $\theta \in \Irr(N)$ and set $I := I_G(\theta)$.  Let $q$ be a
  power of $p$.  Suppose that $N$ is a $p$-group and that $\theta(1)$
  is a power of $q$.  Let $P/N$ be a Sylow $p$-subgroup of $I/N$.
  \begin{enumerate}
  \item Suppose that $\theta$ extends to an irreducible character of
    $P$.  Then $\theta$ extends to $I$ and $\zeta_{G,  \theta}(s)
    = \zeta_{I/N}(s)$.
  \item Suppose that $\theta$ does not extend to $I$ and that the
    character degrees of $P$ are powers of $q$.  Then $\lvert P : N
    \rvert \geq q^2$ and $\zeta_{G,  \theta}(s) \ll \lfloor q^{-2}
    \lvert I:N \rvert \rfloor q^{-s}$.
  \end{enumerate}
\end{lem}

\begin{proof}
  (1) This follows from \cite[Theorems~21.4 and 22.3]{Hu98}.

  (2) We first prove the second statement. Let $\psi_1, \ldots,
  \psi_r$ denote the irreducible constituents of the induced character
  $\ind_N^I(\theta)$ so that $\ind_N^I(\theta) = \sum_{i=1}^r e_i
  \psi_i$.  We claim that (i) $r \leq q^{-2} \lvert I : N \rvert$ and
  (ii) $\psi_i(1) \geq q \theta(1)$ for all $i \in \{1,\ldots,r\}$.
  From this the claim about $\zeta_{G,  \theta}(s)$ follows
  directly.

  The proof of the inequality (ii) is straightforward.  Let $i \in
  \{1,\ldots,r\}$.  Since $\theta$ does not extend to $I$, it does not
  extend to $P$.  But $\phi := \res_P^I(\psi_i)$ admits an irreducible
  constituent whose restriction to $N$ involves $\theta$.  This shows
  that $\phi(1) > \theta(1)$.  Since $\phi(1)$ and $\theta(1)$ are
  powers of $q$, we deduce that $\psi_i(1) = \phi(1) \geq q
  \theta(1)$, as wanted.

  It remains to prove the inequality (i).  Observe that for each $i
  \in \{1,\ldots,r\}$ we have $\res_N^I(\psi_i) = e_i \theta$ so that
  $e_i = \psi_i(1)/\theta(1) \geq q$, using (ii).  Therefore
  $$
  \lvert I : N \rvert \theta(1) = (\ind_N^I(\theta))(1) =
  \sum_{i=1}^r e_i \psi_i(1) \geq r q^2 \theta(1),
  $$
  again using (ii).  Cancelling $\theta(1)$ and rearranging terms
  yields (i).

  The claim $\lvert P : N \rvert \geq q^2$ follows from (i), if we
  replace both $G$ and $I$ by $P$.



\end{proof}

\subsubsection{}\label{subsec:commuting_diagram}
Now suppose that $G$ is a FAb compact $p$-adic analytic group and $N$
is a saturable normal pro-$p$ subgroup of $G$, whose irreducible
characters can be described by the Kirillov orbit method.  A central
step in connecting $\widehat{G} = \Irr(G)$ with $\widehat{N} =
\Irr(N)$ consists in computing the inertia groups $I_G(\theta)$ for
$\theta \in \widehat{N}$.  Let $\mfn$ be the saturable $\Z_p$-Lie
lattice associated to~$N$.  Then each $\theta \in \widehat{N}$
corresponds to an orbit $\Omega = \omega^N$ in the co-adjoint action
of $N$ on $\widehat{\mfn} = \Irr(\mfn)$.  Thus $I_G(\theta)$ is the
set-wise stabiliser of $\Omega$ under the co-adjoint action of $G$ on
$\widehat{\mfn}$, and
\begin{equation} \label{equ:intertia} I_G(\theta) =
  \textup{N}_G(\Omega) = \Cen_G(\omega)N.
\end{equation}

We now specialise to the situation which is most interesting to us.
Let $\gfi$ be an algebraic number field with ring of integers $\gri$,
and let $\Lambda$ be an $\gri$-Lie lattice such that $\gfi
\otimes_\smallgri \Lambda$ is a finite dimensional, semisimple
$\gfi$-Lie algebra.  Let $\lri = \gri_v$ be the completion of $\gri$
at a non-archimedean place $v$ -- to be restricted further during our
discussion -- with maximal ideal $\mfp = \pi \lri$, lying above the
rational prime $p$.  Put $\mfg := \lri \otimes_\smallgri \Lambda$, and
suppose that $\mfn$ is the $\lri$-Lie lattice $\mfg^m = \mfp^m \mfg$
for some $m \in \N_0$.

Observe that multiplication by $\pi^{-m}$ provides a $G$-equivariant
isomorphism of $\lri$-modules $\mfn \rightarrow \mfg$, and a
corresponding $G$-equivariant bijection $\widehat{\mfg} \rightarrow
\widehat{\mfn}$.  By means of these maps we may describe elements of
$\widehat{N}$ by orbits in the co-adjoint action of $N$ on
$\widehat{\mfg}$ rather than $\widehat{\mfn}$.  From a computational
point of view, it is convenient to work with the adjoint action of $G$
on $\mfg$ rather than the co-adjoint action on $\widehat{\mfg}$.  In
general, the necessary translation can be made by means of the Killing
form $\kappa$ associated to the $\gri$-Lie lattice~$\Lambda$.  For
simplicity we shall assume that $\Lambda = \gri \otimes_\Z \Lambda_0$
arises from a $\Z$-Lie lattice $\Lambda_0$ of Chevalley type,
associated to a simply-connected simple group scheme over~$\Z$.  Then
the transition from the co-adjoint to the adjoint action can be
carried out more effectively by using the normalised Killing form
$$
\kappa_0 := (2h^\vee)^{-1} \kappa: \Lambda \times \Lambda \rightarrow
\gri,
$$
where $h^\vee$ denotes the dual Coxeter number.  Writing $d :=
\dim_\gfi(\gfi \otimes_\smallgri \Lambda)$, the bilinear form
$\kappa_0$ can be represented by a structure matrix $B \in \Mat_d(\Z)$
with respect to an $\gri$-basis of $\Lambda$ obtained from a
$\Z$-basis of $\Lambda_0$; \cite[Section~5]{GrNe04}.  The determinant
of $B$ is a non-zero integer.  For instance, for $\Lambda_0 = \spl_n$,
the determinant of $B$ is equal to $n$.  Excluding from our discussion
those places $v$ which lie above primes $p$ dividing the determinant
of $B$, the form $\kappa_0$ extends to a `non-degenerate' form on the
local $\lri$-Lie lattice $\mfg$: for every $x \in \mfg^*$ there exists
$y \in \mfg$ such that $\kappa_0(x,y) \in \lri^*$.

Thus $\kappa_0$ induces a natural $G$-equivariant isomorphism of
$\lri$-modules
$$
\iota_0: \mfg \rightarrow \Hom_{\lri}(\mfg,\lri)^*, \quad
x \mapsto \kappa_0(x,\cdot).
$$ For each $n \in \N$ there is a natural surjective map from
$\Hom_{\lri}(\mfg,\lri)^*$ onto $\Irr_n(\mfg) \cong
\Hom_\lri(\mfg,\lri/\mfp^n)^*$, and these sets partition $\Irr(\mfg)$;
cf.\ Lemma~\ref{lem:dual}.  Thus the normalised Killing form gives
rise to a $G$-equivariant commutative diagram
\begin{equation}\label{equ:diagram}
\begin{CD}
  \mfg^* @>{\cong}>>  \Hom_{\lri}(\mfg,\lri)^* \\
  @VVV @VVV \\
  (\mfg / \mfp^n \mfg)^* @>{\cong}>> \Hom_\lri(\mfg/\mfp^n\mfg,\lri /
  \mfp^n)^* @>{\cong}>> \{ \omega \in \widehat{\mfg} \mid
  \lev_\lri(\omega) = n \} \\
  @VVV @VVV \\
  (\F_q \otimes_\Z \Lambda_0)^* @>{\cong}>> \Hom_{\F_q}(\F_q
  \otimes_\Z \Lambda_0, \F_q)^*
\end{CD}
\end{equation}
where the last row is obtained by reduction modulo $\mfp$ and we have
used the isomorphism $\lri/\mfp \cong \F_q$.

Let us now restrict further to the case $m=1$, i.e.\ $\mfn = \mfg^1 =
\mfp \mfg$.  Consider a non-trivial character $\theta \in \Irr(N)$,
represented by an orbit $\omega^N$ in $\widehat{\mfg} \cong
\widehat{\mfn}$.  Put $n := \lev_\lri(\omega)$, and make use of the
diagram~\eqref{equ:diagram}.  We choose $w \in
\Hom_{\lri}(\mfg,\lri)^*$
such that its image $w \vert_n$ in $\Hom_\lri(\mfg/\mfp^n\mfg,\lri /
\mfp^n)^*$ corresponds to $\omega$, and we write $\overline{w}$ for
the image of $w$ in $\Hom_{\F_q}(\F_q \otimes_\Z \Lambda_0,\F_q)^*$.
By means of the horizontal isomorphisms, induced by the isomorphism
$\iota_0$, we find $x$, $x \vert_n = x + \mfp^n \mfg$, $\overline{x}$
corresponding to $w$, $w \vert_n$, $\overline{w}$ respectively.

Since all the maps in \eqref{equ:diagram} are $G$-equivariant, we
conclude that $\Cen_G(\omega) = \Cen_G(x + \mfp^n \mfg)$.  In general,
$\Cen_G(x + \mfp^n \mfg) \leq \Cen_G(\overline{x})$ and this inclusion
provides an upper bound for the size of the inertia group quotient.
Under certain circumstances, one can prove that $\Cen_G(x) N =
\Cen_G(x + \mfp^n \mfg) N = \Cen_G(\overline{x})$, hence
$$
I_G(\theta) = \Cen_G(\overline{x}) \qquad \text{and} \qquad
I_G(\theta)/N \cong \Cen_{G/N}(\overline{x}).
$$
The quotient group $G/N$ is typically (a subgroup of) a finite group
of Lie type, and linear algebra allows one to describe the various
possibilities for the centraliser $C := \Cen_{G/N}(\overline{x})$.  In
the special cases of interest to us in the current paper, such an
analysis is carried in Appendices~\ref{sec:aux_sl3} and
\ref{sec:aux_su3}.

\subsubsection{} Let $n \in \N$.  We conclude the
Section~\ref{subsec:clifford} by recording some specific results for
the case where the completion of $\Lambda$ leads to one of the
$\lri$-Lie lattices $\spl_n(\lri)$ and $\su_n(\Lri,\lri)$.  Here and
in the following we use the notation introduced at the beginning of
Section~\ref{sec:explicit_cong_sub}.  We fix an unramified quadratic
extension $\Lri \vert \lri$ of compact discrete valuation rings of
characteristic $0$.  The non-trivial Galois automorphism of the
corresponding extension $\Lfi \vert \lfi$ of fields of fractions is
denoted by $\sigma$.  The standard involution on $\Mat_n(\Lfi)$ is
denoted by $\circ$.  As $\Lri$ is unramified over $\lri$, the
uniformiser $\pi$ of $\lri$ also uniformises~$\Lri$.

The results in this section allow one to pin down various inertia
group quotients in the context of Clifford theory for the group
extensions $\SL_n^1(\lri) \unlhd \SL_n(\lri)$ and $\SU_n^1(\Lri,\lri)
\unlhd \SU_n(\Lri,\lri)$; for $n=3$ such an analysis is carried out
explicitly in Sections~\ref{subsec:abscissa_inner} and
\ref{subsec:abscissa_outer}.

\begin{lem}\label{lem:lift_general}
  Let $\mathbf{a}$ be an element of the $\lri$-order $\Mat_n(\lri)$ of
  the central simple $\lfi$-algebra $\Mat_n(\lfi)$.  Writing
  $\overline{\phantom{x}}$ for reduction modulo $\mfp$, suppose that
  the minimum polynomial of $\overline{\mathbf{a}}$ over $\F_q =
  \overline{\lri}$ is equal to the characteristic polynomial of
  $\overline{\mathbf{a}}$ in $\Mat_n(\F_q)$.  Then one has
  $$
  \Cen_{\Mat_n(\F_q)}(\overline{\mathbf{a}}) =
  \F_q[\overline{\mathbf{a}}] \qquad \text{and} \qquad
  \Cen_{\Mat_n(\lri)}(\mathbf{a}) = \lri[\mathbf{a}];
  $$
  consequently, $\Cen_{\Mat_n(\F_q)}(\overline{\mathbf{a}}) =
  \overline{\Cen_{\Mat_n(\lri)}(\mathbf{a})}$.
\end{lem}

\begin{proof}
  Let $f \in \lri[X]$ denote the characteristic polynomial of
  $\mathbf{a}$ over $\lfi$.  Since $\overline{f}$ is equal to the
  minimum polynomial of $\overline{\mathbf{a}}$ over $\F_q$, the
  polynomial $f$ is equal to the minimum polynomial of $\mathbf{a}$
  over $\lfi$.  This implies that the natural
  $\F_q[\overline{\mathbf{a}}]$- and $\lri[\mathbf{a}]$-modules
  $\overline{V} = \F_q^n$ and $V = \lri^n$ are each free of rank $1$.
  Choose a free generator $x$ of $V$ mapping onto a free generator
  $\overline{x}$ of $\overline{V}$.  The centralisers featuring in our
  claim are identical to the endomorphism rings of the free modules
  $\overline{V}$ and $V$.  By considering the images of $\overline{x}$
  and~$x$, these endomorphism rings are seen to be equal to
  $\F_q[\overline{\mathbf{a}}]$ and $\lri[\mathbf{a}]$, respectively.
\end{proof}

\begin{cor}\label{cor:lift_spl}
  Let $\mathbf{x} \in \spl_n(\lri)^*$, and let $\overline{\mathbf{x}}
  \in \spl_n(\F_q)^*$ denote the reduction of $\mathbf{x}$
  modulo~$\mfp$.  Suppose that $p \nmid n$ and that the minimum
  polynomial of $\overline{\mathbf{x}}$ over $\F_q$ coincides with the
  characteristic polynomial of $\overline{\mathbf{x}}$ in
  $\Mat_n(\F_q)$.  Then one has
  $$
  \Cen_{\SL_n(\lri)}(\overline{\mathbf{x}}) =
  \Cen_{\SL_n(\lri)}(\mathbf{x}) \SL_n^1(\lri).
  $$
\end{cor}

\begin{proof}
  Certainly, the group on the right hand side is contained in the
  group on the left hand side.  We prove the reverse inclusion.  Let
  $g \in \Cen_{\SL_n(\lri)}(\overline{\mathbf{x}})$.  Then
  Lemma~\ref{lem:lift_general} produces $h \in
  \Cen_{\Mat_n(\lri)}(\mathbf{x})$ such that $h \equiv g$ modulo
  $\mfp$.  Observe that $\det(h) \in 1 + \mfp$ is an $n$-th power,
  because $p \nmid n$.  Let $\lambda \in 1 + \mfp$ such that
  $\lambda^n = \det(h)^{-1}$.  Then $\lambda h \in
  \Cen_{\SL_n(\lri)}(\mathbf{x})$ such that $h \equiv g$ modulo $\mfp$
  or, equivalently, modulo $\SL_n^1(\lri)$, as wanted.
\end{proof}

\begin{cor}\label{cor:lift_su}
  Let $\mathbf{x} \in \su_n(\Lri,\lri)^*$, and let
  $\overline{\mathbf{x}} \in \su_n(\F_{q^2},\F_q)^*$ denote the
  reduction of $\mathbf{x}$ modulo~$\mfP$.  Suppose that $p > 2$ with
  $p \nmid n$ and that the minimum polynomial of
  $\overline{\mathbf{x}}$ over $\F_{q^2}$ coincides with the
  characteristic polynomial of $\mathbf{x}$ in $\Mat_n(\F_{q^2})$.
  Then one has
  $$
  \Cen_{\SU_n(\Lri,\lri)}(\overline{\mathbf{x}}) =
  \Cen_{\SU_n(\Lri,\lri)}(\mathbf{x}) \SU_n^1(\Lri,\lri).
  $$
\end{cor}

\begin{proof}
  Again, the group on the right hand side is clearly contained in the
  group on the left hand side, and it remains to prove the reverse
  inclusion.  Let $g \in
  \Cen_{\SU_n(\Lri,\lri)}(\overline{\mathbf{x}})$.  By
  Lemma~\ref{lem:lift_general}, we find $h_1 \in
  \Cen_{\Mat_n(\Lri)}(\mathbf{x})$ such that $h_1 \equiv g$ modulo
  $\mfP$.

  We observe that $h_1^\circ\, h_1 \equiv 1$ modulo $\mfP$, and we carry
  out a `Hensel's Lemma type' argument to manufacture, inductively, a
  sequence $h_i$, $i \in \N$, in
  $\Cen_{\Mat_n(\Lri)}(\overline{\mathbf{x}})$ such that
  $$
  h_i^\circ \, h_i \equiv 1 \quad \text{and} \quad h_{i+1} \equiv h_i
  \quad \text{modulo $\mfP^i$}
  $$
  which, in the limit, yields $h = \lim_{i \to \infty} h_i \in
  \Cen_{\GU_n(\Lri,\lri)}(\mathbf{x})$ such that $h \equiv g$ modulo
  $\mfP$.  Similarly as in the proof of Corollary~\ref{cor:lift_spl},
  a modification by a scalar $\lambda \in \GU_1^1(\Lri,\lri)$ such
  that $\lambda^n = \det(h)^{-1}$ then yields the desired element
  $\lambda h \in \Cen_{\SU_n(\Lri,\lri)}(\mathbf{x})$ such that $h
  \equiv g$ modulo $\mfP$ or, equivalently, modulo
  $\SU_n^1(\Lri,\lri)$.

  Let $i \in \N$ and suppose we have determined $h_1, \ldots, h_i \in
  \Cen_{\Mat_n(\Lri)}(\mathbf{x})$ with the desired properties.  Since
  $\mathbf{x}^\circ = -\mathbf{x}$, the centraliser
  $\Cen_{\Mat_n(\Lri)}(\mathbf{x})$ is invariant under the
  involution~$\circ$.  Recall that $\pi$ is a uniformiser for $\Lri$,
  and put $h_{i+1} := h_i + \pi^i y_i$ where
  $$
  y_i := - 2^{-1} \pi^{-i} (h_i^\circ)^{-1} (h_i^\circ \, h_i - 1) \in
  \Cen_{\Mat_n(\Lri)}(\mathbf{x}).
  $$
  Then $h_{i+1} \equiv h_i$ modulo $\mfP^i$, and modulo $\mfP^{i+1}$
  one has
  \begin{align*}
    h_{i+1}^\circ\, h_{i+1} -1 & = (h_i^\circ\, h_i + \pi^i (h_i^\circ
    \, y_i + y_i^\circ \, h_i) + \pi^{2i} y_i^\circ \, y_i) -1 \\
    & \equiv (h_i^\circ \, h_i - 1) + \pi^i \big( (h_i^\circ \, y_i) +
    (h_i^\circ \, y_i)^\circ \big) \\
    & \equiv 0.
  \end{align*}
\end{proof}

\begin{lem}\label{lem:q-power}
  Suppose that $p > 2$ with $p \nmid n$.  Let $H$ be a Sylow pro-$p$
  subgroup of $\SL_n(\lri)$ or $\SU_n(\Lri,\lri)$, and accordingly let
  $N$ be equal to $\SL_n^1(\lri)$ or $\SU_n^1(\Lri,\lri)$.  Suppose
  that $H$ is saturable and that $N$ is saturable and potent.  Let
  $\theta \in \Irr(N)$.

  Then $I_H(\theta)$ corresponds to an $\lri$-Lie sublattice of the
  $\lri$-Lie lattice associated to $H$.  In particular, $\lvert
  I_H(\theta):N \rvert$ is a power of $q = \lvert \lri: \mfp \rvert$.
\end{lem}

\begin{proof}
  First suppose that $H$ is a Sylow pro-$p$ subgroup of $\SL_n(\lri)$
  and that $N = \SL_n^1(\lri)$.  Since $N$ is saturable and potent and
  since $p \nmid n$, we can apply the Kirillov orbit method and use
  the translation via the normalised Killing form in order to describe
  $I_H(\theta)$ in terms of the adjoint action.  Let $\mfn =
  \spl_n^1(\lri)$ denote the $\lri$-Lie lattice associated to $N$.  In
  contrast to the discussion in
  Section~\ref{subsec:commuting_diagram}, we consider directly the
  adjoint action on $\mfn$, without shifting by a factor $\pi^{-1}$.
  Let $\omega^N$ be the co-adjoint orbit representing~$\theta$.  Then
  we have $I_H(\theta) = \Cen_H(x + \mfp^l \mfn) N$, where $l =
  \lev_\lri(\omega)$, $x + \mfp^l \mfn$ corresponds to $\omega$ via
  the normalised Killing form, and the centraliser is formed with
  respect to the adjoint action.  Since $H$ is saturable, it also
  corresponds to an $\lri$-Lie lattice $\mathfrak{h}$, and the
  centraliser $\Cen_H(x + \mfp^l \mfn)$ at the group level corresponds
  to the centraliser $\mathfrak{c} := \Cen_\mathfrak{h}(x + \mfp^l
  \mfn)$ at the Lie lattice level.  Since $\mathfrak{c} + \mfn$ and
  $\mfn$ are $\lri$-Lie sublattices of $\mathfrak{h}$, it follows that
  $\lvert I_H(\theta):N \rvert = \lvert \mathfrak{c} + \mfn : \mfn
  \rvert$ is a power of $q$.

  The argument for the second case, i.e.\ $H$ a Sylow pro-$p$ subgroup
  of $\SU_n(\Lri,\lri)$ and $N = \SU_n^1(\Lri,\lri)$, is similar.  But
  we need to justify that $I_H(\theta)$ can be described in terms of
  the adjoint action of $N$ on the $\lri$-Lie lattice $\mfn =
  \su_n^1(\Lri,\lri)$.  For this purpose we consider
  $\su_n(\Lri,\lri)$ as the subset of the $\Lri$-Lie lattice
  $\spl_n(\Lri)$, consisting of all elements $\mathbf{x}$ such that
  $\mathbf{x}^\circ + \mathbf{x} = 0$.  We claim that the normalised
  Killing form $\kappa_0$ on the $\Lri$-Lie lattice $\spl_3(\Lri)$
  restricts to a non-degenerate bilinear form
  $$ \kappa_0 \vert_{\su_n(\Lri,\lri) \times \su_n(\Lri,\lri)} :
  \su_n(\Lri,\lri) \times \su_n(\Lri,\lri) \rightarrow \lri.
  $$ Indeed, for any $\mathbf{x}, \mathbf{y} \in \su_n(\Lri,\lri)$ one
computes $\kappa_0(\mathbf{x},\mathbf{y})^\sigma =
\kappa_0(\mathbf{x}^\circ,\mathbf{y}^\circ) =
\kappa_0(-\mathbf{x},-\mathbf{y}) = \kappa_0(\mathbf{x},\mathbf{y})$
which implies $\kappa_0(\mathbf{x},\mathbf{y}) \in \lri$.  This shows
that $\kappa_0 \vert_{\su_n(\Lri,\lri)}$ maps into $\lri$.  Of course,
$\kappa_0 \vert_{\su_n(\Lri,\lri) \times \su_n(\Lri,\lri)}$ is
$\lri$-bilinear, and it remains to show that it is non-degenerate over
$\lri$.  Let $\mathbf{x} \in \su_n(\Lri,\lri)^*$.  We need to produce
$\mathbf{y} \in \su_n(\Lri,\lri)$ such that
$\kappa_0(\mathbf{x},\mathbf{y}) \in \lri^*$.  Since $\kappa_0$ is
non-degenerate on the ambient $\Lri$-Lie lattice $\spl_n(\Lri)$, we
find $\mathbf{z} \in \spl_n(\Lri)$ such that $\lambda :=
\kappa_0(\mathbf{x},\mathbf{z}) \in \Lri^*$.  Multiplying $\mathbf{z}$
by a suitable scalar, if necessary, we can arrange for $\lambda$ to be
such that $\lambda + \lambda^\sigma \not \equiv 0$ modulo $\mfP$.
(This is equivalent to the condition that the reduction of $\lambda$
modulo $\mfP$ does not have trace $0$ in the residue field extension
$\F_{q^2} \vert \F_q$.)  Set $\mathbf{y} := \mathbf{z} -
\mathbf{z}^\circ$.  Then $\mathbf{y}^\circ + \mathbf{y} = 0$ so that
$\mathbf{y} \in \su_n(\Lri,\lri)$, and $- \mathbf{x} =
\mathbf{x}^\circ$ implies
  $$
  \kappa_0(\mathbf{x},\mathbf{y}) = \kappa_0(\mathbf{x},\mathbf{z}) -
  \kappa_0(\mathbf{x},\mathbf{z}^\circ) =
  \kappa_0(\mathbf{x},\mathbf{z}) +
  \kappa_0(\mathbf{x}^\circ,\mathbf{z}^\circ) = \lambda +
  \lambda^\sigma \not \equiv 0
  $$
  modulo $\mfp = \lri \cap \mfP$, as wanted.
\end{proof}

\subsection{Arithmetic groups of type ${}^1 \! A_2$}
\label{subsec:abscissa_inner}
In this section we prove the assertion of Theorem~\ref{thmABC:lalu}
for arithmetic subgroups pertaining to inner forms of $A_2$.  After a
short general setup, covering both inner and outer forms of $A_2$, we
specialise to inner forms from Section~\ref{subsubsec:7.3.1} onward.
Outer forms are then treated in Section~\ref{subsec:abscissa_outer} by
similar techniques.

Let $\gfi$ be a number field, with ring of integers $\gri$.  Let $S$
be a finite set of places of $\gfi$, containing all the archimedean
ones.  Let $\mathbf{G}$ be a connected, simply-connected simple,
$\gfi$-defined algebraic group of type $A_2$.  As explained in
Appendix~\ref{subsec:algebraic/arithmetic_groups}, the classification
of semisimple groups over number fields gives a concrete description
of the group $\mathbf{G}(\gfi)$ of $\gfi$-rational points in terms of
central simple algebras, possibly equipped with an involution.  We
consider an $S$-arithmetic subgroup $\Gamma$ of $\mathbf{G}(\gfi)$.
This means that $\Gamma$ is commensurable to $\mathbf{G}(\gri_S) :=
\mathbf{G}(\gfi) \cap \GL_N(\gri_S)$, where implicitly one chooses an
embedding of $\mathbf{G}$ into $\mathbf{GL}_N$ for some $N \in \N$.
Suppose that $\Gamma$ has the Congruence Subgroup Property, i.e.\ that
its congruence kernel is finite, so that $\Gamma$ has polynomial
representation growth.

We are to show that the abscissa of convergence $\alpha(\Gamma)$ of
$\zeta_\Gamma(s)$ is equal to $1$.  Since $\alpha(\Gamma)$ is in fact
an invariant of the commensurability class of $\Gamma$, we may -- for
simplicity -- suppose that $\Gamma = \mathbf{G}(\gri_S)$.  By
\cite[Proposition~4.6]{LaLu08}, the zeta function $\zeta_\Gamma(s)$
thus admits an `Euler product' decomposition
$$
\zeta_\Gamma(s) = \zeta_{\mathbf{G}(\C)}(s)^{\lvert k : \Q \rvert}
\cdot \prod_{v \not \in S} \zeta_{\Gamma_v}(s),
$$ where $\Gamma_v$ is an open subgroup of $\mathbf{G}(\gri_v)$ for
every non-archimedean place $v$ and equal to $\mathbf{G}(\gri_v)$ for
almost all~$v$.  Note that $\mathbf{G}(\C) \cong \SL_3(\C)$ as
$\mathbf{G}$ is of type $A_2$.  The abscissa of convergence of each
archimedean factor $\zeta_{\mathbf{G}(\C)}(s)$ is thus known to be
$2/3$; see \cite[Theorem~5.1]{LaLu08}.  For any non-archimedean place
$v$, there is a finite extension $\Lri$ of $\lri = \gri_v$ such that
$\mathbf{G}(\Lri)$ is commensurable to $\SL_3(\Lri)$; thus
Theorem~\ref{thmABC:poles}, in particular
\eqref{equ:comparison_abscissae}, and Theorem~\ref{thmABC:SL3} show
that the abscissa of convergence of the local factor
$\zeta_{\Gamma_v}(s)$ is at most $2/3$.  Therefore, in order to prove
that the abscissa of convergence $\alpha(\Gamma)$ of the global zeta
function $\zeta_\Gamma(s)$ is equal to $1$, it suffices to specify a
suitable finite set $T$ of places of $\gfi$, with $S \subseteq T$,
such that the abscissa of convergence of the product $\prod_{v \not
\in T} \zeta_{\Gamma_v}(s)$ is equal to $1$.

\subsubsection{}\label{subsubsec:7.3.1}
From now on we specialise to the case where $\mathbf{G}$, and hence
$\Gamma$, is of type ${}^1 \! A_2$.  Let $T$ be the set of all places
$v$ of $\gfi$ such that one of the following holds: (i) $v \in S$,
(ii) $v$ is dyadic or triadic, i.e.\ $\gri_v$ has residue field
characteristic $2$ or $3$, (iii) $\Gamma_v$ is not isomorphic to
$\SL_3(\gri_v)$, (iv) there is a pro-$p$ subgroup $H$ of
$\SL_3(\gri_v)$ containing $\SL_3^1(\gri_v)$ which fails to be
saturable or potent.  We need to check that $T$ is finite.  Clearly,
conditions (i) and (ii) only exclude finitely many places.  Standard
theorems from the theory of orders in central simple algebras over
number fields show that (iii) only excludes finitely many valuations.
Indeed, $\mathbf{G}(\gfi)$ can be thought of as the norm-$1$ group of
a central simple algebra $A$ over $\gfi$, and $\mathbf{G}(\gri_S)$ as
the set of norm-$1$ elements in an $\gri_S$-order $\Omega$ of $A$.
For almost all $v$, the completion $A_v$ is split, hence isomorphic to
$\Mat_3(\gfi_v)$, and the completion $\Omega_v$ is a maximal order of
$A_v \cong \Mat_3(\gfi_v)$, hence isomorphic to $\Mat_3(\gri_v)$;
e.g.\ see \cite[Corollary~10.4, Corollary~11.6, Theorem~17.3]{Re75}.
Finally, we turn our attention to the property (iv).  For almost all
$v$ the field $\gfi_v$ has absolute ramification index $1$, and hence
\cite[III (3.2.7)]{La65} shows that the Sylow pro-$p$ subgroup of
$\SL_3(\gri_v)$ is saturable and therefore torsion-free.  Since
$\dim(\SL_3(\gri_v))$ is uniformly bounded by $8 \lvert k : \Q
\rvert$, Gonz\'alez-S\'anchez and Klopsch's work \cite{GoKl09} on
$p$-adic analytic groups of small dimension shows that the property
(iv) can be avoided by excluding finitely many $v$.

We claim that the abscissa of convergence of the product
$$
\phi(s) := \prod_{v \not \in T} \zeta_{\SL_3(\smallgri_v)}(s)
$$
is equal to $1$.  As shown in \cite{AvKlOnVo_pre_a},
one can compute via Clifford theory, for each $v \not \in T$, an
explicit formula for $\zeta_{\SL_3(\smallgri_v)}(s)$, starting from
our explicit formula \eqref{eq:starting_point} for the local
representation zeta function of the first principal congruence
subgroup $\SL_3^1(\gri_v)$.  As explained in
Remark~\ref{rem:other_paper}, we take a different approach and
manufacture, without relying on \cite{AvKlOnVo_pre_a}, 
sufficient approximations $\psi_v(s)$ of
$\zeta_{\SL_3(\smallgri_v)}(s)$ so that
\begin{itemize}
\item the abscissa of convergence of $\psi(s) := \prod_{v \not \in T}
  \psi_v(s)$ is the same as the one of $\phi(s)$,
\item this abscissa of convergence can be shown to
  be equal to $1$.
\end{itemize}

Indeed, for each $v \not \in T$ there is a natural decomposition
\begin{equation}\label{equ:zeta_decomp}
  \zeta_{\SL_3(\smallgri_v)}(s) = \left( 1 +
    \zeta_{\SL_3(\smallgri_v)}^{\textup{triv}}(s) \right) +
  \zeta_{\SL_3(\smallgri_v)}^{\textup{reg}}(s) +
  \zeta_{\SL_3(\smallgri_v)}^{\textup{irreg}}(s)
\end{equation}
involving three `smaller' Dirichlet generating functions.  These
components enumerate representations which, after restriction to
$\SL_3^1(\gri_v)$, involve only the trivial representation, only
regular and only irregular representations respectively; recall the
terminology introduced in Section~\ref{subsec:sl3princ}.  First we
will prove that the abscissa of convergence of the product $\prod_{v
\not \in T} (1 + \zeta_{\SL_3(\smallgri_v)}^{\textup{triv}}(s))$ is
equal to $1$.  Lemma~\ref{lem:generating_series} and the trivial
observation
\begin{multline}\label{equ:zeta_prod_decom}
  1 + \zeta_{\SL_3(\smallgri_v}^{\textup{triv}}(s) +
  \zeta_{\SL_3(\smallgri_v)}^{\textup{reg}}(s) +
  \zeta_{\SL_3(\smallgri_v)}^{\textup{irreg}}(s) \ll \\ \left( 1 +
    \zeta_{\SL_3(\smallgri_v)}^{\textup{triv}}(s)
    \vphantom{\zeta_{\SL_3(\smallgri_v)}^{\textup{reg}}(s)} \right)
  \left( 1 + \zeta_{\SL_3(\smallgri_v)}^{\textup{reg}}(s) +
    \zeta_{\SL_3(\smallgri_v)}^{\textup{irreg}}(s) \right)
\end{multline}
then show that it suffices to prove that the abscissa of convergence
of the product $\prod_{v \in \mathcal{V}(\gfi) \setminus T} (1 +
\zeta_{\SL_3(\smallgri_v)}^{\textup{reg}}(s) +
\zeta_{\SL_3(\smallgri_v)}^{\textup{irreg}}(s))$ is less than or equal
to $1$.  For this purpose we will approximate the Dirichlet series
$\zeta_{\SL_3(\smallgri_v)}^{\textup{reg}}(s)$ and
$\zeta_{\SL_3(\smallgri_v)}^{\textup{irreg}}(s)$ by suitable rational
functions $\psi_v^{\textup{reg}}(s)$ and $\psi_v^{\textup{irreg}}(s)$
in $q_v$, $q_v^{-s}$ and $2^{-s}$, where $q_v = \lvert \gri_v : \mfp_v
\rvert$ denotes the size of the residue field $\gri_v / \mfp_v$
at~$v$.

\subsubsection{} The first factor on the right hand side of
\eqref{equ:zeta_prod_decom}, viz.\ $1 +
\zeta_{\SL_3(\smallgri_v)}^{\textup{triv}}(s)$, arises from
representations of $\SL_3(\gri_v)$ which are trivial on
$\SL_3^1(\gri_v)$.  Thus $1 +
\zeta_{\SL_3(\smallgri_v)}^{\textup{triv}}(s) =
\zeta_{\SL_3(\F_{q_v})}(s)$, the zeta function of the finite classical
group $\SL_3(\F_{q_v})$.

\begin{pro}\label{pro:first_product}
  The abscissa of convergence of the Euler product
  $$
  \prod_{v \not \in T} \left( 1 +
    \zeta_{\SL_3(\smallgri_v)}^{\textup{triv}}(s) \right) = \prod_{v
    \not \in T} \zeta_{\SL_3(\F_{q_v})}(s)
  $$
  is equal to $1$.
\end{pro}

\begin{proof}
  Let $q$ be a prime power, not divisible by $3$.  Explicit formulae
  for the zeta function of the finite classical group $\SL_3(\F_q)$
  can be computed, for instance, using Deligne-Lusztig Theory; cf.\
  \cite{AvKlOnVo_pre_a} and \cite{LuXX}.
One needs to separate the cases $q \equiv_3 1$ and $q \equiv_3 2$.
  Indeed, for $q \equiv_3 1$ one has
  \begin{align*}
    \zeta_{\SL_3(\F_q)}(s) = 1 & + 1 \cdot (q^2+q)^{-s} + (q-2) \cdot
    (q^2+q+1)^{-s} \\
    & + 6 \cdot ((q+1)(q-1)^2/3)^{-s} + 3 \cdot
    ((q^2+q+1)(q+1)/3)^{-s} \\
    & + 3^{-1} (q+2)(q-1) \cdot ((q+1)(q-1)^2)^{-s} +
    2^{-1} (q^2-q)  \cdot (q^3-1)^{-s} \\
    & + 1 \cdot q^{-3s} + (q-2) \cdot (q^3+q^2+q)^{-s} \\
    & + 6^{-1} (q-1)(q-4) \cdot ((q^2+q+1)(q+1))^{-s},
  \end{align*}
  and for $q \equiv_3 2$ one has
  \begin{align*}
    \zeta_{\SL_3(\F_q)}(s) = 1 & + 1 \cdot (q^2+q)^{-s} + (q-2) \cdot
    (q^2+q+1)^{-s} \\
    & + 3^{-1} (q^2+q) \cdot ((q+1)(q-1)^2)^{-s} + 2^{-1} (q^2-q)
    \cdot
    (q^3-1)^{-s} \\
    & + 1 \cdot q^{-3s} + (q-2) \cdot (q^3+q^2+q)^{-s} \\
    & + 6^{-1}(q-2)(q-3) \cdot ((q^2+q+1)(q+1))^{-s}.
  \end{align*}

  Based on these formulae we can prove our assertion.  Indeed, the
  degree $\lvert k:\Q \rvert$ provides a uniform upper bound for the
  number of valuations $v \not \in T$ which prolong any fixed $p$-adic
  valuation on $\Q$.  Moreover, for each $v$ which prolongs the
  $p$-adic valuation associated to a prime $p$ we have $p \leq q_v =
  \lvert \gri_v:\mfp_v \rvert$, and the Chebotarev Density Theorem
  guarantees that for a positive proportion of primes $p$, i.e.\ for a
  set of primes $p$ of positive Dirichlet density, there exists a
  prolongation $v$ such that $p = q_v$.  In view of the explicit
  formulae for $\zeta_{\SL_3(\F_q)}(s)$ recorded above, this implies
  that the abscissa of convergence of the product $\prod_{v \not \in
  T} \zeta_{\SL_3(\F_{q_v})}(s)$ is equal to the abscissa of
  convergence of the sum $\sum_p (p^{1-2s} + p^{2-3s})$, which is
  equal to~$1$.
\end{proof}

\subsubsection{}\label{subsubsec:7.3.3}
It remains to bound the abscissa of convergence of the second product
on the right hand side of \eqref{equ:zeta_prod_decom}.

\begin{pro}\label{pro:second_product}
  The abscissa of convergence of the Euler product
  $$
  \prod_{v \not \in T} \left( 1+
    \zeta_{\SL_3(\smallgri_v)}^{\textup{reg}}(s) +
    \zeta_{\SL_3(\smallgri_v)}^{\textup{irreg}}(s) \right)
  $$
  is less than or equal to $1$.
\end{pro}

The proof of this proposition will occupy the remainder of the current
Section~\ref{subsec:abscissa_inner}.  By
Lemma~\ref{lem:generating_series}, it suffices to find, for all $v
\not \in T$, approximations $\psi_v^{\textup{reg}}(s)$ and
$\psi_v^{\textup{irreg}}(s)$ so that, locally,
$$
\zeta_{\SL_3(\smallgri_v)}^{\textup{reg}}(s) \ll
\psi_v^{\textup{reg}}(s) \quad \text{and} \quad
\zeta_{\SL_3(\smallgri_v)}^{\textup{irreg}}(s) \ll
\psi_v^{\textup{irreg}}(s)
$$
and, globally, the abscissa of convergence of the sum $\sum_{v \not
  \in T} \left( \psi_v^{\textup{reg}}(s) +
  \psi_v^{\textup{irreg}}(s) \right)$ is at most $1$.

\begin{table}
  \centering
  \caption{Adjoint orbits in $\spl_3(\mathbb{F}_q)$ under
    the action of $\GL_3(\F_q)$}
  \label{table1}
  \begin{tabular}{|l|l|l|l|l|}
    \hline
    type &  & number of orbits & size of each orbit & total
    number \\
    \hline\hline
    0 & & $1$ & $1$ & $1$ \\
    1 & reg.\ & $1$ & $(q^3-1)(q^2-1)q$ & $\approx q^6$ \\
    2 & irreg.\ & $1$ & $(q^3-1)(q+1)$ & $\approx q^4$ \\
    3 & irreg.\ & $q-1$ & $(q^2+q+1)q^2$ & $\approx q^5$ \\
    4a & reg.\ & $(q-1)(q-2)/6$ & $(q^2+q+1)(q+1)q^3$ & $\approx q^8$ \\
    4b & reg.\ & $(q-1)q/2$ & $(q^3-1)q^3$ & $\approx q^8$ \\
    4c & reg.\ & $(q^2-1)/3$ & $(q+1)(q-1)^2q^3$ & $\approx q^8$ \\
    5 & reg.\ & $(q-1)$ & $(q^3-1)(q+1)q^2$ & $\approx q^7$ \\
    \hline
  \end{tabular}
\end{table}

\begin{table}
  \centering
  \caption{Centralisers in $\SL_3(\F_q)$ of elements of
    $\spl_3(\mathbb{F}_q)$ where $q = p^r$}
  \label{table2}
  \begin{tabular}{|l|l|l|}
    \hline
    type & & centraliser in $\SL_3(\F_q)$ \\
    \hline\hline
    0 & & $\SL_3(\F_q)$ \\
    1 & reg.\ & $\mu_3(\F_q) \times \F_q^+ \times \F_q^+ \cong
    C_{\gcd(q-1,3)} \times C_p^r \times C_p^r$ \\
    2 & irreg.\ & $\F_q^* \ltimes \mathcal{H}(\F_q) \cong C_{q-1}
    \ltimes \left( (C_p^r \times C_p^r) \ltimes C_p^r \right)$ \\
    3 & irreg.\ & $\GL_2(\F_q)$ \\
    4a & reg.\ & $\F_q^* \times \F_q^* \cong C_{q-1} \times C_{q-1}$ \\
    4b & reg.\ & $\F_{q^2}^* \cong C_{q^2 -1}$ \\
    4c & reg.\ & $\ker(N_{\F_{q^3} \vert \F_q}) \cong C_{q^2+q+1}$
    \\
    5 & reg.\ & $\F_q^* \times \F_q^+ \cong C_{q-1} \times C_p^r$ \\
    \hline
  \end{tabular}
\end{table}

In order to apply approximative Clifford theory, we require an
overview of the elements in $\spl_3(\mathbb{F}_q)$ up to the adjoint
action of the group $\GL_3(\F_q)$.  We distinguish eight different
types, labelled $0$, $1$, $2$, $3$, $4$a, $4$b, $4$c,~$5$.  The total
number of elements of each type and the isomorphism types of their
centralisers in $\SL_3(\F_q)$ are summarised in Tables~\ref{table1}
and~\ref{table2}; see Appendix~\ref{sec:aux_sl3} for a short
discussion.  In Table~\ref{table2}, $\mu_3(\F_q)$ is the group of
third roots of unity in $\F_q^*$, we denote by $\F_q^+$ the additive
group of the field $\F_q$, we write $\mathcal{H}(\F_q)$ for the
Heisenberg group over $\F_q$, and $\ker(N_{\F_{q^3} \vert \F_q})$ is
the multiplicative group of elements of norm $1$ in $\F_{q^3} \vert
\F_q$.

Let $v$ be a place of $\gfi$ with $v \not \in T$.  The series
$\psi_v^{\textup{reg}}(s)$ is to approximate the second summand in
\eqref{equ:zeta_decomp}.  It splits into five smaller parts
$$
\psi_v^{\textup{reg}}(s) = \psi_v^{\text{$1$}}(s) +
\psi_v^{\text{$4$a}}(s) + \psi_v^{\text{$4$b}}(s) +
\psi_v^{\text{$4$c}}(s) + \psi_v^{\text{$5$}}(s)
$$
which correspond -- via the diagram \eqref{equ:diagram} -- to elements
of types $1$, $4$a, $4$b, $4$c and $5$ in the finite Lie algebra
$\spl_3(\mathbb{F}_q)$.  Each of these summands will be a rational
function in $q_v$, $q_v^{-s}$ and $2^{-s}$, which we derive using
Clifford theory.  Corollary~\ref{cor:lift_spl} shows that for elements
of types $1$, $4$a, $4$b, $4$c and $5$ the centralisers in
$\SL_3(\F_{q_v})$ are isomorphic to the corresponding inertia group
quotients.  We also remark that types $4$a, $4$b and $4$c are somewhat
easier to deal with than the rest, since the corresponding inertia
group quotients are `tame'; they are cyclic of order coprime to $p$.
This implies that characters of these types extend from
$\SL_3^1(\gri_v)$ to their respective inertia groups; cf.\
Lemma~\ref{lem:translation-factor_neu}.


The series $\psi_v^{\textup{irreg}}(s)$ is to approximate the third
summand in \eqref{equ:zeta_decomp}.  It splits into two smaller parts
$$
\psi_v^{\textup{irreg}}(s) = \psi_v^{\text{$2$}}(s) +
\psi_v^{\text{$3$}}(s)
$$
which correspond -- via the diagram \eqref{equ:diagram} -- to elements
of types $2$ and $3$ in $\spl_3(\mathbb{F}_{q_v})$.  As before, each
of these summands will be a rational function in $q_v$, $q_v^{-s}$ and
$2^{-s}$, obtained by means of Clifford theory.
Corollary~\ref{cor:lift_spl} does not apply to elements of types $2$
and $3$ and the inertia group quotients are, in fact, frequently
smaller than the recorded centralisers in $\SL_3(\F_{q_v})$.

For every prime $p$ there are at most $\lvert \gfi : \Q \rvert$
prolongations $v$ to $\gfi$ of the $p$-adic valuation on $\Q$.
Moreover, for any prolongation $v$ of the $p$-adic valuation
associated to a prime $p$, the size $q_v$ of the corresponding residue
field satisfies $p \leq q_v \leq p^{\lvert \gfi : \Q \rvert}$.
Therefore, in order to show that the abscissa of convergence of the
sum $\sum_{v \not \in T} \left( \psi_v^{\textup{reg}}(s) +
\psi_v^{\textup{irreg}}(s) \right)$ is less than or equal to~$1$, it
will be enough to show that for each of the rational functions
$f(q,q^{-s}, 2^{-s})$ in $q$, $q^{-s}$ and $2^{-s}$ which we derive
below to bound the contributions from representations of the various
types and for each fixed exponent $t$ in the range $1 \leq t \leq
\lvert k:\Q \rvert$, the individual series $\sum_q f(q,q^{-s},2^{-s})$
converges for $\real(s) > 1$, where $q$ runs over all prime powers
$p^t$ with exponent $t$.  In practice, it will be enough to consider
the `worst' case $t=1$, i.e.\ $q$ running over all primes.

\subsubsection{}\label{subsubsec:7.3.4} We now work out the details
case by case.  Let $v$ be a place of $\gfi$ with $v \not \in T$.  In
considering the local situation at $v$, it is convenient to adopt a
simplified notation which suppresses explicit references to $v$.  We
write $\lri := \gri_v$, $\mfp := \mfp_v$ for the maximal ideal and
$\F_q := \lri / \mfp$ for the residue field, where $q := q_v$.
Furthermore, we write $\psi_\lri^{\textup{reg}}(s) :=
\psi_v^{\textup{reg}}(s)$, $\psi_\lri^{\textup{irreg}}(s) :=
\psi_v^{\textup{irreg}}(s)$, $\psi_\lri^{\text{$1$}}(s) :=
\psi_v^{\text{$1$}}(s)$, etc.

Set $G := \SL_3(\lri)$, $N := \SL_3^1(\lri)$ and $\gamma(q) := q^{-8}
\lvert G : N \rvert = (1-q^{-3})(1-q^{-2})$.  We use the generous
bound $\gamma(q) > 2^{-1}$.  In each case we need to multiply the
approximate contribution number from the last column of
Table~\ref{table1}, an approximation of the translation factor $\lvert
G:I_G(\theta)\rvert^{-1-s} \zeta_{G,  \theta}(s)$, in accordance
with \eqref{equ:hochheben}, and (an approximation of) the appropriate
series factor, appearing on one of the right hand sides of
\eqref{equ:series-factor}.  In the case of regular representations we
use the precise translation factor $(1- q^{2-3s})^{-1}$; for irregular
representations we use the generous estimate
$$ \frac{1 - q^{1-2s} - q^{2-3s} + q^{4-2s}}{(1 - q^{1-2s}) (1 -
  q^{2-3s})} \ll \frac{1 + q^{4-2s}}{(1 - q^{1-2s}) (1 - q^{2-3s})}.
$$
Lemma~\ref{lem:translation-factor_neu} plays a key role in deriving
a suitable estimate of the term $\zeta_{G,  \theta}(s)$.  We deal
with the different cases arising from Table~\ref{table1} in the
following order: $4$a, $4$b, $4$c, $5$, $2$, $1$, $3$.

\medskip

\noindent
\textit{Type $4$\textup{a}.}  According to Table~\ref{table1}, the
number of elements of type $4$a is bounded by~$q^8$.  Let $\theta \in
\Irr(N)$ be one of the representations of type $4$a.  Then, according
to Table~\ref{table2} and Corollary~\ref{cor:lift_spl}, the inertia
quotient is $I_G(\theta)/N \cong C_{q-1} \times C_{q-1}$.  Hence, by
Lemma~\ref{lem:translation-factor_neu}~(1), we can approximate
$\zeta_{G,  \theta} = \zeta_{I_G(\theta)/N}(s) = (q-1)^2$ by $q^2$
and bound $\lvert G : I_G(\theta) \rvert$ from below by $\gamma(q)q^6
> 2^{-1} q^6$.  In view of Remark~\ref{rem:translation_factors}, this
gives the approximate translation factor $q^2 \cdot (2^{-1}
q^6)^{-1-s} = 2^{1+s} q^{-4-6s}$.  The series factor is
$(1-q^{2-3s})^{-1}$.  Altogether elements of type $4$a contribute to
$\psi_\lri^{\textup{reg}}(s)$ with a summand
\begin{equation*}
  \psi_\lri^{\text{$4$a}}(s) = q^8 \cdot 2^{1+s} q^{-4-6s}
  \cdot (1-q^{2-3s})^{-1} = 2^{1+s} q^{4-6s} (1-q^{2-3s})^{-1}.
\end{equation*}
The sum $\sum_q 2^{1+s} q^{4-6s} (1-q^{2-3s})^{-1}$, where $q$ runs
over all primes, converges for $s \in \C$ with $\real(s) > 5/6$,
and hence for $s \in \C$ with $\real(s)>1$, as wanted.

\smallskip

\noindent
\textit{Types $4$\textup{b} and $4$\textup{c}.} The argument is very
similar to the one for type $4$\textup{a}.

\smallskip

\noindent
\textit{Type $5$.}  According to Table~\ref{table1}, the number of
elements of type $5$ is bounded by~$q^7$.  Let $\theta \in \Irr(N)$ be
one of the representations of type $5$.  Then according to
Table~\ref{table2} and Corollary~\ref{cor:lift_spl} the inertia group
quotient is $I_G(\theta)/N \cong C_{q-1} \times C_p^r$.  Let $P/N$
denote the Sylow-$p$ subgroup of $I_G(\theta)/N$, contained in a
Sylow-$p$ subgroup $H/N$ of $G/N$, and recall that our definition of
$T$ guarantees that the pro-$p$ groups $N$, $P = I_H(\theta)$ and $H$
are saturable and potent.  By Lemma~\ref{lem:q-power}, the saturable
Lie lattice corresponding to $P$ is an $\lri$-Lie
lattice. 
Thus, by Corollary~\ref{cor:powers_of_q}, the character degrees of $P$
and $N$ are powers of $q$.  By Lemma~\ref{lem:translation-factor_neu},
$\theta$ extends to $P$ so that $\zeta_{G,\theta}(s) = (q-1)q \ll
q^2$.  We bound $\lvert G : I_G(\theta) \rvert$ from below by
$\gamma(q) q^6 > 2^{-1}q^6$.  In view of
Remark~\ref{rem:translation_factors}, this gives the approximate
translation factor $q^2 \cdot (2^{-1} q^6)^{-1-s} = 2^{1+s}
q^{-4-6s}$.  The series factor is $(1-q^{2-3s})^{-1}$.  Altogether
elements of type $5$ contribute to $\psi_\lri^{\textup{reg}}(s)$ with
a summand
\begin{equation*}
  \psi_\lri^{\text{$5$}}(s) = q^7 \cdot 2^{1+s} q^{-4-6s} \cdot
  (1-q^{2-3s})^{-1} = 2^{1+s} q^{3-6s} (1-q^{2-3s})^{-1}.
\end{equation*}
The sum $\sum_q 2^{1+s} q^{3-6s} (1-q^{2-3s})^{-1}$, where $q$ runs
over all primes, converges for $s \in \C$ with $\real(s) > 4/6 =
2/3$, and hence for $s \in \C$ with $\real(s) > 1$, as wanted.

\smallskip

\noindent \textit{Type $2$.} According to Table~\ref{table1} there are
less than $2q^4$ contributing elements of type $2$.  Let $\theta \in
\Irr(N)$ be one of the representations of type $2$.  Without
determining the precise inertia group quotient, it suffices for our
estimate to know that $I_G(\theta)/N$ is a subgroup of the centraliser
group modulo $\mfp$, as described in Table~\ref{table2}.  This
centraliser is an extension of the Heisenberg group
$\mathcal{H}(\F_q)$ over $\F_q$ by the multiplicative group $\F_q^*
\cong C_{q-1}$, where the top group acts with a kernel of size
$\gcd(q-1,3)$ and with orbits of equal length $(q-1)/\gcd(q-1,3)$ on
the non-trivial elements of
$\mathcal{H}(\F_q)/[\mathcal{H}(\F_q),\mathcal{H}(\F_q)]$; see
\eqref{equ:centr_type_2}.
In order to make the argument independent of whether $\theta$ extends
to $I_G(\theta)$ or not, we include both possibilities in our
approximation.

First suppose that $\theta$ extends to $I_G(\theta)$.  Then we compute
the approximation as follows.  If the inertia group quotient contains
the full Heisenberg group $\mathcal{H}(\F_q)$, then it is isomorphic
to an extension of the Heisenberg group by a cyclic group of order $m
\mid (q-1)$.  Put $m_0 := \gcd(m,3)$ and $m_1 := m/m_0$.  The zeta
function of the Heisenberg group is known to be
$\zeta_{\mathcal{H}(\F_q)}(s) = q^2 + (q-1) q^{-s}$; see, for
instance, \cite{Is07}. A straightforward application of Clifford
theory, in the spirit of \eqref{equ:hochheben}, thus yields the
following estimate for the zeta function of the inertia group
quotient:
$$
\zeta_{I_G(\theta)/N}(s) \ll m + m_0 m_1^{-1-s} (q^2-1) + m
(q-1)q^{-s}.
$$
We bound the index $\lvert G:I_G(\theta) \rvert$ from below by
$\gamma(q) q^5 m^{-1} > 2^{-1} q^5 m^{-1}$.  Noting that
\begin{align*}
  (m & + m_0 m_1^{-1-s} (q^2-1) + m (q-1)q^{-s}) \cdot
  2^{1+s} q^{-5-5s} m^{1+s}  \\
  & \ll 2^{1+s} q^{-5-5s} (m^{2+s} + m_0^{2+s} q^2 + m^{2+s} q^{1-s}) \\
  & \ll 2^{1+s} (q^{-3-4s} + 9 q^{-3-4s} + q^{-2-5s}) \\
  & = 2^{1+s} (10 q^{-3-4s} + q^{-2-5s})
\end{align*}
and approximating the required series factor by $(1+q^{4-2s})
(1-q^{1-2s})^{-1} (1-q^{2-3s})^{-1}$, we obtain the contributing
summand
\begin{align*}
  \psi_\lri^{\text{$2$(a)}}(s) & = 2q^4 \cdot 2^{1+s} (10
  q^{-3-4s} + q^{-2-5s}) \cdot (1 + q^{4-2s}) (1-q^{1-2s})^{-1} (1-q^{2-3s})^{-1} \\
  & = 2^{2+s} (10 q^{1-4s} + 10 q^{5-6s} + q^{2-5s} + q^{6-7s}) (1-q^{1-2s})^{-1}
  (1-q^{2-3s})^{-1}.
\end{align*}
The sums $\sum_q 2^{2+s} q^{a-bs} (1-q^{1-2s})^{-1}
(1-q^{2-3s})^{-1}$, $(a,b) \in \{(1,4),(5,6),(2,5),(6,7)\}$, where $q$
runs over all primes, converge for $s \in \C$ with $\real(s) > 1$, as
wanted.

Now suppose that the inertia group quotient does not contain the full
Heisenberg group $\mathcal{H}(\F_q)$.  From Lemma~\ref{lem:q-power} we
conclude that $\lvert I_G(\theta):N \rvert \leq q^3$ and we
approximate $\zeta_{I_G(\theta)/N}(s)$ generously by $q^3$.  Bounding
$\lvert G : I_G(\theta) \rvert$ from below by $\gamma(q) q^5 > 2^{-1}
q^5$ and approximating the required series factor as before, we obtain
the contributing summand
\begin{align*}
  \psi_\lri^{\text{$2$(b)}}(s) & = 2q^4 \cdot q^3 \cdot 2^{1+s}
  q^{-5-5s}
  \cdot (1 + q^{4-2s})(1-q^{1-2s})^{-1} (1-q^{2-3s})^{-1} \\
  & = 2^{2+s} (q^{2-5s} + q^{6-7s}) (1-q^{1-2s})^{-1} (1-q^{2-3s})^{-1}.
\end{align*}
The sums $\sum_q 2^{2+s} q^{a-bs} (1-q^{1-2s})^{-1}
(1-q^{2-3s})^{-1}$, $(a,b) \in \{(2,5),(6,7)\}$, where $q$ runs over
all primes, converges for $s \in \C$ with $\real(s) > 1$, as wanted.

Next suppose that $\theta$ does not extend to $I_G(\theta)$.  We use
Lemma~\ref{lem:translation-factor_neu}(2), together with the estimate
$q^4$ for $\lvert I_G(\theta) : N \rvert$, the lower bound
$\gamma(q)q^4>2^{-1}q^4$ for $\lvert G : I_G(\theta) \rvert$ and the
same approximation of the required series factor as before, to obtain
the approximation
\begin{align*}
  \psi_\lri^{\text{$2$(c)}}(s) & = 2q^4 \cdot q^4 \cdot
  q^{-2-s} \cdot 2^{1+s} q^{-4-4s} \cdot (1+q^{4-2s}) (1-q^{1-2s})^{-1}
  (1-q^{2-3s})^{-1} \\
  & = 2^{2+s} (q^{2-5s}+q^{6-7s}) (1-q^{1-2s})^{-1} (1-q^{2-3s})^{-1}.
\end{align*}
The sums $\sum_q 2^{2+s} q^{a-bs} (1-q^{1-2s})^{-1}
(1-q^{2-3s})^{-1}$, $(a,b) \in \{(2,5),(6,7)\}$, where $q$ runs over
all primes, converge for $s \in\C$ with $\real(s) > 1$, as wanted.

\smallskip

\noindent
\textit{Type $1$.}  According to Table~\ref{table1}, the number of
elements of type~$1$ is bounded by $q^6$.  Let $\theta \in \Irr(N)$ be
one of the representations of type $1$.

First suppose that $\theta$ extends to $I_G(\theta)$.  Then, according
to Table~\ref{table2} and Corollary~\ref{cor:lift_spl}, the inertia
group quotient is $I_G(\theta)/N \cong C_{\gcd(q-1,3)} \times C_p^r
\times C_p^r$.  Hence we can approximate $\zeta_{I_G(\theta)/N}(s) =
\gcd(q-1,3) q^2$ by $2^2 q^2$ and bound $\lvert G : I_G(\theta)
\rvert$ from below by $\gamma(q)q^6 > 2^{-3}q^6$.  This gives the
approximate translation factor $2^2 q^2 \cdot (2^{-3} q^6)^{-1-s} =
2^{5+3s} q^{-4-6s}$.  The series factor is $(1-q^{2-3s})^{-1}$.
Altogether this gives a contributing summand
\begin{equation*}
  \psi_\lri^{\text{$1$(a)}}(s) = q^6 \cdot 2^{5+3s} q^{-4-6s}
  \cdot (1-q^{2-3s})^{-1} = 2^{5+3s} q^{2-6s} (1-q^{2-3s})^{-1}.
\end{equation*}
The sum $\sum_q 2^{5+3s} q^{2-6s} (1-q^{2-3s})^{-1}$, where $q$ runs
over all primes, converges for $s \in \C$ with $\real(s) > 2/3$, hence
for $s \in \C$ with $\real(s) > 1$, as wanted.

Next suppose that $\theta$ does not extend to $I_G(\theta)$.  Let
$P/N$ denote the Sylow-$p$ subgroup of $I_G(\theta)/N$, contained in a
Sylow-$p$ subgroup $H/N$ of $G/N$, and recall that our definition of
$T$ guarantees that the pro-$p$ groups $N$, $P = I_H(\theta)$ and $H$
are saturable and potent.  By Lemma~\ref{lem:q-power}, the saturable
Lie lattice corresponding to $P$ is an $\lri$-Lie
lattice. 
Thus, by Corollary~\ref{cor:powers_of_q}, the character degrees of $P$
and $N$ are powers of $q$.  Therefore, according to
Lemma~\ref{lem:translation-factor_neu}(2), we can use similar
estimates as above to approximate the contribution by
\begin{equation*}
  \psi_\lri^{\text{$1$(b)}}(s) = q^6 \cdot 2^2 q^2 \cdot q^{-2-s} \cdot
  2^{3+3s} q^{-6-6s} \cdot (1-q^{2-3s})^{-1}
  = 2^{5+3s} q^{-7s} (1-q^{2-3s})^{-1}.
\end{equation*}
The sum $\sum_q 2^{5+3s} q^{-7s} (1-q^{2-3s})^{-1}$, where $q$ runs
over all primes, converges for $s \in \C$ with $\real(s) > 2/3$,
hence for $s \in \C$ with $\real(s)>1$, as wanted.

\smallskip

\noindent
\textit{Type $3$.} According to Table~\ref{table1}, the number of
elements of type $3$ is bounded by~$q^5$.  The difficulty in this case
is that the precise inertia group quotient is typically significantly
smaller than $\GL_2(\F_q)$, the supergroup given in
Table~\ref{table2}.  Recall that $G = \SL_3(\lri)$ and write $\mfg :=
\spl_3(\lri)$.  There is no direct Lie correspondence between the
compact $p$-adic analytic group $G$ and the $\lri$-Lie lattice $\mfg$,
but we will use an approximate correspondence.  An irreducible
character $\theta$ of $N = \SL_3^1(\lri)$ is represented by a Kirillov
orbit $\Omega = \omega^N$ in $\widehat{\mfn}$, where $\mfn :=
\spl_3^1(\lri) = \mfp \mfg$.  By means of the normalised Killing form,
the inertia group $I_G(\theta)$ can be described in terms of the
centraliser $\Cen_G(x + \mfp^n \mfn)$, formed with respect to the
adjoint action of $G$, for a suitable $x \in \mfn^*$ and level $n \in
\N$.  The groups and Lie lattices involved, i.e.\ $G$, $\mfg$, their
respective subobjects and the relevant actions, naturally embed into
the $\lri$-order $A := \Mat_3(\lri)$.  We are to determine
$$
I_G(\theta)/N = \Cen_G(x + \mfp^n \mfn) N/N,
$$
equivalently its image $H$ under the natural inclusion into the
$\F_q$-vector space
$$
V := (\Cen_A (x + \mfp^n \mfn) + \mfp A) / \mfp A.
$$
As we are interested in the case where $\theta$ is of type~$3$, we
may assume without loss of generality that
\begin{equation}\label{equ:representative_mod_p}
  x + \mfp \mfn =
  \begin{pmatrix}
    \lambda & & \\
    & \lambda & \\
    & & -2\lambda
  \end{pmatrix}
  + \mfp \mfn \qquad \text{where $\lambda \in \mfp \setminus \mfp^2$};
\end{equation}
cf.\ \eqref{equ:centr_type_3}.  In this situation $V$ is a subspace of
the $5$-dimensional $\F_q$-vector space
$$
\left\{
  \begin{pmatrix}
    B &   \\
    & a
  \end{pmatrix}
  \mid B \in \Mat_2(\F_q) \text{ and } a \in \F_q \right\}.
$$
Recalling that the characteristic of the residue field $\F_q$ is
not $3$, we observe that
\begin{itemize}
\item the $\F_q$-subspace $V_0 := (\Cen_\mfg(x + \mfp^n \mfn) + \mfp
  A)/\mfp A$ of $V$, consisting of cosets of matrices of trace~$0$,
  has co-dimension $1$;
\item the relevant subset $H$ of $V$, consisting of cosets of matrices
  of determinant~$1$, has index at least $(q-1)/\gcd(q-1,3) \geq
  \lfloor (q-1)/3 \rfloor$ in the group formed by all cosets of
  invertible matrices in $V$.
\end{itemize}

First suppose that $V$ has $\F_q$-dimension at least $4$.  Then
$\dim_{\F_q}(V_0) \geq 3$ and, considering the image of $V_0 \cong
(\Cen_\mfg(x + \mfp^n \mfn) + \mfp \mfg) / \mfp \mfg$ under
multiplication by the uniformiser $\pi$ in $(\Cen_\mfn(x + \mfp^n
\mfn) + \mfp \mfn) / \mfp \mfn$, we conclude that $\Cen_\mfn(x +
\mfp^n \mfn)$, regarded as an $\lri$-submodule of $\mfn$, has
elementary divisors
$(\lri,\lri,\lri,\lri,\mfp^n,\mfp^n,\mfp^n,\mfp^n)$.  Thus $\lvert
\mfn : \Cen_\mfn(x + \mfp^n \mfn) \rvert = q^{4n}$.

Our analysis in Section~\ref{subsec:sl3princ} shows that $x + \mfp^n
\mfn$ corresponds to a point modulo $\mfp^n$ on the variety
$\mathcal{W}_1$ defined by the set $F_3(\mathbf{Y})$ of
polynomials which controls the integrand of \eqref{equ:sl3
integral}.  These are rare; if we fix $\lambda$ in
\eqref{equ:representative_mod_p} there are $q^{5n}$ lifts modulo
$\mfp^n$.  The resulting `series factor' is $\sum_{n = 0}^\infty
q^{5n} q^{-4n(s+2)/2} = 1/(1-q^{1-2s})$, and the inertia group
quotient $I_G(\theta)/N$ is isomorphic to $\GL_2(\F_q)$ so that
$$
\lvert G : I_G(\theta) \rvert = \lvert \SL_3(\F_q) \rvert \lvert
\GL_2(\F_q) \rvert^{-1} = (q^2+q+1) q^2 \geq q^4.
$$
Let $P/N$ denote the Sylow-$p$ subgroup of $I_G(\theta)/N$, contained
in a Sylow-$p$ subgroup $H/N$ of $G/N$. The group $P/N$ is isomorphic
to the additive group $\F_q^+$.  Recall that our definition of $T$
guarantees that the pro-$p$ groups $N$, $P = I_H(\theta)$ and $H$ are
saturable and potent.  By Lemma~\ref{lem:q-power}, the saturable Lie
lattice corresponding to $P$ is an $\lri$-Lie lattice.
Thus, by Corollary~\ref{cor:powers_of_q}, the character degrees of $P$
and $N$ are powers of $q$.  By
Lemma~~\ref{lem:translation-factor_neu}, the character $\theta$
extends to $P$ so that
\begin{equation}\label{equ:GL2}
  \begin{split}
    \zeta_{G,\theta}(s) & = \zeta_{\GL_2(\F_q)}(s) \\
    & = (q-1) \left( 1 + q^{-s} + (q-2)/2 \cdot (q+1)^{-s} + q/2 \cdot
      (q-1)^{-s} \right) \\
    & \ll q + 2^s q^{2-s}.
  \end{split}
\end{equation}
Thus the contributing summand is
\begin{align*}
  \psi_\lri^{\text{$3$(a)}}(s) & = q^5 \cdot (q + 2^s q^{2-s}) \cdot
  q^{-4-4s} \cdot (1-q^{1-2s})^{-1} \\ & = (q^{2-4s} + 2^s q^{3-5s})
  (1-q^{1-2s})^{-1}.
\end{align*}
The sums $\sum_q q^{2-4s} (1-q^{1-2s})^{-1}$ and $\sum_q 2^s
q^{3-5s}(1-q^{1-2s})^{-1}$, where $q$ runs over all primes, converge
for $s \in \C$ with $\real(s) > 4/5$, hence for $s \in \C$ with
$\real(s)>1$ as wanted.

Now suppose that the dimension of $V$ is at most $3$.  Then
$$
\lvert I_G(\theta) : N \rvert = \lvert H \rvert \leq (q^3-1) \lfloor
(q-1)/3 \rfloor^{-1} \leq 6 q^2,
$$
and we approximate $\zeta_{I_G(\theta)/N}(s)$ generously by $2^3 q^2$.
Bounding $\lvert G : I_G(\theta) \rvert$ from below by $\gamma(q)
2^{-3} q^6 > 2^{-4} q^6$ and approximating the required series factor
by $(1+q^{4-2s}) (1-q^{1-2s})^{-1} (1-q^{2-3s})^{-1}$, we obtain the
contribution
\begin{align*}
  \psi_\lri^{\text{$3$(b)}}(s) & = q^5 \cdot 2^3 q^2 \cdot 2^{4+4s}
  q^{-6-6s} \cdot (1+q^{4-2s}) (1-q^{1-2s})^{-1} (1-q^{2-3s})^{-1}
  \\ & = 2^{7+4s} (q^{1-6s} + q^{5-8s})(1-q^{1-2s})^{-1}
  (1-q^{2-3s})^{-1}.
\end{align*}
The sums $\sum_q 2^{7+4s} q^{a-bs} (1-q^{1-2s})^{-1}
(1-q^{2-3s})^{-1}$, $(a,b) \in \{(1,6),(5,8)\}$, where $q$ runs over
all primes, converge for $s \in \C$ with $\real(s) > 3/4$, hence for
$s \in \C$ with $\real(s) > 1$ as wanted.

\medskip

This finishes the proof of Proposition~\ref{pro:second_product}.
Together with Proposition~\ref{pro:first_product} and the reductions
leading up to these two propositions, this also completes the proof of
Theorem~\ref{thmABC:lalu} in the case of inner forms.



\subsection{Arithmetic groups of type ${}^2
  \! A_2$}\label{subsec:abscissa_outer}

In this section we prove the assertion of Theorem~\ref{thmABC:lalu}
for arithmetic subgroups pertaining to outer forms of $A_2$.
We work in the same setup as described at the beginning of
Section~\ref{subsec:abscissa_inner}, except that we now assume that
the $S$-arithmetic group $\Gamma = \mathbf{G}(\gri_S)$ is an outer
form, i.e.\ of type ${}^2 \! A_2$.  Again, it suffices to specify a
suitable finite set $T$ of places of $\gfi$, with $S \subseteq T$,
such that the abscissa of convergence of the product $\prod_{v \not
\in T} \zeta_{\Gamma_v}(s)$ is equal to $1$.  The argument is going to
be very similar to the one given for inner forms.

\subsubsection{} In this section, it is convenient to write
$\SU_3(\gri_v)$ for the group $\SU_3(\Lri,\gri_v)$, where $\Lri$ is an
unramified quadratic extension of $\gri_v$.  Similarly, we write
$\SU_3(\F_q)$ for the finite group $\SU_3(\F_{q^2},\F_q)$ and
$\su_3(\F_q)$ for the finite Lie algebra $\su_3(\F_{q^2},\F_q)$ over a
residue field $\F_q$, etc.  Let $T$ be the set of all places $v$ of
$\gfi$ such that one of the following holds: (i) $v \in S$, (ii) $v$
is dyadic or triadic, i.e.\ $\gri_v$ has residue field characteristic
$2$ or $3$, (iii) $\Gamma_v$ is not isomorphic to $\SL_3(\gri_v)$ or
$\SU_3(\gri_v)$, (iv) there is a pro-$p$ subgroup $H$ of
$\SL_3(\gri_v)$ containing $\SL_3^1(\gri_v)$, or of $\SU_3(\gri_v)$
containing $\SU_3^1(\gri_v)$, which fails to be saturable or potent.
Then, similarly as in Section~\ref{subsubsec:7.3.1}, one checks
readily that $T$ is finite.  In particular, the discussion in
Appendix~\ref{subsec:algebraic/arithmetic_groups} shows that (iii)
singles out only finitely many places and, furthermore, that
$V_\textup{SL} := \{ v \not \in T \mid \Gamma_v \cong \SL_3(\gri_v)
\}$ and $V_\textup{SU} := \{ v \not \in T \mid \Gamma_v \cong
\SU_3(\gri_v) \}$ have positive Dirichlet density within the set of
all places of $\gfi$.

We claim that the abscissa of convergence of the product
$$
\prod_{v \not \in T} \zeta_{\Gamma_v}(s) = \prod_{v \in V_\textup{SL}}
\zeta_{\SL_3(\smallgri_v)}(s) \prod_{v \in V_\textup{SU}}
\zeta_{\SU_3(\smallgri_v)}(s)
$$
is equal to $1$.  Again we take an approximative approach.  The first
product on the right hand side has abscissa of convergence equal to
$1$ by the arguments given in Section~\ref{subsec:abscissa_inner}.
(For this we make use of the fact that for a positive proportion of
primes $p$ there exists a prolongation $v \in V_\textup{SL}$ such that
$p = q_v$.  This follows from the Chebotarev Density Theorem, applied
to the quadratic extension $\Gfi$ of $\gfi$ which appears in the
description of $\mathbf{G}(\gfi)$ in
Appendix~\ref{subsec:algebraic/arithmetic_groups}.)  It suffices to
show that the abscissa of convergence of the product
$$
\phi(s) := \prod_{v \in V_\textup{SU}} \zeta_{\SU_3(\smallgri_v)}(s)
$$
is less than or equal to $1$.

Similarly to \eqref{equ:zeta_decomp}, for each $v \in V_\textup{SU}$
there is a natural decomposition
\begin{equation}\label{equ:zeta_decomp_SU}
  \zeta_{\SU_3(\smallgri_v)}(s) = \left( 1 +
    \zeta_{\SU_3(\smallgri_v)}^{\textup{triv}}(s) \right) +
  \zeta_{\SU_3(\smallgri_v)}^{\textup{reg}}(s) +
  \zeta_{\SU_3(\smallgri_v)}^{\textup{irreg}}(s)
\end{equation}
involving three `smaller' Dirichlet generating functions.  These
components enumerate representations which, after restriction to
$\SU_3^1(\gri_v)$, involve only the trivial, only regular and only
irregular representations respectively; recall the terminology
introduced in Section~\ref{subsec:sl3princ}.  Similarly to
Sections~\ref{subsubsec:7.3.3} and \ref{subsubsec:7.3.4}, we will
approximate the Dirichlet series
$\zeta_{\SU_3(\smallgri_v)}^{\textup{reg}}(s)$ and
$\zeta_{\SU_3(\smallgri_v)}^{\textup{irreg}}(s)$ by suitable rational
functions $\psi_v^{\textup{reg}}(s)$ and $\psi_v^{\textup{irreg}}(s)$
in $q_v$, $q_v^{-s}$ and $2^{-s}$, where $q_v = \lvert \gri_v : \mfp_v
\rvert$ denotes the size of the residue field $\gri_v / \mfp_v$ at
$v$.  Based on these approximations, an argument akin to the one given
in Section~\ref{subsec:abscissa_inner} shows that the abscissa of
convergence of $\phi(s)$ is less than or equal to~$1$.  This in turn
will yield a proof of Theorem~\ref{thmABC:lalu} for outer forms.

\subsubsection{} The first summand in \eqref{equ:zeta_decomp_SU},
viz.\ $1 + \zeta_{\SU_3(\smallgri_v)}^{\textup{triv}}(s)$, arises from
representations of $\SU_3(\gri_v)$ which are trivial on
$\SU_3^1(\gri_v)$.  Thus $1 +
\zeta_{\SU_3(\smallgri_v)}^{\textup{triv}}(s) =
\zeta_{\SU_3(\F_{q_v})}(s)$, the zeta function of the finite classical
group $\SU_3(\F_{q_v})$.

\begin{pro}\label{pro:first_product_SU}
  The abscissa of convergence of the Euler product
  $$
  \prod_{v \in V_\textup{SU}} \left( 1 +
    \zeta_{\SU_3(\smallgri_v)}^{\textup{triv}}(s) \right) = \prod_{v
    \in V_\textup{SU}} \zeta_{\SU_3(\F_{q_v})}(s)
  $$
  is less than or equal to $1$.
\end{pro}

\begin{proof}
  Let $q$ be a prime power, not divisible by $3$.  Explicit formulae
  for the zeta function of the finite classical group $\SU_3(\F_q)$
  can be computed, for instance, using Deligne-Lusztig Theory;
  cf.\ \cite{LuXX}.  One needs to separate the cases $q \equiv_3 2$
  and $q \equiv_3 1$.  Indeed, for $q \equiv_3 2$ one has
  \begin{align*}
    \zeta_{\SU_3(\F_q)}(s) & = 1 + 1 \cdot
    (q^2-q)^{-s} + q \cdot (q^2-q+1)^{-s} \\
    & \qquad + 6 \cdot
    ((q-1)(q+1)^2/3)^{-s} + 3 \cdot ((q^2-q+1)(q-1)/3)^{-s} \\
    & \qquad + 3^{-1}(q+1)(q-2) \cdot ((q-1)(q+1)^2)^{-s} +
    2^{-1}(q+1)(q-2) \cdot
    (q^3+1)^{-s} \\
    & \qquad + 1 \cdot q^{-3s} + q \cdot (q^3-q^2+q))^{-s} \\
    & \qquad + 6^{-1}(q+1)(q-2) \cdot ((q^2-q+1)(q-1))^{-s}
  \end{align*}
  and for $q \equiv_3 1$ one has
  \begin{align*}
    \zeta_{\SU_3(\F_q)}(s) & = 1 + 1 \cdot
    (q^2-q)^{-s} + q \cdot (q^2-q+1)^{-s} \\
    & \qquad + 3^{-1}(q^2 - q) \cdot ((q-1)(q+1)^2)^{-s} +
    2^{-1}(q+1)(q-2) \cdot (q^3 +1)^{-s} \\
    & \qquad + 1 \cdot q^{-3s}
    + q \cdot (q^3-q^2+q)^{-s} \\
    & \qquad + 6^{-1} (q^2 -q) \cdot ((q^2-q+1)(q-1))^{-s}.
  \end{align*}

  Based on these formulae we can prove our assertion.  Indeed, $\lvert
  k:\Q \rvert$ provides a uniform upper bound for the number of
  valuations $v \not \in T$ which prolong any fixed $p$-adic valuation
  on $\Q$.  Moreover, for each $v$ which prolongs the $p$-adic
  valuation associated to a prime $p$ we have $p \leq q_v = \lvert
  \gri_v:\mfp_v \rvert$.  The Chebotarev Density Theorem guarantees
  that for a positive proportion of primes $p$ there exists a
  prolongation $v$ such that $p = q_v$.  In view of the explicit
  formulae for $\zeta_{\SU_3(\F_q)}(s)$ recorded above, this implies
  that the abscissa of convergence of the product $\prod_{3 \nmid v}
  \zeta_{\SU_3(\F_{q_v})}(s)$ is equal to the abscissa of convergence
  of the sum $\sum_p (p^{1-2s} + p^{2-3s})$, which is equal to~$1$.
  This gives an upper bound for the abscissa of convergence of the
  partial product $\prod_{v \in V_\textup{SU}}
  \zeta_{\SU_3(\F_{q_v})}(s)$.
\end{proof}

\subsubsection{} It remains to bound the abscissa of convergence of
the factor arising from the last two summands in
\eqref{equ:zeta_decomp_SU}.

\begin{pro}\label{pro:second_product_SU}
  The abscissa of convergence of the Euler product
  $$
  \prod_{v \in V_\textup{SU}} \left( 1+
    \zeta_{\SU_3(\smallgri_v)}^{\textup{reg}}(s) +
    \zeta_{\SU_3(\smallgri_v)}^{\textup{irreg}}(s) \right)
  $$
  is less than or equal to $1$.
\end{pro}

The proof of this proposition is based on a similar analysis to the
one in Section~\ref{subsubsec:7.3.4}.  By
Lemma~\ref{lem:generating_series}, it suffices to find, for all $v \in
V_\textup{SU}$, approximations $\psi_v^{\textup{reg}}(s)$ and
$\psi_v^{\textup{irreg}}(s)$ so that, locally,
$$
\zeta_{\SU_3(\gri_v)}^{\textup{reg}}(s) \ll \psi_v^{\textup{reg}}(s) \quad
\text{and} \quad \zeta_{\SU_3(\gri_v)}^{\textup{irreg}}(s) \ll
\psi_v^{\textup{irreg}}(s)
$$
and, globally, the abscissa of convergence of the sum $\sum_{v \in
  V_\textup{SU}} \left( \psi_v^{\textup{reg}}(s) +
  \psi_v^{\textup{irreg}}(s) \right)$ is at most $1$.

\begin{table}
  \centering
  \caption{Adjoint orbits in $\su_3(\F_q)$ under the action of
    $\GU_3(\F_q)$}
  \label{table3}
  \begin{tabular}{|l|l|l|l|l|}
    \hline
    type &  & number of orbits & size of each orbit & total
    number \\
    \hline\hline
    0 & & $1$ & $1$ & $1$ \\
    1 & reg.\ & $1$ & $(q^3+1)(q^2-1)q$ & $\approx q^6$ \\
    2 & irreg.\ & $1$ & $(q^3+1)(q-1)$ & $\approx q^4$ \\
    3 & irreg.\ & $q-1$ & $(q^2-q+1)q^2$ & $\approx q^5$ \\
    4a & reg.\ & $(q-1)(q-2)/6$ & $(q^2-q+1)(q-1)q^3$ & $\approx q^8$ \\
    4b & reg.\ & $(q-1)q/2$ & $(q^3+1)q^3$ & $\approx q^8$ \\
    4c & reg.\ & $(q^2-1)/3$ & $(q^2-1)(q+1)q^3$ & $\approx q^8$ \\
    5 & reg.\ & $(q-1)$ & $(q^3+1)(q-1)q^2$ & $\approx q^7$ \\
    \hline
  \end{tabular}
\end{table}

\begin{table}
  \centering
  \caption{Centralisers in $\SU_3(\F_q)$ of elements of
    $\su_3(\mathbb{F}_q)$ where $q = p^r$}
  \label{table4}
  \begin{tabular}{|l|l|l|}
    \hline
    type & & centraliser in $\SU_3(\F_q)$ \\
    \hline\hline
    0 & & $\SU_3(\F_q)$ \\
    1 & reg.\ & $( \mu_3(\F_{q^2}) \cap \ker(N_{\F_{q^2} \vert \F_q})
    ) \times \F_q^+ \times \F_q^+ \cong
    C_{\gcd(q+1,3)} \times (C_p^r \times C_p^r)$ \\
    2 & irreg.\ & $\ker(N_{\F_{q^2} \vert \F_q}) \ltimes
    \mathcal{H}(\F_q) \cong C_{q+1} \ltimes (C_p^r \times
    C_p^r) \ltimes C_p^r$ \\
    3 & irreg.\ & $\GU_2(\F_q)$ \\
    4a & reg.\ & $\ker(N_{\F_{q^2} \vert \F_q}) \times
    \ker(N_{\F_{q^2} \vert \F_q}) \cong C_{q+1} \times C_{q+1}$ \\
    4b & reg.\ & $\F_{q^2}^* \cong C_{q^2 -1}$ \\
    4c & reg.\ & $\ker(N_{\F_{q^6}
      \vert \F_{q^3}}) \cap \ker(N_{\F_{q^6} \vert \F_{q^2}}) \cong C_{q^2-q+1}$ \\
    5 & reg.\ & $\ker(N_{\F_{q^2} \vert \F_q}) \ltimes \F_q^+ \cong
    C_{q+1} \times C_p^r$ \\
    \hline
  \end{tabular}
\end{table}

In order to apply approximative Clifford theory, we require an
overview of the elements in $\su_3(\F_q)$ up to the adjoint action of
the group $\GU_3(\F_q)$.  We distinguish eight different types,
labelled $0$, $1$, $2$, $3$, $4$a, $4$b, $4$c,~$5$.  The total number
of elements of each type and the isomorphism types of their
centralisers in $\SU_3(\F_q)$ are summarised in Tables~\ref{table3}
and \ref{table4}; see Appendix~\ref{sec:aux_su3} for a short
discussion.  The notation used in Table~\ref{table4} is similar to the
one in Table~\ref{table2}: $\mu_3(\F_{q^2})$ is the group of third
roots of unity in $\F_{q^2}^*$, we denote by $\mathcal{H}(\F_q)$ the
Heisenberg group over $\F_q$, etc.

Let $v$ be a place of $\gfi$ with $v \in V_\textup{SU}$.  The series
$\psi_v^{\textup{reg}}(s)$ is to approximate the second summand in
\eqref{equ:zeta_decomp_SU}.  It splits into smaller parts
\begin{align*}
  \psi_v^{\textup{reg}}(s) & = \psi_v^{\text{$1$}}(s) +
  \psi_v^{\text{$4$a}}(s) + \psi_v^{\text{$4$b}}(s) +
  \psi_v^{\text{$4$c}}(s) + \psi_v^{\text{$5$}}(s) \\
  & = \psi_v^{\text{$1$(a)}}(s) + \psi_v^{\text{$1$(b)}}(s) +
  \psi_v^{\text{$4$a}}(s) + \psi_v^{\text{$4$b}}(s) +
  \psi_v^{\text{$4$c}}(s) + \psi_v^{\text{$5$}}(s),
\end{align*}
which correspond to elements of types $1$, $4$a, $4$b, $4$c and $5$ in
the finite Lie algebra $\su_3(\F_q)$ and where
$\psi_v^{\text{$1$}}(s)$ consists of two summands, following a case
distinction similar to the analogous one in
Section~\ref{subsubsec:7.3.4}.  Each of the summands involved will be
a rational function in $q_v$, $q_v^{-s}$ and $2^{-s}$, which we derive
using Clifford theory.  The series $\psi_v^{\textup{irreg}}(s)$ is to
approximate the third summand in \eqref{equ:zeta_decomp_SU}.  It
splits into smaller parts
\begin{align*}
  \psi_v^{\textup{irreg}}(s) & = \psi_v^{\text{$2$}}(s) +
  \psi_v^{\text{$3$}}(s) \\
  & = \psi_v^{\text{$2$(a)}}(s) + \psi_v^{\text{$2$(b)}}(s) +
  \psi_v^{\text{$2$(c)}}(s) + \psi_v^{\text{$3$(a)}}(s) +
  \psi_v^{\text{$3$(b)}}(s),
\end{align*}
which correspond to elements of types $2$ and $3$ in $\su_3(\F_{q_v})$
and where $\psi_v^{\text{$2$}}(s)$ and $\psi_v^{\text{$3$}}(s)$ divide
into smaller summands, following a case distinction similar to the
analogous one in Section~\ref{subsubsec:7.3.4}.  Again, each of these
summands will be a rational function in $q_v$, $q_v^{-s}$ and
$2^{-s}$, which can be obtained by means of Clifford theory.

\subsubsection{} We now work out the details case by case.  Let $v$ be
a place of $\gfi$ with $v \in V_\textup{SU}$.  In describing the local
situation at $v$, it is convenient to adopt, as in
Section~\ref{subsubsec:7.3.4}, a simplified notation which suppresses
explicit references to $v$.  We write $\lri := \gri_v$, $\mfp :=
\mfp_v$ for the maximal ideal and $\F_q := \lri / \mfp$ for the
residue field, where $q := q_v$.  We denote by $\Lri$ an unramified
quadratic extension of $\lri$, with non-trivial automorphism $\sigma$.
Furthermore, we write $\psi_\lri^{\text{$1$}}(s) :=
\psi_v^{\text{$1$}}(s)$, etc.  We also continue to use the simplified
notation $\GU_3(\lri) := \GU_3(\Lri,\lri)$, etc.

One sets $G := \SU_3(\lri)$, $N := \SU_3^1(\lri)$ and $\gamma(q) :=
q^{-8} \lvert G : N \rvert = (1+q^{-3})(1-q^{-2})$.  We use the
generous bound $\gamma(q) > 2^{-1}$.  In each case we need to multiply
the approximate contribution number from the last column of
Table~\ref{table3}, an approximation of the translation factor $\lvert
G:I_G(\theta)\rvert^{-1-s} \zeta_{G,  \theta}(s)$, in accordance
with \eqref{equ:hochheben}, and (an approximation of) the appropriate
series factor, appearing on one of the right hand sides of
\eqref{equ:series-factor}.  As in Section~\ref{subsubsec:7.3.4}, in
the case of regular representations we use the precise translation
factor $(1- q^{2-3s})^{-1}$; for irregular representations we use the
generous approximative factor $(1+q^{4-2s}) (1 - q^{1-2s})^{-1} (1 -
q^{2-3s})^{-1}$.  As the detailed analysis is very similar to the one
carried out for inner forms, we only give a short indication of the
necessary modifications and we list the resulting approximative
Dirichlet generating functions.  The most involved case is once more
the one of type $3$; here our adaptation makes use of the Cayley map.
We deal with the different cases arising from Table~\ref{table3} in
the following order: $4$a, $4$b, $4$c, $5$, $2$, $1$, $3$.

\medskip

\noindent
\textit{Types $4$\textup{a}, $4$\textup{b}, $4$\textup{c}.}  The
changes for type $4$ are easy to make.  Using
Corollary~\ref{cor:lift_su}, instead of Corollary~\ref{cor:lift_spl},
we compute
$$
\psi_\lri^{\text{$4$a}}(s) = q^8 \cdot 2^{3+2s} q^{-4-6s} \cdot
(1-q^{2-3s})^{-1} = 2^{3+2s} q^{4-6s} (1-q^{2-3s})^{-1},
$$
to control the contribution from elements of type $4$a, and one uses
similar estimates for types $4$b and $4$c.

\smallskip

\noindent
\textit{Type $5$.} Also the changes for type $5$ are straightforward.  Using
Corollary~\ref{cor:lift_su}, instead of Corollary~\ref{cor:lift_spl},
we compute
$$
\psi_\lri^{\text{$5$}}(s) = q^7 \cdot 2^{3+2s} q^{-4-6s} \cdot
(1-q^{2-3s})^{-1} = 2^{3+2s} q^{3-6s} (1-q^{2-3s})^{-1}
$$
to control the contribution from elements of type $5$.

\smallskip

\noindent
\textit{Type $2$.} We explain briefly how to adopt the argument for
elements of type $2$.  According to Table~\ref{table4}, the inertia
quotient $T_G(\theta)/N$ embeds into the semidirect product
$\ker(N_{\F_{q^2} \vert \F_q}) \ltimes \mathcal{H}(\F_q)$, and the
analysis beginning at \eqref{equ:su_centr_type_2} shows that the top
group $\ker(N_{\F_{q^2} \vert \F_q}) \cong C_{q+1}$ acts with kernel
of size $\gcd(q+1,3)$ and with orbits of equal length
$(q+1)/\gcd(q+1,3)$ on the non-trivial elements of
$\mathcal{H}(\F_q)/[\mathcal{H}(\F_q),\mathcal{H}(\F_q)]$.  Arguing as
before, we compute
\begin{align*}
  \psi_\lri^{\text{$2$(a)}}(s) & = q^4 \cdot 2^{3+2s} (10
  q^{-3-4s} + q^{-2-5s}) \cdot (1+q^{4-2s}) (1-q^{1-2s})^{-1} (1-q^{2-3s})^{-1} \\
  & = 2^{3+2s} (10 q^{1-4s} + 10 q^{5-6s} + q^{2-5s} + q^{6-7s})
  (1-q^{1-2s})^{-1}
  (1-q^{2-3s})^{-1}, \\
  \psi_\lri^{\text{$2$(b)}}(s) & = q^4 \cdot 2 q^3 \cdot 2^{2+2s}
  q^{-5-5s} \cdot (1+ q^{4-2s}) (1-q^{1-2s})^{-1} (1-q^{2-3s})^{-1} \\
  & = 2^{3+2s} (q^{2-5s} + q^{6-7s}) (1-q^{1-2s})^{-1} (1-q^{2-3s})^{-1}, \\
  \psi_\lri^{\text{$2$(c)}}(s) & = q^4 \cdot 2q^4 \cdot q^{-2-s} \cdot
  2^{2+2s} q^{-4-4s} \cdot (1+ q^{4-2s}) (1-q^{1-2s})^{-1}
  (1-q^{2-3s})^{-1} \\
  & = 2^{3+2s} (q^{2-5s} + q^{6-7s}) (1-q^{1-2s})^{-1}
  (1-q^{2-3s})^{-1}
\end{align*}
to control the contribution from elements of type $2$.

\smallskip

\noindent
\textit{Type $1$.}  The changes for type $1$ are easy to make.  Using
Corollary~\ref{cor:lift_su}, instead of Corollary~\ref{cor:lift_spl},
we compute the same approximative Dirichlet generating functions
\begin{align*}
  \psi_\lri^{\text{$1$(a)}}(s) & = q^6 \cdot 2^{3+s} q^{-4-6s}
  \cdot (1-q^{2-3s})^{-1} = 2^{3+s} q^{2-6s} (1-q^{2-3s})^{-1}, \\
  \psi_\lri^{\text{$1$(b)}}(s) & = q^6 \cdot 4 q^2 \cdot q^{-2-s}
  \cdot 2^{1+s} q^{-6-6s} \cdot (1-q^{2-3s})^{-1} = 2^{3+s} q^{-7s}
  (1-q^{2-3s})^{-1}.
\end{align*}
to control the contribution from elements of type $1$.

\smallskip

\noindent
\textit{Type $3$.} The argument for type~$3$ requires more substantial
modifications.  The number of elements is
bounded by $q^5$, and the difficulty is that the precise inertia group
quotient is typically significantly smaller than $\GU_2(\F_q)$, the
supergroup supplied by Table~\ref{table4}.  In
Section~\ref{subsubsec:7.3.4} where we dealt with inner forms, the
associative algebra $\Mat_3(\lri)$ was used to translate between the
compact analytic group $\SL_3(\lri)$ and the $\lri$-Lie lattice
$\spl_3(\lri)$.  In the situation at hand, there is no associative
algebra which connects $G = \SU_3(\lri)$ and $\mfg = \su_3(\lri)$ at
the level of $\lri$, but one can argue via the Cayley maps;
see~\cite[II.10 and VI.2]{We97}.  For any subset $Y \subseteq
\Mat_3(\Lri)$, we denote by $Y_\textup{gen}$ the set of elements of
$Y$ which do not have an eigenvalue congruent to $-1$ modulo~$\mfp$.
The Cayley maps $\Cay: \GU_3(\lri)_\textup{gen} \rightarrow
\gu_3(\lri)_\textup{gen}$ and $\cay: \gu_3(\lri)_\textup{gen}
\rightarrow \GU_3(\lri)_\textup{gen}$ which are each defined by the
mapping rule
\begin{equation}\label{equ:cayley}
y \mapsto (1-y)(1+y)^{-1} = (1+y)^{-1} (1-y)
\end{equation}
are easily seen to be mutual inverses of each other; cf.\
Appendix~\ref{sec:Cayley}.

As in Section~\ref{subsubsec:7.3.4} fix $x \in \mfn^*$ and a level $n
\in \N$ such that $I_G(\theta) = \Cen_G(x + \mfp^n \mfn) N$.  From the
definition of the Cayley maps it is clear that we have inclusions
$$
\begin{CD}
  \Cen_{\GU_3(\lri)}(x + \mfp^n \mfn)_\textup{gen} @>{\Cay}>>
  \Cen_{\gu_3(\lri)}(x + \mfp^n \mfn)_\textup{gen} @>\cay>>
  \Cen_{\GU_3(\lri)}(x + \mfp^n \mfn)_\textup{gen}.
\end{CD}
$$
We need to translate these inclusions into a quantitative assertion
about the cardinalities of the inertia group quotient $I_G(\theta)/N =
\Cen_{\SU_3(\lri)}(x + \mfp^n \mfn) N/N$ and the Lie lattice quotient
$(\Cen_{\su_3(\lri)}(x + \mfp^n \mfn) + \mfn)/\mfn$.

Elementary considerations, using $q \geq 5$, show that there exists
$\lambda \in \Lri^*$ with $\lambda + \lambda^\sigma = 0$ such that
$\gu_3(\lri) \subseteq \gu_3(\lri)_\textup{gen} \cup (\gu_3(\lri) +
\lambda \Id)_\textup{gen}$.  This implies that
$$
\lvert (\Cen_{\gu_3(\lri)}(x + \mfp^n \mfn) + \gu_3^1(\lri)) :
\gu_3^1(\lri) \rvert \leq 2 \cdot \lvert \Cen_{\GU_3(\lri)}(x + \mfp^n
\mfn) \GU_3^1(\lri) : \GU_3^1(\lri) \rvert.
$$
Similarly, there exist $\lambda_1, \ldots, \lambda_4 \in \Lri^*$ with
$\lambda_i \lambda_i^\sigma = 1$ such that $$\GU_3(\lri) \subseteq
\bigcup_{i=1}^4 (\lambda_i \GU_3(\lri))_\textup{gen}.$$  This implies
that
$$
4 \cdot \lvert \Cen_{\gu_3(\lri)}(x + \mfp^n \mfn) + \gu_3^1(\lri)) :
\gu_3^1(\lri) \rvert \geq \lvert \Cen_{\GU_3(\lri)}(x + \mfp^n \mfn)
\GU_3^1(\lri) : \GU_3^1(\lri) \rvert.
$$
At the same time we have
$$
  \lvert \Cen_{\gu_3(\lri)}(x + \mfp^n \mfn) + \gu_3^1(\lri) :
  \Cen_{\su_3(\lri)}(x + \mfp^n \mfn) + \gu_3^1(\lri)
  \rvert = q
$$
and
$$
q+1 \geq \lvert \Cen_{\GU_3(\lri)}(x + \mfp^n \mfn) \GU_3^1(\lri) :
\Cen_{\SU_3(\lri)}(x + \mfp^n \mfn) \GU_3^1(\lri) \rvert \geq
\frac{q+1}{\gcd(q+1,3)}.
$$
Altogether we thus obtain the generous estimates
$$
4^{-1} \leq \frac{\lvert \Cen_{\SU_3(\lri)}(x + \mfp^n \mfn) N/N
  \rvert}{\lvert (\Cen_{\su_3(\lri)}(x + \mfp^n \mfn) + \mfn)/\mfn \rvert}
\leq 12.
$$

With this preparation, we now distinguish two cases, as for elements
of type $3$ in the setting of inner forms.  In the first case, $x +
\mfp^n \mfn$ corresponds to a point modulo $\mfp^n$ on the variety
$\mathcal{W}_1$ defined by the set $F_3(\mathbf{Y})$ of
polynomials which controls the integrand.  These are rare and lead to
the effective `series factor' $1/(1 - q^{1-2s})$.  The inertia group
quotient is isomorphic to $\GU_2(\F_q)$ so that $\lvert G :
I_G(\theta) \rvert = q^2 (q^2 - q + 1) \geq 2^{-1} q^4$.  A similar
argument as before shows that $\theta$ extends to $I_G(\theta)$ so
that
\begin{equation}\label{equ:GU2}
  \begin{split}
    \zeta_{G,\theta}(s) & = \zeta_{\GU_2(\F_q)}(s) \\
    & = (q+1) \left( 1 + q^{-s} + (q-2)/2 \cdot (q+1)^{-s} + q/2 \cdot
      (q-1)^{-s} \right) \\
    & \ll 2q + 2^{1+s} q^{2-s}.
  \end{split}
\end{equation}
Thus the contributing summand is
\begin{align*}
  \psi_\lri^{\text{$3$(a)}}(s) & = q^5 \cdot (2q + 2^{1+s} q^{2-s})
  \cdot
  2^{1+s} q^{-4-4s} \cdot (1-q^{1-2s})^{-1} \\
  & = (q^{2-4s} + 2^{1+s} q^{3-5s}) (1-q^{1-2s})^{-1}.
\end{align*}

In the second case,
$$
\lvert I_G(\theta) : N \rvert \leq 12 \cdot \lvert
(\Cen_{\su_3(\lri)}(x + \mfp^n \mfn) + \mfn) : \mfn \rvert \leq 12
q^2,
$$
and we approximate $\zeta_{I_G(\theta)/N}(s)$ generously by $2^4 q^2$.
Bounding $\lvert G : I_G(\theta) \rvert$ from below by $\gamma(q)
2^{-4} q^6 > 2^{-5} q^6$ and approximating the required series factor
by $(1+q^{4-2s}) (1-q^{1-2s})^{-1} (1-q^{2-3s})^{-1}$, we obtain the
contribution
\begin{align*}
  \psi_\lri^{\text{$3$(b)}}(s) & = q^5 \cdot 2^4 q^2 \cdot 2^{5+5s}
  q^{-6-6s} \cdot (1+q^{4-2s}) (1-q^{1-2s})^{-1} (1-q^{2-3s})^{-1} \\
  & = 2^{9+5s} (q^{1-6s} + q^{5-8s}) (1-q^{1-2s})^{-1}
  (1-q^{2-3s})^{-1}.
\end{align*}

\medskip

As in in Section~\ref{subsubsec:7.3.4}, direct inspection shows that
the estimates provided for the contributions of the various types
justify Proposition~\ref{pro:second_product_SU}.  Together with
Proposition~\ref{pro:first_product_SU} and the reductions leading up
to these two propositions, this completes the proof of
Theorem~\ref{thmABC:lalu} in the case of outer forms.

\part*{Appendices}

\appendix

\section{Algebraic groups of type $A_2$}
\label{subsec:algebraic/arithmetic_groups}

Let $n \in \N$.  We are interested in arithmetic subgroups $\Gamma$ of
semisimple algebraic groups, in particular simple algebraic groups of
type $A_2$.  These groups arise in the following way.  Let
$\mathbf{G}$ be a connected, simply-connected semisimple algebraic
group defined over a number field $\gfi$, together with a fixed
$\gfi$-embedding into $\GL_N$ for some $N \in \N$.  Let $\gri_S$
denote the ring of $S$-integers in~$\gfi$, for a finite set $S$ of
places of $\gfi$ including all the archimedean ones.  We consider
groups $\Gamma$ which are commensurable to $\mathbf{G}(\gri_S) =
\mathbf{G}(\gfi) \cap \GL_n(\gri_S)$.

In this paper we are concerned with groups $\mathbf{G}$ which are
simple of type $A_n$, for the special case $n=2$.  These groups are
subdivided into groups of types ${}^1 \! A_n$ (pertaining to inner
forms) and ${}^2 \!  A_n$ (pertaining to outer forms) over $\gfi$.
Based on \cite[Propositions~2.17 and 2.18]{PlRa94}, we extract from
the general classification of simple algebraic groups of type $A_n$
%
%
the following information relevant to the special case of interest to
us.

The classification allows for the following groups of type ${}^1 \!
A_2$ over the number field $\gfi$:
\begin{enumerate}
\item $\mathbf{G}(\gfi) = \SL_3(\gfi)$,
\item $\mathbf{G}(\gfi) = \SL_1(D)$ where $D$ is a central division
  algebra of index $3$ over $\gfi$.
\end{enumerate}
If $\Gamma = \mathbf{G}(\gri_S)$ for either one of these
$\gfi$-defined algebraic groups, then for almost all places $v$, with
$v \not \in S$, the local group $\Gamma_v = \mathbf{G}(\gri_v)$ is
isomorphic to $\SL_3(\gri_v)$.  For the finitely many exceptional
places $v$ the local group $\Gamma_v$ is commensurable to either
$\SL_3(\gri_v)$ or $\SL_1(\Ldr)$ where $\Ldr$ is the (unique) division
algebra of index $2$ over $\gfi_v$.

Similarly, the classification allows for the following groups of type
${}^2 \!  A_2$ over the number field $\gfi$:
\begin{enumerate}
\item $\SU_3(\Gfi,f) = \{ g \in \GL_3(\Gfi) \mid \det(g) = 1 \text{
    and } g^\circ H g = H \}$ where $\Gfi = \gfi(\sqrt{\delta})$,
  $\delta \in \gri$ not a square, is a quadratic field extension with
  non-trivial Galois automorphism $\sigma$, the operation $\circ$ is
  the $\Gfi \vert \gfi$-involution `conjugate transpose' of the
  central algebra $A = \Mat_3(K)$ over $\Gfi$, given by
  $\mathbf{a}^\circ = (\mathbf{a}^\sigma)^\textup{t}$, and $H \in
  \Mat_3(\Gfi)$ is the structure matrix of a non-degenerate Hermitian
  sesquilinear form $f$ on $V = \Gfi^3$,
\item $\SU_1(D,f) = \{ g \in \GL_1(D) \mid \textup{Nrd}_{D \vert
    \Gfi}(g) = 1 \text{ and } g^\tau \eta g = \eta \}$ where $\Gfi =
  \gfi(\sqrt{\delta})$, $\delta \in \gri$ not a square, is a quadratic
  field extension, $A = D$ is a central division algebra of index $3$
  over $\Gfi$ with a $\Gfi \vert \gfi$-involution $\tau$, and $\eta
  \in D$ with $\eta^\tau = \eta$.
\end{enumerate}

Suppose that $\Gamma = \mathbf{G}(\gri_S)$ for either one of these
$\gfi$-defined algebraic groups, defined in terms of a central simple
algebra $A$ with $\Gfi \vert \gfi$-involution.  By \cite[Remark~6.6.3
and Theorem~10.2.5]{Sc85} there are essentially two types of
completions $A_v$ of $A$, as $v$ runs over places of $\gfi$ not in
$S$.  For places $v \not \in S$ of $\gfi$ which are non-decomposed in
$\Gfi \vert \gfi$, the completion of the $\gfi$-algebra $\Gfi$ at $v$
yields a quadratic field extension $\Gfi_v \vert \gfi_v$ and, apart
from finitely many exceptional places, $A_v$ is isomorphic to
$\Mat_3(\Gfi_v)$ equipped with the standard $\Gfi_v \vert
\gfi_v$-involution.  In this situation, excluding again finitely many
places, we have $\Gamma_v \cong \SU_3(\gri_v)$, where $\SU_3(\gri_v)
:= \SU_3(\Lri,\gri_v)$ with $\Lri$ denoting an unramified quadratic
extension of $\gri_v$.  For places $v \not \in S$ of $\gfi$ which are
decomposed in $\Gfi \vert \gfi$ the completion $\Gfi_v$ is isomorphic
to $\gfi_v \times \gfi_v$ and, apart from finitely many exceptions,
$A$ is isomorphic to $\Mat_3(\gfi_v) \times \Mat_3(\gfi_v)$ with the
involution $(X,Y) \to (Y^\textup{t},X^\textup{t})$.  In this
situation, excluding again finitely many places, we have $\Gamma_v
\cong \SL_3(\gri_v)$.

For $v \not \in S$, let $q_v$ denote the residue field cardinality of
$\gri_v$, and define the sets of rational primes $P_\textup{SL} := \{
p \mid \exists v \not \in S : \Gamma_v \cong \SL_3(\gri_v) \text{ and
} p = q_v \text{ is prime}\}$ and $P_\textup{SU} := \{ p \mid \exists
v \not \in S : \Gamma_v \cong \SU_3(\gri_v) \text{ and } p = q_v
\text{ is prime} \}$.  In connection with our approach toward outer
forms in Section~\ref{subsec:abscissa_outer}, e.g.\ in the proof of
Proposition~\ref{pro:first_product_SU}, the following lemma is
noteworthy.

\begin{lem}
  Each of the sets $P_\textup{SL}$ and $P_\textup{SU}$ has positive
  Dirichlet density within the set of all primes.
\end{lem}

\begin{proof}
  In both cases the assertion is a consequence of the Chebotarev
  Density Theorem.  In the case of $P_\textup{SL}$ the argument is
  straightforward; the set of primes $p$ which are totally split in
  $\Gfi$ has positive density.

  In order to show that $P_\textup{SU}$ has positive density, it
  suffices to supply a positive proportion of primes $p$ which admit a
  prolongation $v$ to $\gfi$ such that $\gfi_v$ has residue field of
  size $q_v = p$ and such that $v$ is inert in the quadratic extension
  $\Gfi | \gfi$.  For simplicity, we consider all number fields to be
  subfields of a fixed algebraic closure of $\Q$, e.g.\ within $\C$.

  We distinguish two cases.  First suppose that $\Gfi$ is not
  contained in the normal closure $N$ of $\gfi$.  Then the normal
  closure of $\Gfi$ has degree $2$ over $N$, and the Chebotarev
  Density Theorem shows that the proportion of primes $p$ which are
  totally split in $\gfi$, viz.\ $\lvert N : \Q \rvert^{-1}$, is
  positive and twice as much as the proportion of primes $p$ which are
  totally split in $\Gfi$.  Apart from finitely many exceptions, these
  primes are of the desired kind.

  We now turn to the more interesting case that $\Gfi$ is contained in
  the normal closure $N$ of $\gfi$.  In the Galois group $G =
  \textup{Gal}(N \vert \Q)$, let $H_\gfi$ denote the subgroup
  consisting of automorphisms which fix element-wise $\gfi$, and let
  $H_\Gfi$ be the subgroup consisting of automorphisms which fix
  element-wise $\Gfi$.  Then $H_\Gfi \trianglelefteq H_\gfi$ with
  $\lvert H_\gfi : H_\Gfi \rvert = 2$.  Choose $h \in H_\gfi$ such
  that $H_\gfi = \langle h \rangle H_\Gfi$, and let $C$ denote the
  conjugacy class of $h$ in $G$.  The Chebotarev Density Theorem
  supplies a positive proportion of primes $p$, unramified in $N \vert
  \Q$, such that the Frobenius conjugacy class of $p$ in $G$ is equal
  to $C$.  Fix one of these primes $p$, let $w = w_N$ be a
  prolongation of $p$ to $N$, and write $w_\Gfi := w \vert_\Gfi$ and
  $v := w \vert_\gfi$ for the restrictions of $w$ to $\Gfi$ and
  $\gfi$.  Then the decomposition group $G_w$ of $w$ maps
  isomorphically onto the cyclic Galois group of the unramified
  extension $N_w \vert \Q_p$.  Without loss of generality let us
  assume that $G_w = \langle h \rangle$.  As $h$ acts as the identity
  on $\gfi$, we have $\lvert \gfi_v : \Q_p \rvert = 1$, i.e.\ $q_v =
  p$, and as $h$ does not act as the identity on $\Gfi$, we must have
  $\lvert \Gfi_{w_\Gfi} : \gfi_v \rvert = 2$.  Thus $p$ is of the
  desired kind.
\end{proof}

\section{Adjoint action  of $\GL_3(\F_q)$ on $\spl_3(\F_q)$}
\label{sec:aux_sl3}
Let $\F_q$ be a finite field of characteristic not equal to $3$.  In
Section~\ref{subsec:abscissa_inner}, we require an overview of the
elements in $\spl_3(\F_q)$ up to conjugacy under the group
$\GL_3(\F_q)$.  We distinguish eight different types, labelled $0$,
$1$, $2$, $3$, $4$a, $4$b, $4$c,~$5$.  The total number of elements of
each type and the isomorphism types of their centralisers in
$\SL_3(\F_q)$ are summarised in Tables~\ref{table1} and \ref{table2}.
We briefly discuss the eight different types.

Type $0$ consists of the zero matrix, which does not feature in our
calculation but is shown for completeness.  Its centraliser is the
entire group $\SL_3(\F_q)$.

Type $1$ consists of nilpotent matrices with minimal polynomial equal
to $X^3$ over~$\F_q$.  The centraliser of a typical element is
\begin{equation}\label{equ:centr_type_1}
  \Cen_{\SL_3(\F_q)} \left( \Bigl(\begin{smallmatrix} 0 & 1 & 0 \\ 0 & 0 & 1
      \\ 0 & 0 & 0 \end{smallmatrix} \Bigr) \right) = \left\{
    \Bigl( \begin{smallmatrix} a & b & c \\ 0 & a & b \\ 0 & 0 &
      a \end{smallmatrix} \Bigr) \in \GL_3(\F_q) \mid  a^3 = 1 \right\}
\end{equation}
and matrices of type $1$ are regular.

Type $2$ consists of nilpotent matrices with minimal polynomial equal
to $X^2$ over~$\F_q$.  The centraliser of a typical element is
\begin{equation}\label{equ:centr_type_2}
  \Cen_{\SL_3(\F_q)} \left( \Bigl(\begin{smallmatrix} 0 & 1 & 0 \\ 0 & 0 & 0
      \\ 0 & 0 & 0 \end{smallmatrix} \Bigr) \right) = \left\{
    \Bigl( \begin{smallmatrix} a & b & c \\ 0 & a & 0 \\ 0 & d &
      e \end{smallmatrix} \Bigr) \in \GL_3(\F_q) \mid a^2 e = 1 \right\}
\end{equation}
and matrices of type $2$ are irregular.

Type $3$ consists of semisimple matrices with eigenvalues $\lambda \in
\F_q \setminus \{0\}$ of multiplicity $2$ and $\mu := -2\lambda$.  The
minimal polynomial of such elements over $\F_q$ is equal to
$(X-\lambda)(X-\mu)$.  The centraliser of a typical element is
\begin{equation}\label{equ:centr_type_3}
  \Cen_{\SL_3(\F_q)} \left( \Bigl(\begin{smallmatrix} \lambda & 0 & 0 \\ 0 &
      \lambda & 0 \\ 0 & 0 & \mu \end{smallmatrix} \Bigr) \right) =
  \left\{ \Bigl( \begin{smallmatrix} a & b & 0 \\ c & d & 0 \\ 0 & 0 &
      e \end{smallmatrix} \Bigr) \in \GL_3(\F_q) \mid (ad-bc) e = 1
  \right\}
\end{equation}
and matrices of type $3$ are irregular.

Type $4$a consists of semisimple matrices with distinct eigenvalues
$\lambda,\mu,\nu := -\lambda-\mu \in \F_q \setminus \{0\}$.  The
minimal polynomial of such elements over $\F_q$ is equal to
$(X-\lambda)(X-\mu)(X-\nu)$.  The centraliser of a typical element is
\begin{equation}\label{equ:centr_type_4a}
\Cen_{\SL_3(\F_q)} \left( \Bigl(\begin{smallmatrix} \lambda & 0 & 0 \\
    0 & \mu & 0 \\ 0 & 0 & \nu \end{smallmatrix} \Bigr) \right) =
\left\{ \Bigl( \begin{smallmatrix} a & 0 & 0 \\ 0 & b & 0 \\ 0 & 0 &
    c \end{smallmatrix} \Bigr) \in \GL_3(\F_q) \mid abc = 1 \right\}
\end{equation}
and matrices of type $4$a are regular.

Type $4$b consists of semisimple matrices with eigenvalues
$\lambda,\mu := \lambda^q \in \F_{q^2} \setminus \F_q$ and $\nu :=
-\lambda-\mu \in \F_q$.  The minimal polynomial of such elements over
$\F_q$ is equal to $(X-\lambda)(X-\mu)(X-\nu)$.  The centraliser of a
typical element is isomorphic to the multiplicative group of the field
$\F_{q^2}$ and matrices of type $4$b are regular.

Type $4$c consists of semisimple matrices with eigenvalues
$\lambda,\mu := \lambda^q, \nu := \lambda^{q^2} \in \F_{q^3} \setminus
\F_q$ with $\lambda + \mu + \nu = 0$.  The minimal polynomial of such
elements over $\F_q$ is equal to $(X-\lambda)(X-\mu)(X-\nu)$.  The
centraliser of a typical element is isomorphic to the group of
elements of norm $1$ in the field $\F_{q^3}$ and matrices of type $4$c
are regular.

Type $5$ consists of matrices with eigenvalues $\lambda \in \F_q
\setminus \{0\}$ of multiplicity $2$ and $\mu := -2\lambda$.  The
minimal polynomial of such elements over $\F_q$ is equal to
$(X-\lambda)^2(X-\mu)$.  The centraliser of a typical element is
\begin{equation}\label{equ:centr_type_5}
  \Cen_{\SL_3(\F_q)} \left( \Bigl(\begin{smallmatrix} \lambda & 1 & 0 \\ 0 &
      \lambda & 0 \\ 0 & 0 & \mu \end{smallmatrix} \Bigr) \right) =
  \left\{ \Bigl( \begin{smallmatrix} a & b & 0 \\ 0 & a & 0 \\ 0 & 0 &
      c \end{smallmatrix} \Bigr) \in \GL_3(\F_q) \mid a^2 c = 1 \right\}
\end{equation}
and matrices of type $5$ are regular.


\section{Adjoint action of $\GU_3(\F_{q^2},\F_q)$ on
  $\su_3(\F_{q^2},\F_q)$}
\label{sec:aux_su3}

In Section~\ref{subsec:abscissa_inner}, we require an overview of the
elements in $\su_3(\F_{q^2},\F_q)$ up to conjugacy under the group
$\GU_3(\F_{q^2},\F_q)$.  We assume that the characteristic of $\F_q$
is at least $5$, and for the purpose of explicit computations we write
$\F_{q^2} = \F_q(\sqrt{\rho})$, where $\rho \in \F_q$ is not a square.
We distinguish eight different types, labelled $0$, $1$, $2$, $3$,
$4$a, $4$b, $4$c,~$5$.  The total number of elements of each type and
the isomorphism types of their centralisers in $\SU_3(\F_{q^2},\F_q)$
are summarised in Tables~\ref{table3} and \ref{table4}.  After
recalling some preliminary tools to investigate the adjoint orbits, we
discuss the eight different types.

\subsection{}\label{sec:Cayley} Our computations are based on
classical results of Ennola and on the Cayley map, which allows us to
translate between conjugacy classes in $G := \GU_3(\F_{q^2},\F_q)$ and
adjoint orbits in $\mfg := \gu_3(\F_{q^2},\F_q)$.  We refer to
\cite[II.10 and VI.2]{We97} for basic properties of the Cayley map.
The group $G$ and the Lie algebra $\mfg$ are naturally embedded in the
associative $\F_q$-algebra $\Mat_3(\F_{q^2})$ with $\F_{q^2} \vert
\F_q$-involution $\circ$ given by $x^\circ = \overline{x}^\textup{t}$,
i.e.\ one obtains $x^\circ$ by applying the non-trivial Galois
automorphism of $\F_{q^2} \vert \F_q$ to each entry of $x$ and then
taking the transpose.  For any subset $M$ of $\Mat_3(\F_{q^2})$ let
$M_\textup{gen}$ denote the set of elements of $M$ which do not have
$-1$ as an eigenvalue.  The Cayley maps $\Cay: G_\textup{gen}
\rightarrow \mfg_\textup{gen}$ and $\cay: \mfg_\textup{gen}
\rightarrow G_\textup{gen}$, which are each defined by the mapping
rule \eqref{equ:cayley}, are easily seen to be mutual inverses of each
other, and they are compatible with conjugation by elements of $G$.

Put $u(q) := \lvert \GU_3(\F_{q^2},\F_q) \rvert = q^3
(q+1)(q^2-1)(q^3+1)$.  According to \cite{En62}, two elements $X,Y \in
\GU_3(\F_{q^2},\F_q)$ are conjugate in $\GU_3(\F_{q^2},\F_q)$ if and
only if they are similar, i.e.\ conjugate in $\GL_3(\F_{q^2})$.
Moreover, the size of the conjugacy class in $\GU_3(\F_{q^2},\F_q)$
consisting of matrices which are similar to a given $X \in
\GL_3(\F_{q^2})$ is $\gamma_G(X) u(q) c(X)^{-1}$, where
\begin{align*}
  \gamma_G(X) & := \lvert \{ \Gamma \in \Mat_3(\F_{q^2}) \mid
  \Gamma^\circ
  = \Gamma \text{ non-singular, } X^\circ \Gamma X = \Gamma \} \rvert, \\
  c(X) & := \lvert \Cen_{\GL_3(\F_{q^2})}(X) \lvert.
\end{align*}

Suppose that $x \in \mfg$ does not have $-1$ as an eigenvalue and put
$X := \cay(x)$.  A short computation shows that for all $\Gamma \in
\Mat_3(\F_{q^2})$ with $\Gamma^\circ = \Gamma$ one has
$$
X^\circ \Gamma X = \Gamma \quad \text{if and only if} \quad x^\circ
 \Gamma + \Gamma x = 0
$$
If follows that the adjoint orbit of $x$, viz.\
$$
\{ x^g \mid g \in G \} = \Cay (\{ X^g \mid g \in G \}),
$$
has size $\gamma_\mfg(x) u(q) c(x)^{-1}$, where
\begin{align*}
  \gamma_\mfg(x) & := \lvert \{ \Gamma \in \Mat_3(\F_{q^2}) \mid
  \Gamma^\circ =
  \Gamma \text{ non-singular, } x^\circ \Gamma + \Gamma x = 0 \} \rvert, \\
  c(x) & := \lvert \Cen_{\GL_3(\F_{q^2})}(x) \lvert.
\end{align*}

Next suppose that $x \in \mfg$ has $-1$ as an eigenvalue.  Since $q
\geq 5$ we find $\lambda \in \F_{q^2}$ with $\Tr_{\F_{q^2} \vert
  \F_q}(\lambda) = 0$ such that $x_0 := x + \lambda \Id_3 \in \mfg$
does not have eigenvalue $-1$.  Clearly, the adjoint orbits of $x$ and
$x_0$ have the same size, $c(x) = c(x_0)$ and $\gamma_\mfg(x) =
\gamma_\mfg(x_0)$.  Hence the size of the adjoint orbit of $x$ is
still given by the term $\gamma_\mfg(x) u(q) c(x)^{-1}$.  More
generally, for any $x \in \gl_3(\F_{q^2})$, the same formula gives the
size of the adjoint orbit in $\gu_3(\F_{q^2},\F_q)$ consisting of
matrices which are similar to $x$.

\subsection{} We are now ready to discuss the eight different types.
Type $0$ consists of the zero matrix, which does not feature in our
calculation but is shown for completeness.  Its centraliser is the
entire group $\SU_3(\F_{q^2},\F_q)$.

Type $1$ consists of nilpotent matrices with minimal polynomial equal
to $X^3$ over~$\F_{q^2}$.  Let $x := \left( \begin{smallmatrix} 0 & 1
    & \\ & 0 & 1 \\ & & 0 \end{smallmatrix} \right)$.  Then the
non-singular matrices $\Gamma \in \Mat_3(\F_{q^2})$ with $\Gamma^\circ
= \Gamma$ and $x^\circ \Gamma + \Gamma x = 0$ are the matrices over
$\F_{q^2}$ of the form
$$
\Gamma =
\begin{pmatrix}
  0 & 0 & \gamma \\
  0 & -\gamma & \beta \\
  \overline{\gamma} & \overline{\beta} & \alpha
\end{pmatrix}, \quad \text{where $\overline{\alpha} = \alpha$,
  $\overline{\beta} = -\beta$ and $\overline{\gamma} = \gamma \not =
  0$.}
$$
This shows that $\gamma_\mfg(x) = (q-1) q^2$.  From
\eqref{equ:centr_type_1} we gather that $c(x) = (q^2-1) q^4 $ so that
$\gamma_\mfg(x) u(q) c(x)^{-1} = (q^3+1)(q^2-1)q$.

Put $\Gamma_0 := \left( \begin{smallmatrix} 0 & 0 & -1 \\ 0 & 1 & 0 \\
  -1 & 0 & 0 \end{smallmatrix} \right)$.  The centraliser of a typical
element of type $1$ in $\SU_3(\F_{q^2},\F_q)$ is isomorphic to
\begin{multline}\label{equ:su_centr_type_1}
  \Cen_{\SL_3(\F_{q^2})} (x) \cap \{ Y \in \GL_3(\F_{q^2}) \mid
  Y^\circ  \Gamma_0 Y = \Gamma_0 \} \\
  = \left\{ \Bigl( \begin{smallmatrix} a & b & c \\ 0 & a & b \\ 0 & 0
      & a \end{smallmatrix} \Bigr) \in \GL_3(\F_{q^2}) \mid a^3 = a
    \overline{a} = 1, \overline{a} b - a \overline{b} = b \overline{b}
    - \overline{a} c - a \overline{c} = 0 \right\}.
\end{multline}
The elements of this group can be conveniently parameterised in terms
of $(a,a\overline{b},a\overline{c})$ and a short computation yields
that the group is isomorphic to $(\mu_3(\F_{q^2}) \cap
\ker(N_{\F_{q^2} \vert \F_q})) \times \F_q^+ \times \F_q^+$.  Matrices
of type $1$ are regular.

Type $2$ consists of nilpotent matrices with minimal polynomial equal
to $X^2$ over~$\F_{q^2}$.  Let $x := \left( \begin{smallmatrix} 0 & 1
    & 0\\ 0 & 0 & 0\\ 0 & 0 & 0 \end{smallmatrix} \right)$.  Then the
non-singular matrices $\Gamma \in \Mat_3(\F_{q^2})$ with $\Gamma^\circ
= \Gamma$ and $x^\circ \Gamma + \Gamma x = 0$ are the matrices over
$\F_{q^2}$ of the form
$$
\Gamma =
\begin{pmatrix}
  0 & \beta & 0 \\
  \overline{\beta} & \alpha & \gamma \\
  0 & \overline{\gamma} & \delta
\end{pmatrix}, \quad \text{where $\overline{\alpha} = \alpha$,
  $\overline{\beta} = -\beta \not = 0$ and $\overline{\delta} = \delta
  \not = 0$.}
$$
This shows that $\gamma_\mfg(x) = (q-1)^2 q^3$.  From
\eqref{equ:centr_type_2} we gather that $c(x) = (q^2-1)^2 q^6$ so that
$\gamma_\mfg(x) u(q) c(x)^{-1} = (q^3+1)(q-1)$.

Put $\Gamma_0 := \left( \begin{smallmatrix} 0 & \sqrt{\rho} & 0 \\
  -\sqrt{\rho} & 0 & 0 \\ 0 & 0 & 1 \end{smallmatrix} \right)$.  Then the
centraliser of a typical element of type $2$ in $\SU_3(\F_{q^2},\F_q)$
is isomorphic to
\begin{multline}\label{equ:su_centr_type_2}
  \Cen_{\SL_3(\F_{q^2})} (x) \cap \{ Y \in \GL_3(\F_{q^2}) \mid
  Y^\circ
  \Gamma_0 Y = \Gamma_0 \} \\
  = \left\{ \Bigl( \begin{smallmatrix} a & b & c \\ 0 & a & 0 \\ 0 & d
      & e \end{smallmatrix} \Bigr) \in \GL_3(\F_{q^2}) \mid a^2 e = a
    \overline{a} = 1, a\overline{b}\sqrt{\rho} -
    \overline{a}b\sqrt{\rho} + d \overline{d} =
    a\overline{c}\sqrt{\rho} + d\overline{e} = 0 \right\}.
\end{multline}
The elements of this group can be parameterised in terms of
$(a,a\overline{b}\sqrt{\rho},d\overline{e},a\overline{c}\sqrt{\rho})$
and a short computation yields that the group is isomorphic to
$\ker(N_{\F_{q^2} \vert \F_q}) \ltimes \mathcal{H}(\F_q)$.  Indeed,
conjugation by the monomial matrix $\left( \begin{smallmatrix}
\sqrt{\rho} & 0 & 0 \\ 0 & 0 & 1 \\ 0 & 1 & 0 \end{smallmatrix}
\right)^{-1}$ maps the group described in \eqref{equ:su_centr_type_2}
isomorphically onto
$$
\left\{ \Bigl( \begin{smallmatrix} a & 0 & 0 \\ 0 & a^{-2} & 0 \\
    0 & 0 & a \end{smallmatrix} \Bigr) \in \GL_3(\F_{q^2}) \mid a
  \overline{a} = 1 \right\} \ltimes
\left\{ \Bigl( \begin{smallmatrix} 1 & \overline{d} & b_0 \\ 0 & 1 & d \\
    0 & 0 & 1 \end{smallmatrix} \Bigr) \in \GL_3(\F_{q^2}) \mid b_0 +
  \overline{b_0} = d \overline{d} \right\}.
$$
The normal factor in this semi-direct product is seen to be isomorphic
to the Heisenberg group $\mathcal{H}(\F_q)$ by `formally setting
$\sqrt{\rho} = 1$', i.e.\ via the map
$$
\Bigl( \begin{smallmatrix} 1 & \overline{d} & b_0 \\ 0 & 1 & d \\
  0 & 0 & 1 \end{smallmatrix} \Bigr) \mapsto
\Bigl( \begin{smallmatrix} 1 & \phi(\overline{d}) & \phi(b_0) \\ 0 & 1
  & \phi(d) \\ 0 & 0 & 1 \end{smallmatrix} \Bigr), \qquad \text{where
  $\phi(z) := (z + \overline{z})/2 + (z -
  \overline{z})/(2\sqrt{\rho})$.}
$$
Matrices of type $2$ are irregular.

Type $3$ consists of semisimple matrices with eigenvalues $\lambda \in
\F_{q^2} \setminus \{0\}$ of multiplicity $2$ and $\mu := -2\lambda$
of multiplicity $1$.  The condition $\overline{\lambda} = -\lambda$
implies that $\lambda \in \F_q \sqrt{\rho} \setminus \{0\}$ so that
there are $q-1$ choices for $\lambda$ and $q-1$ corresponding orbits.
The minimal polynomial of such elements over $\F_{q^2}$ is equal to
$(X-\lambda)(X-\mu)$.

Let $x := \left( \begin{smallmatrix} \lambda & & \\ & \lambda & \\ & &
    \mu \end{smallmatrix} \right)$.  Then the non-singular matrices
$\Gamma \in \Mat_3(\F_{q^2})$ with $\Gamma^\circ = \Gamma$ and
$x^\circ \Gamma + \Gamma x = 0$ are the matrices over $\F_{q^2}$ of
the form
$$
\Gamma =
\begin{pmatrix}
  \alpha & \beta & 0 \\
  \overline{\beta} & \gamma & 0 \\
  0 & 0 & \delta
\end{pmatrix}, \quad \text{where $\overline{\alpha} = \alpha$,
  $\overline{\gamma} = \gamma$, $\alpha \gamma - \beta
  \overline{\beta} \not = 0$ and $\overline{\delta} = \delta \not =
  0$.}
$$
This shows that $\gamma_\mfg(x) = (q^2+1) (q-1) q$.  From
\eqref{equ:centr_type_3} we gather that $c(x) = (q^4-1) (q^2-1)^2 q^2$ so that
the size of each orbit is $\gamma_\mfg(x) u(q) c(x)^{-1} =
(q^2-q+1)q^2$.

Taking $\Gamma_0 := \Id_3$, the centraliser of the typical element $x$
of type $3$ in $\SU_3(\F_{q^2},\F_q)$ is
\begin{multline}\label{equ:su_centr_type_3}
  \Cen_{\SU_3(\F_{q^2},\F_q)} (x) \\
  = \left\{ \Bigl( \begin{smallmatrix} a & b & 0 \\ c & d & 0 \\ 0 & 0
      & e \end{smallmatrix} \Bigr) \in \GL_3(\F_{q^2}) \mid a
    \overline{a} + c \overline{c} = b \overline{b} + d \overline{d} =
    (ad - bc)e = 1, a \overline{b} + c \overline{d} = 0 \right\}.
\end{multline}
Inspection shows that this group is isomorphic to
$\GU_2(\F_{q^2},\F_q)$, and matrices of type $3$ are irregular.

Types $4$a, $4$b and $4$c classify semisimple matrices with distinct
non-zero eigenvalues $\lambda,\mu,\nu := -\lambda-\mu$ in an extension
of $\F_{q^2}$.  The minimal polynomial of such elements over
$\F_{q^2}$ is equal to $(X-\lambda)(X-\mu)(X-\nu)$.  The condition $\{
\overline{\lambda}, \overline{\mu}, -\overline{\lambda} -
\overline{\mu} \} = \{-\lambda, -\mu, \lambda+\mu \}$ allows for three
possibilities.

Type $4$a consists of those semisimple matrices where
$\overline{\lambda} = -\lambda$ and $\overline{\mu} = -\overline{\mu}$
so that, in particular, $\lambda, \mu \in \F_{q^2} \setminus \{0\}$.
The condition that $\lambda, \mu, -\lambda-\mu$ are distinct means
that $\mu \not \in \{\lambda, -2 \lambda, -\lambda/2 \}$.  A short
computation shows that there are $(q-1)(q-2)/6$ possibilities for $\{
\lambda, \mu, \nu \}$ and hence a corresponding number of orbits.

Let $x := \left( \begin{smallmatrix} \lambda &
    & \\ & \mu & \\
    & & \nu \end{smallmatrix} \right)$.  Then the non-singular
matrices $\Gamma \in \Mat_3(\F_{q^2})$ with $\Gamma^\circ = \Gamma$
and $x^\circ \Gamma + \Gamma x = 0$ are the matrices over $\F_{q^2}$
of the form
$$
\Gamma =
\begin{pmatrix}
  \alpha & 0 & 0 \\
  0 & \beta & 0 \\
  0 & 0 & \gamma
\end{pmatrix}, \quad \text{where $\overline{\alpha} = \alpha \not =
  0$, $\overline{\beta} = \beta \not = 0$ and $\overline{\gamma} =
  \gamma \not = 0$.}
$$
This shows that $\gamma_\mfg(x) = (q-1)^3$.  From
\eqref{equ:centr_type_4a} we gather that $c(x) = (q^2-1)^3$ so that
the size of each orbit is $\gamma_\mfg(x) u(q) c(x)^{-1} =
(q^2-q+1)(q-1)q^3$.

Taking $\Gamma_0 := \Id_3$, the centraliser of the typical element $x$
of type $4a$ in $\SU_3(\F_{q^2},\F_q)$ is
\begin{equation}\label{equ:su_centr_type_4a}
  \Cen_{\SU_3(\F_{q^2},\F_q)} (x) \\
  = \left\{ \Bigl( \begin{smallmatrix} a & 0 & 0 \\ 0 & b & 0 \\ 0 & 0
      & c \end{smallmatrix} \Bigr) \in \GL_3(\F_{q^2}) \mid a
    \overline{a} = b \overline{b} = abc = 1 \right\}.
\end{equation}
Inspection shows that this group is isomorphic to $\ker(N_{\F_{q^2}
  \vert \F_q}) \times \ker(N_{\F_{q^2} \vert \F_q})$, and matrices of
type $4$a are regular.

Type $4$b consists of those semisimple matrices where
$\overline{\lambda} = -\mu$ and $\overline{\mu} = -\lambda$ so that,
in particular, $\lambda, \mu \in \F_{q^2} \setminus \{0\}$.  The
condition that $\lambda, \mu = -\overline{\lambda}, \nu =
\overline{\lambda}-\lambda$ are distinct means that $\lambda \not \in
\{ -\overline{\lambda}, 2 \overline{\lambda}, \overline{\lambda}/2 \}$.
Thus there are $q(q-1)/2$ possibilities for $\{ \lambda, \mu, \nu
\}$ and a corresponding number of orbits.

Let $x := \left( \begin{smallmatrix} \lambda &
    & \\ & \mu & \\
    & & \nu \end{smallmatrix} \right)$.  Then the non-singular
matrices $\Gamma \in \Mat_3(\F_{q^2})$ with $\Gamma^\circ = \Gamma$
and $x^\circ \Gamma + \Gamma x = 0$ are the matrices over $\F_{q^2}$
of the form
$$
\Gamma =
\begin{pmatrix}
  0 & \alpha & 0 \\
  \overline{\alpha} & 0 & 0 \\
  0 & 0 & \beta
\end{pmatrix}, \quad \text{where $\alpha \not = 0$ and
$\overline{\beta} = \beta \not = 0$.}
$$
This shows that $\gamma_\mfg(x) = (q+1)(q-1)^2$.  From
\eqref{equ:centr_type_4a} we gather that $c(x) = (q^2-1)^3$ so that
the size of each orbit is $\gamma_\mfg(x) u(q) c(x)^{-1} =
(q^3+1)q^3$.

Put $\Gamma_0 := \left( \begin{smallmatrix} 0 & 1 & 0 \\
    1 & 0 & 0 \\ 0 & 0 & 1 \end{smallmatrix} \right)$.
Then the centraliser of a typical element of type $4$b in
$\SU_3(\F_{q^2},\F_q)$ is isomorphic to
\begin{multline}\label{equ:su_centr_type_4}
  \Cen_{\SL_3(\F_{q^2})} (x) \cap \{ Y \in \GL_3(\F_{q^2}) \mid
  Y^\circ
  \Gamma_0 Y = \Gamma_0 \} \\
  = \left\{ \Bigl( \begin{smallmatrix} a & 0 & 0 \\ 0 & b & 0 \\ 0 & 0
      & c \end{smallmatrix} \Bigr) \in \GL_3(\F_{q^2}) \mid a b c = a
    \overline{b} = 1 \right\}.
\end{multline}
The elements of this group can be conveniently parameterised in terms
of $a$ and the group is isomorphic to $\F_{q^2}^*$.  Matrices of
type $4$b are regular.

Type $4$c consists of those semisimple matrices where
$\overline{\lambda} = -\mu$, $\overline{\mu} = -\nu$ and, consequently
$\overline{\nu} = -\lambda$.  Then $\nu = -\lambda -\mu$ implies that
$\overline{\overline{\lambda}} - \overline{\lambda} + \lambda = 0$.  A
short computation yields that the relevant values for $\lambda$ are
$\lambda = \lambda_0 \sqrt{\rho}$, where $\lambda_0 \in F_{q^3}
\setminus \{0\}$ with $\Tr_{\F_{q^3} \vert \F_q}(\lambda_0) = 0$.  In
particular, it follows that $\F_q(\lambda) = \F_{q^6}$.  There are
$(q^2-1)/3$ possibilities for $\{ \lambda, \mu, \nu \}$ and a
corresponding number of orbits.

Let $x$ be an element of type $4$c.  Then the non-singular matrices
$\Gamma \in \Mat_3(\F_{q^2})$ with $\Gamma^\circ = \Gamma$ and $x^\circ
\Gamma + \Gamma x = 0$ are in one-to-one correspondence to non-zero
elements of $\gu_1(\F_{q^6},\F_{q^3})$.  This shows that
$\gamma_\mfg(x) = q^3-1$.  The centraliser of $x$ in $\GL_3(\F_{q^2})$
is isomorphic to $\F_{q^6}^*$ so that $c(x) = q^6-1$.  Therefore the
size of each orbit is $\gamma_\mfg(x) u(q) c(x)^{-1} =
(q^2-1)(q+1)q^3$.

The centraliser of a typical element of type $4$c in
$\GU_3(\F_{q^2},\F_q)$ is isomorphic to $\GU_1(\F_{q^6},\F_{q^3})
\cong \ker(N_{\F_{q^6} \vert \F_{q^3}})$.  The centraliser of such
an element in $\SU_3(\F_{q^2},\F_q)$ is isomorphic to
$\ker(N_{\F_{q^6}
  \vert \F_{q^3}}) \cap \ker(N_{\F_{q^6} \vert \F_{q^2}})$.  Matrices
of type $4$c are regular.

Type $5$ consists of semisimple matrices with eigenvalues $\lambda \in
\F_{q^2} \setminus \{0\}$ of multiplicity $2$ and $\mu := -2\lambda$
of multiplicity $1$.  The condition $\overline{\lambda} = -\lambda$
implies that $\lambda \in \F_q \sqrt{\rho} \setminus \{0\}$ so that
there are $q-1$ choices for $\lambda$ and $q-1$ corresponding orbits.
The minimal polynomial of such elements over $\F_{q^2}$ is equal to
$(X-\lambda)^2(X-\mu)$.

Let $x := \left( \begin{smallmatrix} \lambda & 1
    & \\ & \lambda & \\
    & & \mu \end{smallmatrix} \right)$.  Then the non-singular
matrices $\Gamma \in \Mat_3(\F_{q^2})$ with $\Gamma^\circ = \Gamma$
and $x^\circ \Gamma + \Gamma x = 0$ are the matrices over $\F_{q^2}$
of the form
$$
\Gamma =
\begin{pmatrix}
  0 & \beta & 0 \\
  \overline{\beta} & \alpha & 0 \\
  0 & 0 & \gamma
\end{pmatrix}, \quad \text{where $\overline{\alpha} = \alpha$,
  $\overline{\beta} = -\beta \not = 0$ and $\overline{\gamma} = \gamma
  \not = 0$.}
$$
This shows that $\gamma_\mfg(x) = (q-1)^2 q$.  From
\eqref{equ:centr_type_5} we gather that $c(x) = (q^2-1)^2 q^2$ so that
the size of each orbit is $\gamma_\mfg(x) u(q) c(x)^{-1} =
(q^3+1)(q-1)q^2$.

Put $\Gamma_0 := \left( \begin{smallmatrix} 0 & \sqrt{\rho} & 0 \\
    -\sqrt{\rho} & 0 & 0 \\ 0 & 0 & 1 \end{smallmatrix} \right)$.
Then the centraliser of a typical element of type $5$ in
$\SU_3(\F_{q^2},\F_q)$ is isomorphic to
\begin{multline}\label{equ:su_centr_type_5}
  \Cen_{\SL_3(\F_{q^2})} (x) \cap \{ Y \in \GL_3(\F_{q^2}) \mid
  Y^\circ \Gamma_0 Y = \Gamma_0 \} \\
  = \left\{ \Bigl( \begin{smallmatrix} a & b & 0 \\ 0 & a & 0 \\ 0 & 0
      & c \end{smallmatrix} \Bigr) \in \GL_3(\F_{q^2}) \mid a^2 c = a
    \overline{a} = 1, a\overline{b} - \overline{a}b = 0 \right\}.
\end{multline}
The elements of this group can be conveniently parameterised in terms
of $(a,a\overline{b})$ and a short computation yields that the group
is isomorphic to $\ker(N_{\F_{q^2} \vert \F_q}) \ltimes \F_q^+$.
Matrices of type $5$ are regular.

\bigskip

\begin{acknowledgements}
The authors would like to thank Alexander Lubotzky as well as the
 following institutions: the Batsheva de Rothschild Fund for the
 Advancement of Science, the EPSRC, the Mathematisches
 Forschungsinstitut Oberwolfach, the National Science Foundation and
 the Nuffield Foundation.
\end{acknowledgements}

%



\end{document}